\documentclass{amsart}

\usepackage{multicol}
\usepackage{lmodern}

\usepackage[T1]{fontenc}
\usepackage{amsmath}
\usepackage{amsfonts}
\usepackage{amssymb}
\usepackage{graphicx}
\usepackage{mathrsfs}
\usepackage{amsthm}
\usepackage{enumerate}
\usepackage{leftidx}
\usepackage{hyperref}
\usepackage{extarrows}
\usepackage[all]{xy}
\usepackage{stmaryrd}
\usepackage{qtree}
\usepackage{wasysym}
\usepackage[OT2,T1]{fontenc}
\usepackage{epstopdf}
\usepackage{hyperref}
\numberwithin{equation}{subsection}
\newtheorem{thm}[subsubsection]{Theorem}
\newtheorem{conj}[subsubsection]{Conjecture}

\newtheorem*{thm*}{Theorem}
\newtheorem*{thmA}{Theorem A}
\newtheorem*{thmAprime}{Theorem A'}
\newtheorem*{thmB}{Theorem B}
\newtheorem*{thmC}{Theorem C}

\newtheorem{cor}[subsubsection]{Corollary}
\newtheorem{lem}[subsubsection]{Lemma}
\newtheorem{prop}[subsubsection]{Proposition}
\theoremstyle{definition}
\newtheorem{defn}[subsubsection]{Definition}
\theoremstyle{remark}
\newtheorem{rem}[subsubsection]{Remark}

\DeclareMathOperator{\edyn}{EDynk}
\DeclareMathOperator{\dyn}{Dynk}

\DeclareMathOperator{\ad}{ad}

\DeclareMathOperator{\Aut}{Aut}

\DeclareMathOperator{\BC}{BC}
\DeclareMathOperator{\PP}{\mathcal P}

\DeclareMathOperator{\Gal}{Gal}

\DeclareMathOperator{\Hom}{Hom}

\DeclareMathOperator{\N}{N}

\DeclareMathOperator{\Gr}{Gr}

\DeclareMathOperator{\Res}{Res}

\DeclareMathOperator{\id}{id}

\DeclareMathOperator{\Stab}{Stab}

\DeclareMathOperator{\topp}{top}

\DeclareMathOperator{\vol}{vol}
\DeclareMathOperator{\un}{un}

\newcommand{\lprod}[1]{\langle #1  \rangle}
\newcommand{\C}{\ensuremath{\mathbb{C}}}

\newcommand{\Z}{\ensuremath{\mathbb{Z}}}

\newcommand{\Q}{\ensuremath{\mathbb{Q}}}

\newcommand{\Ok}{\ensuremath{\mathcal{O}}}

\newcommand{\abs}[1]{\left\vert#1\right\vert}
\newcommand{\set}[1]{\left\{#1\right\}}

\newcommand{\hatS}{\widehat{\mathcal S}}
\newcommand{\To}{\longrightarrow}
\newcommand{\isom}{\overset{\sim}{\To}}
\newcommand{\ignore}[1]{{}}
\newcommand{\lang}{\leftidx^{L}}

\RequirePackage{textcmds}\relax
\usepackage{enumitem}

\setlist[enumerate]{leftmargin=*}
\setlist[itemize]{leftmargin=*}

\title[Twisted orbital integrals and irreducible components]{Twisted orbital integrals and irreducible components of affine Deligne--Lusztig varieties}

\author[Rong Zhou]{Rong Zhou}\email{r.zhou@imperial.ac.uk}
\address{Imperial College London Department of Mathematics\\Huxley Building \\ 180 Queen's Gate \\
	London SW7 2AZ, UK}

\author[Yihang Zhu]{Yihang Zhu}\email{yihang@math.columbia.edu}
\address{Columbia University  Department of Mathematics \\ 2990 Broadway \\ New York, NY 10027, USA}

\begin{document}
	\begin{abstract}
		We analyze the asymptotic behavior of certain twisted orbital integrals arising from the study of affine Deligne--Lusztig varieties. The main tools include the Base Change Fundamental Lemma and $q$-analogues of the Kostant partition functions. As an application we prove a conjecture of Miaofen Chen and Xinwen Zhu, relating the set of irreducible components of an affine Deligne--Lusztig variety modulo the action of the $\sigma$-centralizer group to the Mirkovi\'c--Vilonen basis of a certain weight space of a representation of the Langlands dual group.
	\end{abstract}
\maketitle

 \setcounter{tocdepth}{1}
 \tableofcontents{}

\section{Introduction}
\subsection{The main result}
First introduced by Rapoport \cite{rapoportguide},  affine Deligne--Lusztig varieties play an important role in arithmetic geometry and the Langlands program.
One of the main motivations to study affine Deligne--Lusztig varieties comes from the theory of $p$-adic uniformization, which was studied by various authors including \v{C}erednik \cite{Cerednik}, Drinfeld \cite{Drinfeld}, Rapoport--Zink \cite{RZ96}, and more recently Howard--Pappas \cite{Howard2015} and Kim \cite{kim}. In this theory, a $p$-adic formal scheme known as the Rapoport--Zink space uniformizes a tubular neighborhood in an integral model of a Shimura variety around a Newton stratum. The reduced subscheme of the Rapoport--Zink space is a special example of affine Deligne--Lusztig varieties. Thus in many cases, understanding certain cycles on Shimura varieties reduces to understanding the geometry of the affine Deligne--Lusztig variety.

\ignore{In a parallel story over function fields, affine Deligne--Lusztig varieties also arise naturally in the study of local shtukas, see for instance Hartl--Viehmann \cite{HartlViehmann}.}

\ignore{Understanding of basic geometric properties of affine Deligne--Lusztig varieties has proved fruitful for arithmetic applications. For instance an understanding of the connected components \cite{CKV} was applied to the proof of a version of the Langlands--Rapoport Conjecture by Kisin \cite{kisin2012modp}. The geometry of the supersingular locus of Hilbert modular varieties, which is also a question closely related to affine Deligne--Lusztig varieties via $p$-adic uniformization, was applied to arithmetic level raising in the recent work of Liu--Tian \cite{LiuTian}.  }

In this paper, we concern the problem of parameterizing the irreducible components of affine Deligne--Lusztig varieties. We introduce some notations. Let $F$ be a non-archimedean local field with valuation ring $\Ok_F$ and residue field $k_F= \mathbb{F}_q$. Fix a uniformizer $\pi_F \in F$. Let $L$ be the completion of the maximal unramified extension of $F$, and let $\sigma$ be the Frobenius automorphism of $L$ over $F$. Let $G$ be a connected reductive group scheme over $\Ok_F$. We fix $T\subset G$ to be the centralizer of a maximal $\Ok_F$-split torus, and fix a Borel subgroup $B\subset G$ containing $T$. For $\mu\in X_*(T)^+$ and $b\in G(L)$, the \emph{affine Deligne--Lusztig variety} associated to $(G,\mu,b)$ is defined to be $$X_\mu(b)=\{g\in G(L)/G(\Ok_L)\mid g^{-1}b\sigma(g)\in G(\Ok_L)\mu(\pi_F)G(\Ok_L)\}.$$ More precisely, the above set is the set of $\overline {\mathbb F}_q$-points of a scheme or a perfect scheme, depending on whether $F$ has equal or mixed characteristic. See \cite{BS} and \cite{Zhu} for the result in mixed characteristic.

Let $\Sigma^{\topp}$ be the set of top-dimensional irreducible components of $X_\mu(b)$. The group $$J := J_b(F)=\{g\in G(L)\mid g^{-1}b\sigma(g)=b\}$$ naturally acts on $X_\mu(b)$. Our goal is to understand the set
$J\backslash\Sigma^{\topp}$ of $J$-orbits in $\Sigma ^{\topp}$.

The motivation for studying this set is to understand cycles in the special fiber of Shimura varieties; in particular cycles arising from the basic, or supersingular locus. In \cite{xiaozhu}, the authors used the description of $J\backslash\Sigma^{\topp}$ in some special cases to prove certain cases of the Tate conjecture for Shimura varieties. In a different situation, a description of the components in the supersingular locus was used to study certain arithmetic level-raising phenomena in \cite{LiuTian}. After the work of Xiao--Zhu \cite{xiaozhu}, Miaofen Chen and Xinwen Zhu formulated a general conjecture relating $J\backslash\Sigma^{\topp}$ to the Mirkovi\'{c}--Vilonen cycles in the affine Grassmannian.
To state the conjecture we introduce some more notations.

Let $\widehat{G}$ denote the Langlands dual group of $G$ over $\C$, equipped with a Borel pair $\widehat{T}\subset \widehat B$, where $\widehat T$ is a maximal torus dual to $T$ and equipped with an algebraic action by $\sigma$. Let $\hatS$ be the identity component of the $\sigma$-fixed points of $\widehat T$. In $ X^*(\hatS)$, there is a distinguished element $\lambda_b$, determined by $b$. It is the ``best integral approximation'' of the Newton cocharacter of $b$, but we omit its precise definition here (see Definition \ref{defn:lambda_b}). The fixed $\mu$ determines a highest weight representation $V_{\mu}$ of $\widehat{G}$. We write $V_\mu(\lambda_b)_{\mathrm{rel}}$ for the $\lambda_b$-weight space in $V_{\mu}$ with respect to the action of $\hatS$.

\begin{conj}[Miaofen Chen, Xinwen Zhu]\label{conj:chenzhu1}There exists a natural bijection between $J \backslash\Sigma^{\topp} $ and the Mirkovi\'{c}--Vilonen basis of $V_{\mu} (\lambda_b)_{\mathrm{rel}}$. In particular,\begin{align}\label{eq:intro numerical}
	\abs{J\backslash\Sigma^{\topp} }  = \dim V_{\mu} (\lambda_b)_{\mathrm {rel}}
	.\end{align}
\end{conj}
Our main result is the following
\begin{thmA}[Corollary \ref{cor:CZ}]Conjecture \ref{conj:chenzhu1} holds.
\end{thmA}
When the group $G$ is quasi-split but not necessarily unramified, we are able to prove an analogous result, see Appendix \ref{app} for the details.
\ignore{In Appendix \ref{app}, we prove the following
	\begin{thmAprime}[Theorem \ref{thm: main quasi-split}] Conjecture \ref{conj:chenzhu1} generalizes to quasi-split groups that are not necessarily unramified, in a way that is compatible with the ramified geometric Satake in \cite{Zhuram}.
	\end{thmAprime} In the rest of the introduction we shall only discuss Theorem A.}
\subsection{Previous results}
Previously, partial results towards Conjecture \ref{conj:chenzhu1} have been obtained by Xiao--Zhu \cite{xiaozhu}, Hamacher--Viehmann \cite{HV}, and Nie \cite{nie}, based on a common idea of reduction to the superbasic case (which goes back to \cite{GHKR}).

More precisely, Xiao--Zhu \cite{xiaozhu} proved the conjecture for general $G$, general $\mu$, and \emph{unramified} $b$, meaning that $J_b$ and $G$ are assumed to have equal $F$-rank.

Hamacher--Viehmann \cite{HV} proved the conjecture under either of the following two assumptions:
\begin{itemize}
	\item  The cocharacter $\mu$ is minuscule, and $G$ is split over $F$.
	\item The cocharacter $\mu$ is minuscule, and $b$ is superbasic in $M$, where $M$ is the largest Levi of $G$ inside which $b$ is basic. (In particular if $b$ is basic then they assume that $b$ is superbasic).
\end{itemize}

More recently, Nie \cite{nienew} proved the conjecture for arbitrary $G$ under the assumption that $\mu$ is a sum of dominant minuscule coweights. In particular it holds when the Dynkin diagram of $G_{\overline{F}}$ only involves factors of type $A$.
Moreover, Nie constructed a surjection from the Mirkovi\'{c}--Vilonen basis to the set $J\backslash\Sigma^{\topp}$ in all cases. Thus in order to prove the conjecture, it suffices to prove the numerical relation (\ref{eq:intro numerical}) for groups besides type $A$.

After we finished this work, Nie uploaded online new versions of the preprint \cite{nienew}, in which he proved Conjecture \ref{conj:chenzhu1} in full generality, using methods independent of ours. Our work only uses the weaker results of his as stated in the above paragraph.

\subsection{Some features of the method}\label{subsec:novelty}
Our proof of Conjecture \ref{conj:chenzhu1} is based on an approach completely different from the previous works. The key idea is to use the  Lang--Weil estimate to relate the cardinality of $J\backslash \Sigma ^{\topp}$ to the asymptotic behavior of the number of points on $X_{\mu} (b)$ over a finite field, as the finite field grows.

We show that the number of points over a finite field, when counted suitably, is given by a twisted orbital integral. Thus we reduce the problem to the asymptotic behavior of twisted orbital integrals. We study the latter using explicit methods from local harmonic analysis and representation theory, including the Base Change Fundamental Lemma and the Kato--Lusztig formula.

In our proof, polynomials that are linear combinations of the $\mathbf{q}$-analogue of Kostant partition functions appear, and the key computation is to estimate their sizes. These polynomials (denoted by $\mathfrak M^0_{\lambda} (\mathbf q)$ in the paper) can be viewed as a non-dominant generalization of the $\mathbf{q}$-analogue of Kostant weight multiplicity. Some properties of them are noted in \cite{Pan}, but beyond this there does not seem to have been a lot of study into these objects. From our proof, it seems reasonable to expect that a more thorough study of the combinatorial and geometric properties of these polynomials would shed new light on the structure of affine Deligne--Lusztig varieties, as well as the structure of twisted orbital integrals.

An interesting point in our proof is that we need to apply the Base Change Fundamental Lemma, which is only available in general for mixed characteristic local fields. In fact, the proofs of this result by Clozel \cite{clozelFL} and Labesse \cite{labFL} rely on methods only available over characteristic zero, for example the trace formula of Deligne--Kazhdan. Thus our method crucially depends on the geometric theory of mixed characteristic affine Grassmannians as in \cite{BS} and \cite{Zhu}. To deduce Conjecture \ref{conj:chenzhu1} also for equal characteristic local fields, we apply results of He \cite{He} to prove the following.
\begin{thmB}[Theorem \ref{thm:parahoric}, Theorem \ref{prop:ind of p}] For any $Z \in \Sigma ^{\topp}$, the stabilizer $Stab_Z J$ is a parahoric subgroup of $J$. Moreover, these parahoric subgroups, as well as the quotient set $J \backslash \Sigma ^{\topp}$, are independent of the local field $F$ in a precise sense. In particular, the truth of Conjecture \ref{conj:chenzhu1} transfers between different local fields.
\end{thmB}

Our method also allows us to show that the stabilizers of the action of $J$ on $\Sigma ^{\topp}$ are all hyperspecial, when $b$ is basic and unramified. This reproves a result of Xiao--Zhu \cite{xiaozhu}, see Remark \ref{rem:M} below. For ramified $b$, we obtain the following result, which refines Conjecture \ref{conj:chenzhu1} in that it provides information about the stabilizers.
\begin{thmC}[Corollary \ref{cor:CZ} and (\ref{eq:to cite in intro})] Assume $\mathrm{char} (F) =0$. Assume $G$ is $F$-simple and adjoint, without type $A$ or $E_6$ factors. Assume $b$ is basic and ramified. Then we have
	$$\abs{J \backslash \Sigma ^{\topp}}\cdot \mathscr L_b = \sum _{ Z\in J \backslash \Sigma^{\topp} } \vol (\Stab_Z J )^{-1},$$
	where the volumes $\vol (\Stab_Z J )$ are computed with a fixed Haar measure on $J$, and $\mathscr L_b$ is a constant that depends only on $b$ and not on $\mu$.
\end{thmC}

In a future work we shall explore the possibility of applying Theorems A, B, C to determine the stabilizers in general.

\subsection{Overview of the proof} We now explain in more detail our proof of Conjecture \ref{conj:chenzhu1}. A standard reduction allows us to assume that $b$ is basic, and that $G$ is adjoint and $F$-simple. Throughout we also assume that $G$ is not of type $A$, which is already sufficient by the work of Nie \cite{nienew}. To simplify the exposition, we also assume that $G$ is split and not of type $E_6$. Then $\hatS  = \widehat T$, and we drop the subscript ``$\mathrm{rel}$'' for the weight spaces in Conjecture \ref{conj:chenzhu1}.

For any $s \in \Z_{>0}$, we let $F_{s}$ be the unramified extension of $F$ of degree $s$. We denote by $\mathcal H_s$ the spherical Hecke algebra $\mathcal H(G(F_s)// G(\Ok_{F_s}) )$. We may assume without loss of generality that $b$ is $s_0$-decent for a fixed $s_0 \in \mathbb N$, meaning that $b \in G(F_{s_0})$ and $$b \sigma(b)\cdots \sigma ^{s_0-1} (b) =1.$$

As mentioned above, our idea is to use the Lang--Weil estimate to relate the number of irreducible components to the asymptotics of twisted orbital integrals. Since $X_\mu(b)$ is only locally of (perfectly) finite type and we are only counting $J$-orbits of irreducible components, we need a suitable interpretation of the Lang--Weil estimate. The precise output is the following.
\begin{prop}[Proposition \ref{prop:point counting}]\label{prop:intro TO} Let $s \in s_0 \mathbb N$. Let $f_{\mu,s} \in \mathcal H_s$ be the characteristic function of $G(\Ok_{F_s})\mu(\pi_F)G(\Ok_{F_s}),$ and let $TO_b(f_{\mu ,s})$ denote the twisted orbital integral of $f_{\mu,s}$ along $b\in G(F_s)$. We have
	\begin{align}\label{eq:intro TO}
	TO_b  (f_{\mu ,s}) = \sum _{ Z\in J\backslash \Sigma^{\topp} } \vol (\Stab_ZJ)^{-1} q^{s \dim X_{\mu} (b)} + o(q^{s \dim X_{\mu} (b) }), \quad s \gg 0 .
	\end{align}
\end{prop}

To proceed, we apply the Base Change Fundamental Lemma to compute $TO_b (f_{\mu,s})$. There are two problems in this step. Firstly, the Base Change Fundamental Lemma can only be applied to stable twisted orbital integrals. This problem is solved because one can check that $TO_b(f_{\mu,s})$ is in fact equal to the corresponding  stable twisted orbital integral. Secondly, the general Base Change Fundamental Lemma is only available for $\mathrm{char} (F )=0$. The way to circumvent this was already discussed in \S \ref{subsec:novelty} above.

We define $
\delta : =(\mathrm{rk}_F J_b- \mathrm{rk}_F G)/2.$ Up to lower order error terms,  we may combine the  above-mentioned computation of $TO_b(f_{\mu,s})$ with asymptotics of the Kato--Lusztig formula \cite{Kato} to rewrite the left hand side of (\ref{eq:intro TO}), and we may use the dimension formula for $X_{\mu} (b)$ (by Hamacher \cite{Ham}
and Zhu \cite{Zhu}) to rewrite the right hand side of (\ref{eq:intro TO}). The result is the following:
\begin{align}\label{eq:BCFL} \sum _{\lambda \in X^*(\widehat T) ^+, \lambda \leq \mu} &  \dim V_{\mu} (\lambda) \cdot  \mathfrak M^0_{s\lambda} (q^{-1})  \\ \nonumber & = \pm \sum _{ Z\in J \backslash \Sigma^{\topp}} \vol (\Stab_ZJ )^{-1} q^{s\delta} + o(q^{s \delta}), \quad  s \gg 0,
\end{align}
where each $\mathfrak M^0_{s\lambda} (q^{-1})$ is the value at $\mathbf q= q^{-1}$ of a polynomial $\mathfrak M^0_{s\lambda} (\mathbf q) \in \mathbb C[\mathbf q]$, given explicitly in terms of the $\mathbf q$-analogues of Kostant's partition functions (see Definition \ref{defn:frakM} and \S \ref{subsec:Kostant partitions}).

The key computation needed to further analyze (\ref{eq:BCFL}) is summarized in the following.
\begin{prop}\label{prop:intro key} Let $\lambda_b ^+ \in X^*(\widehat T) ^+$ be the dominant conjugate of $\lambda_b$. For all $\lambda \in X^*(\widehat T) ^+ - \set{\lambda_b ^+}$, we have
	$$ \mathfrak M^0_{s\lambda} (q^{-1}) = o(q ^{s \delta}),~ s\gg 0 . $$
\end{prop}

When $G$ is the split adjoint $E_6$, we only prove a weaker form of Proposition \ref{prop:intro key}, which also turns out to be sufficient for our purpose.

Proposition \ref{prop:intro key} tells us that on the left hand side of (\ref{eq:BCFL}), only the summand indexed by $\lambda = \lambda _b^+$ has the ``right size''. Taking the limit, we obtain	\begin{align}\label{eq:in theorem C}
\dim V_{\mu} (\lambda_b)\cdot \mathscr L_b = \sum _{ Z\in J \backslash \Sigma^{\topp} } \vol (\Stab_Z J )^{-1},
\end{align} where $\mathscr L_b$ is independent of $\mu$.

In (\ref{eq:in theorem C}), we already see both the number $\dim V_{\mu} (\lambda_b)$ and the set $J\backslash \Sigma ^{\topp}$. In order to deduce the desired (\ref{eq:intro numerical}), one still needs some information on the volume terms $\vol (\mathrm{Stab}_ZJ)$. It turns out that even very weak information will suffice. In \S \ref{sec:action} we show that the right hand side of (\ref{eq:in theorem C}) is equal to $R(q)$, where $R(T)\in\mathbb{Q}(T)$ is a rational function which is independent of $F$ in a precise sense. Moreover we show that $|R(0)|=|J\backslash \Sigma^{\topp}|.$  Therefore, the desired (\ref{eq:intro numerical}) will follow from (\ref{eq:in theorem C}), if we can show that
\begin{align}\label{eq:intro to show}
\mathscr L_b =S(q),\quad \text{for some~}S(T) \in \mathbb Q (T) \text{~with~}|S(0)|=1.
\end{align}
A remarkable feature of the formulation (\ref{eq:intro to show}) is that it is independent of $\mu$. We recall that in the works of Hamacher--Viehmann and Nie, special assumptions on $\mu$ are made. Hence we are able to bootstrap from known cases of Conjecture \ref{conj:chenzhu1} (for example when $\mu=\lambda_b^+$) to establish (\ref{eq:intro to show}), and hence to establish Conjecture \ref{conj:chenzhu1} in general.

We end our discussion with the following remark.
\begin{rem}\label{rem:M}
	At the moment, we are unable to directly compute the rational functions $S(T)$ appearing in (\ref{eq:intro to show}) in general. To do this would require a much better understanding of the polynomials $\mathfrak{M}_{s\lambda}^0(\mathbf q)$. We are however able to compute $S(T)$ in a very special case. When $b$ is a basic unramified element in the sense of \cite{xiaozhu}, we show directly that (\ref{eq:intro to show}) is satisfied by $S(T) \equiv 1$, see \S \ref{sub:unr}. From this we deduce the conjecture for $b$, as well as the equality $\vol(\Stab_ZJ)=1$ for each $Z\in \Sigma^{\topp}$. This last equality implies (according to our normalization) that $\Stab_ZJ$ is a hyperspecial subgroup of $J$. This gives another proof of a result in \cite{xiaozhu}, avoiding their use of Littelmann paths.	
\end{rem}
\ignore{\begin{rem}
		Our proof of Proposition \ref{prop:intro key} in the case $G$ is split of type $D_n$ with $n$ odd turns out to be significantly different from the other cases. In this case we devise a combinatorial method to analyze the polynomials $\mathfrak{M}_{s\lambda}^0(\mathbf q)$, using certain binary trees whose vertices are decorated by pairs of roots in the root system, see \S \ref{subsec:combinatorics for Dn}. This method could possibly be generalized to analyze more instances of $\mathfrak{M}_{s\lambda}^0(\mathbf q)$.
	\end{rem}
}
\subsection{Organization of the paper}
In \S \ref{sec:prelim}, we introduce notations and state the Chen--Zhu conjecture. In \S \ref{sec:action}, we study the action of $J$ on $\Sigma ^{\topp}$, proving Theorem B. In \S \ref{sec:count}, we prove Proposition \ref{prop:intro TO}, and then apply the Base Change Fundamental Lemma to compute twisted orbital integrals. In \S \ref{sec:Sat} we review the relationship between the coefficients of the Satake transform and the $\mathbf q$-analogue of Kostant's partition functions, and draw some consequences. In \S \ref{sec:main} we state Proposition \ref{key-bound} as a more technical version of Proposition \ref{prop:intro key}. We then deduce Conjecture \ref{conj:chenzhu1} from Proposition \ref{key-bound}. The proof of Proposition \ref{key-bound} is given in \S \ref{sec:proof of key est}, \S \ref{sec:part II}, and \S \ref{sec:part III}, by analyzing the root systems case by case. In Appendix \ref{app}, we generalize our main result to quasi-split groups.
\subsection{Notations and conventions}
We order $\mathbb N$ by divisibility, and write $s \gg 0$ to mean ``for all sufficiently divisible $s \in \mathbb N$''. The notation $\lim_{s\to \infty}$ will always be understood as the limit taken with respect to the divisibility on $\mathbb N$. If $f(s),g(s)$ are $\mathbb C$-valued functions defined for all sufficiently divisible $s \in \mathbb N$, we write $$f(s) = o(g(s))$$ to mean that $\lim _{s\to \infty} f(s)/g(s)=0$, where the limit is understood in the above sense. We write $$f(s) = O(g(s))$$ to mean that $$\exists M > 0 ~ \exists s_0 \in \mathbb N~\forall s \in s_0 \mathbb N, ~\abs{f(s)/g(s)} < M.$$ In this case we do not require $f(s)/g(s)$ to be bounded, or even defined, for all $s \in \mathbb N$.

For any finitely generated abelian group $X$, we write $X_{\mathrm{free}}$ for the free quotient of $X$. We let $\C[X]$ be the group algebra of $X$ over $\C$, and denote by $e^x$ the element in $\C[X]$ corresponding to $x\in X$.

We use $\mathbf q$, or $\mathbf q^{-1}$, or sometimes $\mathbf q^{-1/2}$, to denote the formal variable in a polynomial or power series ring.

The following lemma is elementary and will be used repeatedly in the paper. We omit its proof.
\begin{lem}\label{lem:elementary about coinv} Let $\Gamma$ be a finite group.	Let $X$ be a $\Z[\Gamma]$-module which is a finite free $\Z$-module. As usual define the norm map $\N : X \to X, ~ x\mapsto \sum _{\gamma \in \Gamma } \gamma (x).$ Let $Y \subset X$ be a $\Gamma$-stable subgroup. Then the following statements hold.
	\begin{enumerate}
		\item\label{item:1 in coinv} The kernel of the map $Y \to X_{\Gamma,\mathrm{free}}$ is equal to $\set{y\in Y \mid \N (y) =0}.$ In particular, it is also equal to the kernel of $Y \to Y_{\Gamma, \mathrm{free}}$.
		\item\label{item:2 in coinv} Suppose $Y$ has a finite $\Z$-basis which is stable under $\Gamma$. Then the $\Gamma$-orbits in this $\Z$-basis define distinct elements of $Y_{\Gamma}$, which form a $\mathbb Z$-basis of $Y_{\Gamma}$. In particular $Y_{\Gamma}$ is a finite free $\Z$-module.
		\item\label{item:3 in coinv} The map $\N:X \to X$ factors through a map
		$ X_{\Gamma} \to X^{\Gamma}.$  We have a canonical isomorphism
		$ X_{\Gamma} \otimes \Q \isom X^{\Gamma} \otimes \Q $ given by $\frac{1}{\abs{\Gamma}} \N$. \qed
	\end{enumerate}	
\end{lem}

\section{Notations and preliminaries}\label{sec:prelim}
\subsection{Basic notations}\label{subsec:basic notations}
Let $F$ be a non-archimedean local field with valuation ring $\Ok_F$ and residue field $k_F= \mathbb F_q$. Let $\pi_F\in F$ be a uniformizer. Let $p$ be the characteristic of $k_F$. Let $L$ be the completion of the maximal unramified extension of $F$, with valuation ring $\Ok_{L}$ and residue field $k =\overline{ k_F}$. Let $\Gamma=\Gal(\overline F/F)$ be the absolute Galois group. Let $\sigma$ be the Frobenius of $L$ over $F$.

Let $G$ be a connected reductive group over $\Ok_F$. In particular its generic fiber $G_F$ is an unramified reductive group over $F$, i.e.~is quasi-split and splits over an unramified extension of $F$. Then $G(\Ok_F)$ is a hyperspecial subgroup of $G(F)$. Fix a maximal $\Ok_F$-split torus $A$ of $G$. Let $T$ be the centralizer of $A_F$ in $G_F$, and fix a Borel subgroup $B\subset G_F$ containing $T$. Then $T$ is an unramified maximal torus in $G_F$. In the following we often abuse notation and simply write $G$ for $G_F$.

Note that $T_L$ is a split maximal torus in $G_L$. Let $V$ be the apartment of $G_L$ corresponding to $T_L$.
The hyperspecial vertex $\mathfrak{s}$ corresponding to $G(\Ok_L)$ is then contained in $V$. We have an identification $V\cong X_*(T)\otimes \mathbb{R}$ sending $\mathfrak s$ to $0$. Let $\mathfrak{a} \subset V$ be the alcove whose closure contains $\mathfrak s$, such that the image of $\mathfrak a$ under $V\cong X_*(T)\otimes \mathbb{R}$ is contained in the anti-dominant chamber. The action of $\sigma$ induces an action on $V$, stabilizing both $\mathfrak{a}$ and $\mathfrak{s}$. We let $\mathcal I$ be the Iwahori subgroup of $G(L)$ corresponding to $\mathfrak {a}$.

\subsection{The Iwahori--Weyl group}\label{subsec:IW}
The relative Weyl group $W_0$ over $L$ and the Iwahori--Weyl group $W$ are defined by
$$
W_0=N(L)/T(L), \qquad W=N(L)/T(L) \cap \mathcal{I},
$$
where $ N$ denotes the normalizer of $T$ in $G$. Note that $W_0$ is equal to the absolute Weyl group, as $T_L$ is split.

We have a natural exact sequence
\begin{align}\label{eq:exact seq for W}
1 \To X_*(T) \To W \To W_0 \To 1.
\end{align}
The canonical action of $N(L)$ on $V$ factors through an action of $W$, and we split the above exact sequence by identifying $W_0$ with the subgroup of $W$ fixing $\mathfrak s \in V$. See \cite[Proposition 13]{HainesRapoport} for more details. When considering an element $\lambda \in X_*(T)$ as an element of $W$, we write $t^\lambda \in W$.
For any $w \in W$, we choose  a representative $\dot w \in N(L)$.

Let $W_a$ be the associated affine Weyl group, and $\mathbb{S}$ be the set of simple reflections associated to $\mathfrak{a}$.  Since $\mathfrak{a}$ is $\sigma$-stable, there is a natural action of $\sigma$ on $\mathbb{S}$. We let $\mathbb{S}_0\subset \mathbb{S}$ be the set of simple reflections fixing $\mathfrak{s}$. Then $W$ contains $W_a$ as a normal subgroup, and we have a natural splitting $W=W_a \rtimes \Omega,$ where $\Omega$ is the stabilizer of $\mathfrak{a}$ in $W$ and is isomorphic to $\pi_1(G)$. The length function $\ell$ and the Bruhat order $\le$ on the Coxeter group $(W_a,\mathbb S)$ extend in a natural way to $W$.

For any subset $P$ of $\mathbb S$, we shall write $W_P$ for the subgroup of $W$ generated by $P$.

For $w, w' \in W$ and $s\in \mathbb S$, we write $w \xrightarrow{s}_\sigma w'$ if $w'=s w \sigma(s)$ and $\ell(w') \le \ell(w)$. We write $w \to_\sigma w'$ if there is a sequence $w=w_0, w_1, \dotsc, w_n=w'$ of elements in $W$ such that for any $i$, $w_{i-1} \xrightarrow{s_i}_\sigma w_i$ for some $s_i \in \mathbb S$. Note that if moreover, $\ell(w')<\ell(w)$, then there exists $i$ such that $\ell(w)=\ell(w_i)$ and $s_{i+1} w_i \sigma(s_{i+1})<w_i$.

We write $w \approx_\sigma w'$ if $w \to_\sigma w'$ and $w' \to_\sigma w$. It is easy to see that $w  \approx_\sigma w'$ if $w \to_\sigma w'$ and $\ell(w)=\ell(w')$. We write $w\ \tilde{\approx}_\sigma w'$ if there exists $\tau\in \Omega$ such that $w \approx_\sigma \tau w'\sigma(\tau)^{-1}$.

\subsection{The set $B(G)$}\label{subsec:B(G)} For any $b \in G(L)$, we denote by $[b]=\{g^{-1} b \sigma(g)\mid g \in G( L)\}$ its $\sigma$-conjugacy class. Let $B(G)$ be the set of $\sigma$-conjugacy classes of $G(L)$. The $\sigma$-conjugacy classes have been classified by Kottwitz in \cite{kottwitzisocrystal} and \cite{kottwitzisocrystal2}, in terms of the \emph{Newton map} $\bar \nu$ and the \emph{Kottwitz map} $\kappa$. The Newton map is a map
\begin{align}\label{eq:newton map}
\bar \nu: B(G) \to (X_*(T)^+_{\Q})^\sigma,
\end{align}
where $X_*(T)^+_{\Q}$ is the set of dominant elements in $X_*(T)_{\mathbb{Q}}:=X_*(T) \otimes \Q$. The Kottwitz map is a map $$\kappa = \kappa_G: B(G) \to \pi_1(G)_\Gamma.$$ By \cite[\S 4.13]{kottwitzisocrystal2}, the map
\begin{align}\label{eq:Kott invts}
(\bar \nu, \kappa): B(G) \to (X_*(T)^+_{ \Q})^\sigma \times \pi_1(G)_\Gamma
\end{align}
is injective.

Let $B(W,\sigma)$ denote the set of $\sigma$-conjugacy classes of $W$. The map $W \to G(L), w \mapsto \dot w$ induces a map $$\Psi: B(W, \sigma) \To B(G), $$ which is independent of the choice of the representatives $\dot w$. By \cite{He}, the map $\Psi$ is surjective. We again denote by $(\bar \nu, \kappa)$ the composition of (\ref{eq:Kott invts}) with $\Psi$. This composed map can be described explicitly, see \cite[\S 1.2]{HZ} for details.

The map $\Psi $ is not injective. However, there exists a canonical lifting to the set of \emph{straight $\sigma$-conjugacy classes}. By definition, an element $w \in W$ is called {\it $\sigma$-straight} if for any $n \in \mathbb N$, $$\ell(w \sigma(w) \cdots \sigma^{n-1}(w))=n \ell(w).$$ This is equivalent to the condition that $\ell(w)=\langle\bar \nu_w, 2 \rho\rangle$, where $\rho$ is the half sum of all positive roots. A $\sigma$-conjugacy class of $W$ is called {\it straight} if it contains a $\sigma$-straight element. It is easy to see that the minimal length elements in a given straight $\sigma$-conjugacy class are exactly the $\sigma$-straight elements.

\begin{thm}[{\cite[Theorem 3.7]{He}}]\label{str-bg}
	The restriction of $\Psi: B(W, \sigma) \to B(G)$ gives a bijection from the set of straight $\sigma$-conjugacy classes of $W$ to $B(G)$. \qed
\end{thm}

\subsection{The affine Deligne--Lusztig variety $X_{P,w}(b)$}\label{2}
Let $\mathcal{P}$ be a standard $\sigma$-invariant parahoric subgroup of $G(L)$, i.e.~a $\sigma$-invariant parahoric subgroup that contains $\mathcal{I}$. In the following, we will generally abuse of notation to use the same symbol to denote a parahoric subgroup and the underlying parahoric group scheme. We denote by $P \subset \mathbb{S}$ the set of simple reflections  corresponding to $\mathcal P$. Then $\sigma(P)=P$. We have $$G(L)=\bigsqcup_{w \in W_P \backslash W/W_P} \mathcal{P}(\Ok_L) \dot w \mathcal{P}(\Ok_L).$$

For any $w \in W_P \backslash W/W_P$ and $b \in G(L)$, we set $$X_{P, w}(b)(k) :=\{g \mathcal{P} (\Ok _L) \in G(L)/ \mathcal{P}(\Ok_L) \mid g^{-1} b \sigma(g) \in \mathcal{P}(\Ok_L) \dot w \mathcal{P}(\Ok_L)\}.$$ If $\mathcal{P}=\mathcal{I}$ (corresponding to $P =\emptyset$), we simply write $X_w(b)(k)$ for $X_{\emptyset,w}(b)(k)$.

We freely use the standard notations concerning loop groups and partial affine flag varieties, see \cite[\S 9]{BS} or \cite[\S 1.4]{Zhu}. When $\mathrm{char} (F)> 0$, it is known that $X_{P,w}(b)(k)$ could be naturally identified with the set of $k$-points of a locally closed sub-ind scheme $X_{P,w}(b)$ of the partial affine flag variety $\Gr _{\mathcal P}$. When $\mathrm{char} (F) =0$, thanks to the recent breakthrough by Bhatt--Scholze \cite[Corollary 9.6]{BS} (cf.~also \cite{Zhu}), we can again identify $X_{P,w}(b)(k)$ with the $k$-points of a locally closed perfect sub-ind scheme $X_{P,w} (b)$ of the Witt vector partial affine flag variety $\Gr_{\mathcal{P}}$. In both cases, the (perfect) ind-scheme $X_{P,w} (b)$ is called an \emph{affine Deligne--Lusztig variety}, and one could consider topological notions related to the Zariski topology on $X_{P,w} (b)$.

We are mainly interested in the case when $\mathcal P = G_{\Ok_L}$. In this case the corresponding set of simple reflections is $K : =\mathbb{S}_0$. Recall from \S \ref{subsec:IW} that we fixed a splitting of (\ref{eq:exact seq for W}) using the hyperspecial vertex $\mathfrak s$. According to this splitting, the subgroup $W_0$ of $W$ is the same as $W_K$, the subgroup generated by $K = \mathbb S_0$. We have identifications $$W_K\backslash W/W_K = W_0 \backslash (X_*(T) \rtimes W_0) / W_0 \cong  X_*(T)/W_0 \cong  X_*(T)^+.$$
For $\mu\in X_*(T)^+$, we write $X_\mu(b)$ for $X_{K,t^\mu}(b)$.

We simply write $\Gr _G$ for $\Gr _{G_{\Ok_L}}$.
The relationship between the hyperspecial affine Deligne--Lusztig variety $X_{\mu} (b) \subset \Gr _G$ and the Iwahori affine Deligne--Lusztig varieties $X_{w} (b) \subset \Gr _{\mathcal I}$ is as follows. We have a projection $
\pi:\mathcal{FL}\rightarrow \Gr_{G}
$ which exhibits $\mathcal{FL}:=\Gr_{\mathcal I}$ as an \'{e}tale fibration over $\Gr_{G}$. Indeed the fiber of this map is isomorphic to  the fpqc quotient $L^+G/L^+\mathcal{I}$ where $L^+G, L^+\mathcal I$ are the positive loop groups attached to $G$ and $\mathcal{I}$. More concretely, $L^+G/L^+\mathcal{I}$ is a finite type flag variety over $k$ when $\mathrm{char} (F)> 0$, and is the perfection of a finite type flag variety over $k$ when $\mathrm{char} (F) = 0$. We have $$\pi^{-1}(X_\mu(b) )=X(\mu,b)^K:= \bigcup_{w\in W_0t^\mu W_0}X_w(b).$$

\subsection{Basic information about $X_\mu(b)$} \label{subsec:nonemptiness pattern}  For $\lambda, \lambda' \in X_*(T)_\Q$, we write $\lambda \leq \lambda'$ if $\lambda'-\lambda$ is a non-negative rational linear combination of positive coroots. Let $\mu \in X_*(T)^+$. As in \cite{RV}, we set $$B(G, \mu)=\{ [b]\in B(G)\mid \kappa([b])=\mu^\natural, \bar \nu_b\leq \mu^\diamond \}.$$
Here $\mu^\natural$ denotes the image of $\mu$ in $\pi_1(G)_\Gamma$, and $\mu^\diamond \in X_*(T) _{\Q}$ denotes the Galois average of $\mu$.

The following result is proved by Kottwitz \cite{kottwitzHodgeNewton} and Gashi \cite{Gashi}, strengthening earlier results of Rapoport--Richartz \cite{RR}, Kottwitz--Rapoport \cite{KR}, and Lucarelli \cite{lucarelli}.
\begin{thm}\label{KR-Gas}For $b \in G(L)$, we have $X_\mu(b) \neq \emptyset$ if and only if $[b]\in B(G,\mu)$.	\qed
\end{thm}
We now let $\mu\in X_*(T)^+$ and let $b \in G(L)$ such that $[b] \in B(G,\mu)$.
\begin{defn}
	We define $\mathrm{def}_G(b) : = \mathrm{rk}_F G - \mathrm{rk}_F J_b$, called the \emph{defect} of $b$.
\end{defn}

\begin{thm}\label{thm:dim}
	If $\mathrm{char} (F) > 0$, then $X_{\mu} (b)$ is a scheme locally of finite type over $k$. If $\mathrm{char} (F) = 0$, then $X_{\mu} (b)$ is a perfect scheme locally perfectly of finite type over $k$. In both cases the Krull dimension of $X_{\mu} (b)$ is equal to
	$$ \lprod{\mu - \bar \nu_b, \rho} - \frac{1}{2}\mathrm{def}_G(b). $$
\end{thm}\begin{proof} The local (perfectly) finiteness is proved by Hamacher--Viehmann \cite[Lemma 1.1]{HV}, cf.~\cite{HVfinite}. The dimension formula is proved by Hamacher \cite{Ham} and Xinwen Zhu \cite{Zhu}, strengthening earlier results of G\"{o}rtz--Haines--Kottwitz--Reuman \cite{GHKR} and Viehmann \cite{viehmanndim}.
\end{proof}
\begin{defn}For any (perfect) scheme $X$, we write $\Sigma(X)$ for the set of irreducible components of $X$. When $X$ is of finite Krull dimension, we write $\Sigma ^{\topp} (X)$ for the set of top dimensional irreducible components of $X$.
\end{defn}

Define the group scheme $J_b$ over $F$ by
\begin{align}\label{eq:defn of J}
J_b(R) = \set{g\in G(R\otimes _F L) \mid g^{-1} b \sigma (g) = b }
\end{align}
for any $F$-algebra $R$. Then $J_b $ is an inner form of a Levi subgroup of $G$, see \cite[\S 1.12]{RZ96} or \cite[\S 1.11]{RR}. The group $J_b(F)$ acts on $X_{\mu}(b)$ via scheme automorphisms. In particular $J_b(F)$ acts on $\Sigma (X_{\mu} (b))$ and on $\Sigma ^{\topp} (X_{\mu} (b))$. The following finiteness result is proved in \cite[Theorem 1.1]{HVfinite}.
\begin{lem}\label{lem:finiteness} The set $J_b(F) \backslash \Sigma (X_{\mu } (b))$ is finite.\qed
\end{lem}
\begin{defn} We write $\mathscr N(\mu,b)$ for the cardinality of $J_b(F) \backslash \Sigma^{\topp}(X_\mu(b)) $.
\end{defn}

\subsection{The Chen--Zhu conjecture}\label{subsec:CZ Conj}
In this paper we shall utilize the usual Langlands dual group (as a reductive group over $\C$ equipped with a pinned action by the Galois group), rather than the Deligne--Lusztig dual group which is used in \cite{HV}. As a result, our formulation of the Chen--Zhu conjecture below differs from \cite[Conjecture 1.3, \S 2.1]{HV}. However it can be easily checked that the two formulations are equivalent.

The Frobenius $\sigma$ acts on $X_*(T)$ via a finite-order automorphism, which we denote by $\theta$. Let $\widehat G$ be the usual dual group of $G$ over $\C$, which is a reductive group over $\C$ equipped with a Borel pair $(\widehat B, \widehat {T})$ and isomorphisms $X^*(\widehat{ T}) \cong X_*(T), X_*(\widehat{ T} ) \cong X^*(T)$. These last isomorphisms, which will be regarded as equalities, identify the positive roots (resp.~coroots) with the positive coroots (resp.~roots). For more details on the dual group see \S \ref{subsec:general facts} below.

\begin{defn}\label{defn:subtorus}
	Let $\hatS$ be the identity component of the $\hat \theta$-fixed points of $\widehat T$. Equivalently, $\hatS$ is the sub-torus of $\widehat T$ such that the map $X^*(\widehat T) \to X^*(\hatS)$ is equal to the map $X^*(\widehat T) \To X^* (\widehat T)_{\hat \theta , \mathrm{free}}.$
\end{defn}

\begin{defn}\label{defn:relative weight space}
	For $\mu \in X_*(T) ^+ = X^*(\widehat T)^+$, let $V_{\mu}$ be the highest weight representation of  $\widehat{G}$ of highest weight $\mu$. For all $\lambda' \in X^*(\widehat{T})$, we write $V_{\mu} (\lambda')$ for the $\lambda'$-weight space in $V_{\mu}$ as a representation of $\widehat{T}$. For all $\lambda \in X^*(\hatS)$, we write $V_{\mu} (\lambda)_{\mathrm{rel}}$ for the $\lambda$-weight space in $V_{\mu}$ as a representation of $\hatS$. \end{defn}

As in \S \ref{subsec:nonemptiness pattern}, let $\mu \in X_*(T)^+$, and let $[b]\in B(G,\mu)$. By Lemma \ref{lem:elementary about coinv} (\ref{item:3 in coinv}) we identify $X_*(T)_{\Q} ^{\theta}$ with $$ X_*(T) _{ \theta} \otimes \Q = X_*(T)_{\theta ,\mathrm{free} } \otimes \Q = X^*(\widehat T) _{\hat \theta , \mathrm{free}} \otimes \Q = X^*(\hatS) \otimes \Q, $$ and we shall view $\bar \nu_b$ (see (\ref{eq:newton map})) as an element of $X^*(\hatS) \otimes \Q$. We also have $\kappa(b) \in \pi _1 (G) _{\Gamma} = \pi_1 (G)_{\sigma}$, which is equal to the image of $\mu$.

Let $\widehat Q$ be the root lattice inside $X^*(\widehat T)$. Applying Lemma \ref{lem:elementary about coinv} to $X= X^*(\widehat T)$ and $Y = \widehat Q$, we obtain:
\begin{itemize}
	\item $\widehat Q_{\hat \theta}$ is a free $\Z$-module. It injects into $X^*(\widehat T) _{\hat \theta}$ and also injects into $X^*(\widehat T) _{\hat \theta, \mathrm{free}} = X^*(\hatS)$.
	\item The image of the simple roots in $\widehat Q$ in $\widehat Q_{\hat \theta}$ (as a set) is a $\Z$-basis of $\widehat Q_{\hat \theta}$. We call members of this $\Z$-basis the \emph{relative simple roots} in $\widehat Q_{\hat \theta}$.
\end{itemize}

\begin{lem}\label{lem:defn of lambda_b}There is a unique element $ \tilde \lambda_b \in X^*(\widehat T)_{\hat \theta}$ satisfying the following conditions:
	\begin{enumerate}
		\item The image of $\tilde \lambda_b$ in $\pi_1(G) _{\sigma}$ is equal to $\kappa(b)$.
		\item In $X^*(\hatS) \otimes \Q$, the element $ (\tilde \lambda_b)|_{\hatS} - \bar \nu_b$ is equal to a linear combination of the relative simple roots in $\widehat Q_{\hat \theta}$, with coefficients in $\Q \cap (-1,0]$. Here $(\tilde \lambda_b)|_{\hatS}$ denotes the image of $\tilde \lambda_b$ under the map $X^*(\widehat T) _{\hat \theta} \to X^*(\widehat T) _{\hat \theta,\mathrm{free}} = X^*(\hatS)$.
	\end{enumerate}
\end{lem}
\begin{proof}
	This is just a reformulation of \cite[Lemma 2.1]{HV}.
\end{proof}
\begin{defn}\label{defn:lambda_b}
	Let $\tilde \lambda_b \in X^*(\widehat T) |_{\hat \theta}$ be as in Lemma \ref{lem:defn of lambda_b}. We write $\lambda_b$ for $(\tilde \lambda_b )|_{\hatS} \in X^*(\hatS)$.
\end{defn}
\begin{conj}[Miaofen Chen, Xinwen Zhu]\label{conj:chenzhu} Let $\mu \in X_*(T)^+$ and let $[b]\in B(G,\mu)$. There exists a natural bijection between $J_b(F) \backslash\Sigma ^{\topp} (X_{\mu} (b))$ and the Mirkovi\'{c}--Vilonen basis of $V_{\mu} (\lambda_b) _{\mathrm{rel}}$.
\end{conj}
\begin{defn}
	For $\mu \in X_*(T)^+$ and $[b]\in B(G,\mu)$, we write $\mathscr N(\mu,b)$ for the cardinality of $J_b(F)\backslash\Sigma^{\topp}(X_\mu(b))$, and write $\mathscr M(\mu,b)$ for $\dim V_{\mu} (\lambda_b)_{\mathrm{rel}}$.
\end{defn}
Conjecture \ref{conj:chenzhu} has the following numerical consequence:

\begin{conj}[Numerical Chen--Zhu]\label{conj:numerical Chen-Zhu} Let $\mu \in X_*(T)^+$ and let $[b]\in B(G,\mu)$. We have $$\mathscr{N}(\mu,b)=\mathscr M(\mu,b).$$
\end{conj}

In \cite{nie}, Nie obtained the following results:
\begin{thm}[Nie]\label{thm:nie orginal}
	$\quad$
	\begin{enumerate}
		\item In order to prove Conjecture \ref{conj:chenzhu}, it suffices to prove it when $G$ is adjoint and $b$ is basic.
		\item There is a natural surjective map from the Mirkovi\'{c}--Vilonen basis of $V_\mu(\lambda_b)_{\mathrm{rel}}$ to the set $J_b(F)\backslash\Sigma^{\topp}(X_\mu(b))$. Thus in order to prove Conjecture \ref{conj:chenzhu}, it suffices to prove Conjecture \ref{conj:numerical Chen-Zhu}.
		\item Conjecture \ref{conj:chenzhu} holds if $\mu$ is a sum of  dominant minuscule elements. In particular, it holds if all absolute simple factors of $G^{\ad}$ are of type $A$.
	\end{enumerate}
\end{thm}

\begin{rem}\label{rem:history}
	After the present paper was finished, Nie uploaded online new versions of the preprint \cite{nienew}, in which he proved Conjecture \ref{conj:chenzhu} in full generality. His methods are independent of ours. The present paper  depends logically only on Nie's results stated in Theorem \ref{thm:nie orginal}, see Remark \ref{rem:logical} for more details.
\end{rem}

A further standard argument, for example \cite[\S6]{HZ}, shows that one can also reduce the proof of Conjecture \ref{conj:numerical Chen-Zhu} to the case where $G$ is $F$-simple. Therefore in view of Theorem \ref{thm:nie orginal} (1) (2), we have:
\begin{prop}\label{prop:reduction}In order to prove Conjecture \ref{conj:chenzhu}, it suffices to prove Conjecture \ref{conj:numerical Chen-Zhu} when $G$ is adjoint, $F$-simple, and $b\in G(L)$ represents a basic $\sigma$-conjugacy class.	\qed
\end{prop}

\section{The action of $J_b(F)$}
\label{sec:action}
\subsection{The stabilizer of a component}In this section we study the stabilizer in $J_b(F)$ of an irreducible component of $X_\mu(b) $. Here as before we let $\mu \in X_*(T) ^+$ and $[b] \in B(G,\mu)$. The first main result is the following.

\begin{thm}\label{thm:parahoric}
	The stabilizer in $J_b(F)$ of each $Z\in \Sigma(X_\mu(b))$ is a parahoric subgroup of $J_b(F)$.
\end{thm}

We first reduce this statement to a question about the Iwahori affine Deligne--Lusztig varieties $X_w(b), w \in W_0 t^{\mu} W_0$. Note that $J_b(F)$ acts on each $X_w(b)$ via automorphisms.
\begin{prop}\label{prop:projection} The projection $
	\pi:\mathcal{FL}\rightarrow \Gr_{G}
	$ induces a bijection between $
	\Sigma(X_\mu(b) )$ and $\Sigma(X(\mu,b)^K)$ compatible with the action of $J_b(F)$. Moreover, this bijection maps $\Sigma ^{\topp} (X_{\mu} (b))$ onto $\Sigma ^{\topp} (X(\mu, b) ^K)$.  \end{prop}
\begin{proof}
	This follows from the fact that the fiber of $\pi$ is (the perfection of) a flag variety.
\end{proof}
In view of Proposition \ref{prop:projection}, the proof of Theorem \ref{thm:parahoric} reduces to showing that the stabilizer of each irreducible component of $X(\mu,b)^K$ is a parahoric subgroup of $J_b(F)$.

Now let $Y \in \Sigma ( X(\mu,b)^K)$. Then since each $X_w(b)$ is locally closed in $\mathcal{FL}$, there exists $w\in W_0t^\mu W_0$ such that $ Y\cap X_w(b) $ is open dense in $Y$ and is an irreducible component of $X_w(b)$. Since the action of $J_b(F)$ on $X(\mu,b)^K$ preserves $X_w(b)$, it follows that $j\in J_b(F)$ stabilizes $Y$ if and only if $j$ stabilizes $Y\cap X_w(b)$. Hence we have reduced to showing that the stabilizer in $J_b(F)$ of any element of $\Sigma(X_w(b))$ is a parahoric subgroup. We will show that this is indeed the case in Proposition \ref{prop:stab parahoric} below.

One important tool needed in our proof is the following result, which is \cite[Corollary 2.5.3]{GoHe}.
\begin{prop}\label{DL-reduction}Let $w\in W$,  and let $s\in\mathbb{S}$ be a simple reflection. \begin{enumerate}
		\item\label{item:1 in DL} If $\ell(sw\sigma(s))=\ell (w)$, then there exists a universal homeomorphism $X_w(b)\rightarrow X_{sw\sigma(s)}(b)$.
		\item\label{item:2 in DL} If $\ell(sw\sigma(s))<\ell(w)$, then there is a decomposition $X_{w}(b)=X_1\sqcup X_2$, where $X_1$ is closed and $X_2$ is open, and such that there exist morphisms $X_1\rightarrow X_{sw\sigma(s)}(b)$ and $X_2\rightarrow X_{sw}(b)$, each of which is the composition of a Zariski-locally trivial fiber bundle with one-dimensional fibers and a universal homeomorphism.
	\end{enumerate}
	Moreover the universal homeomorphism in (\ref{item:1 in DL}) and the morphisms $X_1\rightarrow X_{sw\sigma(s)}(b)$ and $X_2\rightarrow X_{sw}(b)$ in (\ref{item:2 in DL}) are all equivariant for the action of $J_b(F)$. \qed
\end{prop}

\begin{prop}\label{prop:stab parahoric}
	Assume $X_w(b)\neq\emptyset$ and let $Z\in \Sigma(X_w(b))$. The stabilizer in
	$J_b(F)$ of $Z$ is a parahoric subgroup of $J_b(F)$.
\end{prop}
\begin{proof}
	We prove this by induction on $\ell(w)$. Assume first that $w\in W$ is of minimal length in its $\sigma$-conjugacy class. Then $X_w(b)\neq \emptyset$ implies $\Psi(w)=b$, i.e.~$w$ and $b$ represent the same $\sigma$-conjugacy class in $B(G)$,  by \cite[Theorem 3.5]{He}. In this case, by \cite[Theorem 4.8]{He} and its proof, there is an explicit description of the stabilizer of an irreducible component which we recall.
	
	Let ${}^PW\subset W$ (resp.~$W^{\sigma(P)} \subset W$) denote the set of minimal representatives of the cosets $W_P\backslash W$ (resp.~$W/ W_{\sigma(P)}$). Let ${} ^P W ^{\sigma (P)}$ be the intersection ${}^P W \cap W^{\sigma (P)}$ (cf.~\cite[\S 1.6]{He}). By \cite[Theorem 2.3]{He}, there exists $P\subset \mathbb{S}$, $x\in {}^PW^{\sigma(P)}$, and $u\in W_P$, such that:
	\begin{itemize}
		\item $W_P$ is finite.
		\item $x$ is $\sigma$-straight and $x^{-1}\sigma(P)x=P$.
		
	\end{itemize}
	In this case, there is a $J_b(F)$-equivariant universal homeomorphism between $X_w(b)$ and $X_{ux}(b)$, and we have
	$\Psi(ux)=\Psi(w)$, see \cite[Corollary 4.4]{He}. Hence we may assume $w=ux$. By \cite[Lemma 3.2]{He} we have $\Psi(x)=\Psi(w)$, and therefore we may assume $b=\dot{x}$. Upon replacing $P$, we may assume $P$ is minimal with respect to a fixed choice of $x$ and $u$ satisfying the above properties.
	
	Let $\mathcal{P}$ denote the parahoric subgroup of $G(L)$ corresponding to $P$. The proof of \cite[Theroem 4.8]{He} shows that$$X_{ux}(\dot{x})\cong J_{\dot{x}}(F)\times_{J_{\dot{x}}(F)\cap\mathcal{P}}X_{ux}^{\mathcal{P}}(\dot{x}),$$ where $X_{ux}^\mathcal{P}(\dot{x})$ is the reduced $k$-subscheme of the (perfectly) finite type scheme $L^+\mathcal P/L^+\mathcal I$ whose $k$-points are $$X_{ux}^\mathcal{P}(\dot{x}) (k) = \{g\in \mathcal{P} (\Ok_L)/\mathcal{I}(\Ok_L)\mid g^{-1}\dot{x}\sigma(g)\in\mathcal{I}(\Ok_L)\dot{u}\dot{x}\mathcal{I}(\Ok_L)\}. $$ Thus it suffices to show the stabilizer in $J_{\dot{x}}(F)\cap \mathcal{P}(\Ok_L)$ of an irreducible component of $X^{\mathcal{P}}_{ux}(\dot{x})$ is a parahoric subgroup of $J_{\dot x} (F)$.

	Let $\overline{\mathcal{P}}$ denote the algebraic group over $k$, which is the reductive quotient of the special fiber of $\mathcal{P}$. Recall its Weyl group is naturally identified with $W_P$. Then $\mathcal{I}$ is the pre-image of a Borel subgroup $\overline{\mathcal{I}}$ of $\overline {\mathcal P}$ under the reduction map $\mathcal{P}\rightarrow \overline{\mathcal{P}}$. Let $\sigma_{\dot{x}}$ denote the automorphism of $\overline{\mathcal{P}}$ given by $\overline{p}\mapsto \dot{x}^{-1}\sigma(\overline{p})\dot{x}$. Then the natural map $L^+\mathcal{P}/L^+\mathcal{I}\rightarrow \overline{\mathcal{P}}/\overline{\mathcal{I}}$ induces an identification between $X^{\mathcal{P}}_{\dot{u}\dot{x}}(\dot{x})$ and (the perfection of) the finite type Deligne--Lusztig variety $$X'=\{\overline{p}\in \overline{\mathcal{P}}/\overline{\mathcal{I}}\mid \overline{p}^{-1}\sigma_{\dot{x}}(\overline{p})\in \overline{\mathcal{I}}\dot{u}\overline{\mathcal{I}}\}.$$
	
	The natural projection map $\mathcal{P}\rightarrow \overline{\mathcal{P}}$ takes $J_{\dot{x}}(F)\cap\mathcal{P}(\Ok_L)$ to $\overline{\mathcal{P}}^{\sigma_{\dot{x}}}$, and the action of $J_{\dot{x}}(F)\cap\mathcal{P}(\Ok_L)$ factors through this map. Since $P$ is minimal satisfying $u\in W_P$ and since $x^{-1}\sigma(P)x=P$,
	it follows that $u$ is not contained in any $\sigma_{\dot{x}}$-stable parabolic subgroup of $W_P$. Therefore by \cite[Corollary 1.2]{Go}, $X'$ is irreducible. It follows that the stabilizer of the irreducible component $1\times X^{\mathcal{P}}_{ux}(\dot{x})\subset X_{ux}(\dot{x})$ is $J_{\dot{x}}(F)\cap \mathcal{P}(\Ok_L)$, which is a parahoric of $J_{\dot{x}}(F)$. It also follows that the stabilizer of any other irreducible component of $ X_{ux}(\dot{x})$ is a conjugate parahoric.
	
	Now we assume $w$ is not of minimal length in its $\sigma$-conjugacy class. By \cite[Corollary 2.10]{HeNie}, there exists $w'\tilde{\approx}_{\sigma} w$ and $s\in \mathbb{S}$ such that $sw'\sigma(s)<w'$. Then by Proposition \ref{DL-reduction}, there is a  $J_b(F)$-equivariant universal homeomorphism between $X_w(b)$ and $X_{w'}(b)$.  Thus it suffices to prove the result for $X_{w'}(b)$.
	
	Let $Z'\in \Sigma  (X_{w'}(b))$, and let $X_1$ and $X_2$ be as in Proposition \ref{DL-reduction}. We have either $Z'\cap X_1$ or $Z'\cap X_2$ is open dense in $Z'$. Assume $Z'\cap X_1$ is open dense in $Z'$; the other case is similar. Since $J_b(F)$ preserves $X_1$, it suffices to show that the stabilizer of $Z'\cap X_1$ is a parahoric. From the description of $X_1$, there exists an element $V\in\Sigma(X_{sw'\sigma(s)}(b))$ such that $Z' \cap X_1 \rightarrow V$ is a fibration and is $J_b(F)$-equivariant. Therefore by induction, the stabilizer of $V$ is a parahoric of $J_b(F)$, and hence so is the stabilizer of $Z'\cap X_1$.
\end{proof}

\subsection{Independence of $F$ and volumes of stabilizers}
The second main result of this section is that the set of $J_b(F)$-orbits of irreducible components of $X_\mu(b)$ and the volume of the stabilizer of an irreducible component depend only on the affine root system together with the action of the Frobenius. In particular, it is \emph{independent of $F$} in a manner which we will now make precise. This fact is a key observation that we will need for later applications.

By \cite[\S6]{He}, the set of $J_b(F)$-orbits of top dimensional irreducible components of $X_w(b)$ depends only on the affine root system of $G$ together with the action of $\sigma$. This is proved by using the Deligne--Lusztig reduction method to relate the number of orbits to coefficients of certain \emph{class polynomials}, which can be defined purely in terms of the affine root system for $G$, see \emph{loc.~cit.}~for details. In view of the fibration $$\pi: X(\mu,b)^K\rightarrow X_\mu(b) $$
it follows that the same is true for $X_\mu(b) $. In particular, the number $\mathscr N(\mu,b)$ depends only on the affine root system and hence does not depend on the local field $F$.

We will need the following stronger result. To state it, we introduce some notations. Let $F'$ be another local field with residue field $\mathbb{F}_{q'}$. Let $G'$ be a connected reductive group over $\Ok_{F'}$. Let $T' \subset B' \subset G'_{F'}$ be analogous to $T \subset B \subset G_F$ as in \S \ref{subsec:basic notations}. Define the hyperspecial vertex $\mathfrak s'$, the apartment $V'$, and the anti-dominant chamber $ \mathfrak a'$ analogously to $\mathfrak s, V , \mathfrak a$. Assume there is an identification $V \cong V'$ that maps $X_*(T)^+$ into $X_*(T)^+$, maps $\mathfrak a$ into $\mathfrak a'$, maps $\mathfrak s$ to $\mathfrak s'$, and induces a $\sigma$-$\sigma'$ equivariant bijection between the affine root systems. Here $\sigma'$ denotes the $q'$-Frobenius acting on the affine roots system of $G'$. We fix such an identification once and for all. To the pair $(\mu,b)$, we attach a corresponding pair $(\mu',b')$ for $G'$ as follows. The cocharacter $\mu' \in X_*(T') ^+$ is defined to be the image of $\mu$ under the identification $X_*(T)^+ \cong X_*(T')^+$. To construct $b'$, we note that since $b$ is basic, it is represented by a unique $\sigma$-conjugacy class in $\Omega$. The identification fixed above induces an identification of Iwahori--Weyl groups $W\cong W'$, which induces a bijection on length-zero elements. Then $b'$ is represented by the corresponding length-zero element in $W'$.

By our choice of $b'$, the affine root systems of $J_b$ and $J_{b'}$ together with the actions of Frobenius are identified. We thus obtain a bijection between standard parahoric subgroups of $J_b$ and those of $J_{b'}$. Let $\mathcal{J}\subset J_b(F)$  and $\mathcal{J}'\subset J_{b'}(F')$ be parahoric subgroups. We say that $\mathcal{J}$ and $\mathcal{J}'$ are {\it conjugate}, if the standard parahoric conjugate to $\mathcal J$ is sent to the standard parahoric conjugate to $\mathcal J'$ under the above-mentioned bijection. In the following, we write $J: = J_b(F)$ and $J' : = J_{b'} (F')$.

\begin{thm}\label{prop:ind of p} There is a bijection $$J\backslash \Sigma^{\topp}(X_\mu(b) )\isom J'\backslash \Sigma^{\topp}(X_{\mu'}(b'))$$ with the following property. If $Z\in\Sigma^{\topp}(X_\mu(b) )$ and $Z'\in\Sigma^{\topp}(X_{\mu'}(b'))$ are such that $JZ$ is sent to $J'Z'$, then the parahoric subgroups $\Stab_Z(J)\subset J$ and $\Stab_{Z'}(J') \subset J'$ are conjugate.
\end{thm}

The theorem will essentially follow from the next lemma.

\begin{lem}\label{lem:bijection for Iwahori}
	Let $w'\in W'$ correspond to $w\in W$ under the identification $W\cong W'$. Then there is a  bijection $$\Theta: J\backslash \Sigma^{\topp}(X_w(b)) \isom J' \backslash \Sigma^{\topp}(X_{w'}(b'))$$ with the following property. If $Z \in \Sigma ^{\topp} (X_w(b))$ and $Z' \in \Sigma ^{\topp} (X_{w'}(b'))$ are such that $\Theta( JZ)  = J' Z'$, then $\Stab_Z(J)$ and $\Stab_{Z'}(J')$ are conjugate.
\end{lem}
\begin{proof}
	We induct on $\ell(w)$. First assume $w$ is minimal length in its $\sigma$-conju\-gacy class. Then by \cite{He}, $X_w(b)\neq\emptyset$  if and only if $\Psi(w)=b$, which holds if and only if $\Psi(w')=b'$, if and only if $X_{w'}(b')\neq\emptyset$. If this holds, then by \cite{He} the group $J$ acts transitively on $\Sigma^{\topp}(X_w(b))$, and similarly the group $J'$ acts transitively on $\Sigma^{\topp}(X_{w'}(b'))$. Hence the two sets $J\backslash \Sigma^{\topp}(X_w(b)) $ and $J' \backslash \Sigma^{\topp}(X_{w'}(b'))$ are both singletons. Let $\Theta$ be the unique map between them. The desired conjugacy of the stabilizers follows from the computation of $\Stab_Z(J)$ in Proposition \ref{prop:stab parahoric}.
	
	Now assume $w$ is not of minimal length in its $\sigma$-conjugacy class. Let $Z\in \Sigma^{\topp}(X_{w}(b))$. Then as in the proof of Proposition \ref{prop:stab parahoric}, there exists $w_1\tilde{\approx}_{\sigma}w$ and $s\in\mathbb{S}$ such that $sw_1\sigma(s)<w_1$. Then $X_w(b)$ is universally homeomorphic to $X_{w_1}(b)$. We fix such a universal homeomorphism and we obtain a corresponding element $Z_1\in\Sigma^{\topp}(X_{w_1}(b))$. By Proposition \ref{DL-reduction}, there exists $U\in \Sigma^{\topp}(X_{sw_1\sigma(s)}(b))$ or $U\in  \Sigma^{\topp}(X_{sw_1} (b))$ such that $Z_1$ is universally homeomorphic to a fiber bundle over $U$.
	We assume $U\in \Sigma^{\topp}(X_{sw_1\sigma(s)}(b))$; the other case is similar. Then $\Stab_Z(J)=\Stab_{U}(J)$. Note that the choice of $U$ depends on the choice of $w_1$ and a universal homeomorphism $X_{w}(b)\cong X_{w_1}$. However upon fixing these choices, the $J$-orbit of $U$ is canonically associated to the $J$-orbit of $Z$.
	
	By the induction hypothesis, we have a bijection
	$$\Theta_1: J\backslash \Sigma ^{\topp} (X_{sw_1 \sigma (s)}(b)) \isom J' \backslash\Sigma ^{\topp} (X_{s'w_1' \sigma' (s')}(b')),$$ where $s', w_1 ' \in W'$ correspond to $s,w_1$ respectively.
	Choose $$U' \in \Sigma ^{\topp} (X_{s'w_1' \sigma' (s')}(b'))$$ such that $J' U'=\Theta_1(JU).$ By the induction hypothesis, $\Stab _U(J)$ is conjugate to $\Stab_{U'}(J')$. Reversing the above process we obtain
	$Z'\in \Sigma^{\topp}(X_{w'}(b')))$ such that $\Stab_{U'} (J')=\Stab_{Z'}(J')$. Again the $J'$-orbit of $Z'$ is canonically associated to $U'$ upon fixing the universal homeomorphism $X_{w'(b')}\cong X_{w_1'}(b')$.
	
	We define the map $\Theta$ to send $JZ$ to $J'Z'$. Switching the roles of $G$ and $G'$, we obtain the inverse map of $\Theta$, and so $\Theta$ is a bijection as desired.
\end{proof}

\begin{proof}[Proof of Theorem \ref{prop:ind of p}] For each $w\in W$, fix a bijection $$\Theta:J\backslash\Sigma^{\topp}(X_w(b)) \isom  J' \backslash \Sigma ^{\topp} (X_{w'}(b')) $$ as in Lemma \ref{lem:bijection for Iwahori}. Let $Z\in\Sigma^{\topp}(X_{\mu}(b))$. Then the pre-image $\pi^{-1}(Z)$ under the projection $\pi:X(\mu,b)^K\rightarrow X_\mu(b) $ is a top dimensional irreducible component of $X(\mu,b)^K$. Hence there exists a unique $w\in W$ such that $X_w(b)\cap \pi^{-1}(Z)$ is open dense in $Z$. Moreover we have $X_w(b)\cap \pi^{-1}(Z)\in \Sigma^{\topp}(X_w(b))$. Write $Y$ for $X_w(b)\cap \pi^{-1}(Z)$, and choose $Y' \in \Sigma ^{\topp} (X_{w'}(b'))$ such that $\Theta (JY) = J' Y'$. Then since $\dim X(\mu,b)^K=\dim X(\mu',b')^{K'}$, the closure of $Y'$ in $X(\mu',b')^{K'}$ gives an element of $\Sigma^{\topp}(X(\mu',b')^{K'})$, whose $J'$-orbit is independent of the choice of $Y'$. Taking the image of the last element under the projection $X(\mu',b')^{K'}\rightarrow X_{\mu'}(b')$ we obtain an element $Z'\in \Sigma^{\topp}(X_{\mu'}(b'))$ by dimension reasons, and the orbit $J'Z'$ is independent of the choice of $Y'$. Moreover $\Stab_Z(J)$ is conjugate to $\Stab_{Z'}(J')$ since $\Stab_Y (J)$ is conjugate to $\Stab_{Y'} (J')$. The association $JZ \mapsto J'Z'$ gives a well-defined map
	$$J\backslash \Sigma^{\topp}(X_\mu(b) )\To J'\backslash \Sigma^{\topp}(X_{\mu'}(b')) $$ which satisfies the condition in the proposition. Switching the roles of $G$ and $G'$ we obtain the inverse map.
\end{proof}





For later applications we need some information on the sizes of the stabilizers appearing in Theorem \ref{prop:ind of p}.
We now assume that $b$ is basic, so that $G$ and $J_b$ are inner forms. Since $b$ is basic we may choose a representative $\dot{\tau}$ for $b$ where $\tau\in\Omega\subset W$. Using this one may identify the Iwahori--Weyl groups for $J_b$ and $G$ respecting the base alcoves. However the Frobenius action on $W$ (or $\mathbb{S}$), defined by $J_b$, is given by $\tau\sigma$, where $\tau$ acts via left multiplication. See for example \cite[\S5]{HZ} for more details. Since $G$ and $J_b$ are inner forms, the choice of a Haar measure on $G(F)$ determines a Haar measure on $J_b(F)$, and vice versa, see for example \cite[\S1]{kottTama}.

\begin{defn} We fix the Haar measure on $J_b(F)$ such that the volume of $G(\Ok_F)$ is $1$. For each $Z \in \Sigma ^{\topp} (X_{\mu} (b))$, we denote by $\vol (Z)$ the volume of the compact open subgroup $\Stab_Z (J_b(F))$ of $J_b(F)$ (see Theorem \ref{thm:parahoric}) under this Haar measure.
\end{defn}

\begin{cor}\label{cor:vol-ind-p}
	For each $Z\in \Sigma^{\topp}(X_\mu(b) )$, there exists a rational function $R(t)\in \mathbb{Q}(t)$ such that $$\vol(Z)=R(q).$$
	Moreover this rational function satisfies $R(0)=e(J_b)$ and is independent of the local field $F$. Here $e(J_b)$ is the Kottwitz sign $(-1) ^{\mathrm{rk}_F J_b - \mathrm{rk}_F G}$. More precisely, in the notation of Theorem \ref{prop:ind of p}, if $J'Z'$ corresponds to $JZ$, then $\vol(Z')=R(q')$.
\end{cor}
\begin{proof}
	Since $J_b$ splits over an unramified extension, the volume of a standard parahoric of $J_b(F)$ corresponding to any $\tau\sigma$-stable subset $K_J\subset \mathbb{S}$ can be calculated in terms of the affine root system. More precisely, let $\mathcal{K}_J$ be the corresponding parahoric subgroup of $J_b(F)$ and $\mathcal{I}_J$ be the standard Iwahori subgroup of $J_b(F)$. Then we have \begin{align*}\vol(\mathcal{K}_J(\Ok_F)) & =\frac{\vol(\mathcal{K}_J(\Ok_F))}{\vol(\mathcal{I}_J(\Ok_F))}.\vol(\mathcal{I}(\Ok_F)).\frac{\vol(\mathcal{I}_J(\Ok_F))}{\vol(\mathcal{I}(\Ok_F))} \\ &  =\frac{[\mathcal{K}_J(\Ok_F):\mathcal{I}_J(\Ok_F)]}{[G(\Ok_F):\mathcal{I}(\Ok_F)]}.\frac{\vol(\mathcal{I}_J(\Ok_F))}{\vol(\mathcal{I}(
		\Ok_F))} , \end{align*} where $\mathcal{I}$ is the standard Iwahori subgroup of $G(F)$ (whereas previously we denoted by $\mathcal I$ the standard Iwahori subgroup of $G(L)$). The term $[\mathcal{K}_J(\Ok_F):\mathcal{I}_J(\Ok_F)]$ (resp.~$[G(\Ok_F):\mathcal{I}(\Ok_F)]$) is just the number of $\mathbb{F}_q$-points in the finite type full flag variety associated to the reductive quotient of the special fiber of $\mathcal{K}_J$ (resp.~$G$).
	
	For any connected reductive group $\overline{H}$ over $\mathbb{F}_q$ and $\overline{B}$ a Borel subgroup, let $W_{\overline{H}}$ denote the absolute Weyl group. Then we have the Bruhat decomposition  $$\overline{H}/\overline{B}(\overline {\mathbb F}_q)=\bigsqcup_{w\in W_{\overline{H}}}S_w.$$
	We have $S_w(\mathbb{F}_q)\neq\emptyset$ if and only if $\sigma(w)=w$, in which case $S_w$ is an affine space of dimension $\ell(w)$ defined over $\mathbb{F}_q$. In particular $$\overline{H}/\overline{B}(\mathbb{F}_q)=\sum_{w \in W_{\overline H}^{\sigma}}q^{\ell(w)}.$$
	
	It follows that $[\mathcal{K}_J(\Ok_F):\mathcal{I}_J(\Ok_F)]$ and $[G(\Ok_F):\mathcal{I}(\Ok_F)]$ are both polynomials in $q$ with coefficients in $\Z$ and constant coefficient $1$, and the polynomials depend only on the root systems of the corresponding reductive quotients of the special fiber.
	
	Similarly the ratio $\frac{\vol(\mathcal{I}_J(\mathcal{O}_F))}{\vol(\mathcal{I}(\mathcal{O}_F))}$ can be computed as the ratio $$\frac{\det(1-q^{-1}\varsigma_J\mid V)}{\det(1-q^{-1}\varsigma\mid V)}=\frac{\det(q-\varsigma_J\mid V)}{\det(q-\varsigma\mid V)}$$
	where $\varsigma$ denotes the linear action of the Frobenius on $V=X_*(T)_{\mathbb{R}}$, and similarly for $\varsigma _J$, see \cite[\S1]{kottTama}. This is also a ratio of polynomials in $q$ with coefficients in $\Z$, and the ratio at $q=0$ is equal to $$\det(\varsigma_J)/\det(\varsigma)=(-1)^{\mathrm{rk_F}J_b-\mathrm{rk_F}G}=e(J_b). $$  Moreover the polynomials depend only on the affine root system of $G$ and the element $b$. The result follows.
\end{proof}

Finally we record the following immediate consequence of Theorem \ref{prop:ind of p}.

\begin{cor}
	If Conjecture \ref{conj:numerical Chen-Zhu} is true for all $p$-adic fields $F$, then it is true for all local fields $F$. \qed
\end{cor}
\textbf{From now on we will assume that $F$ is a $p$-adic field.}
\section{Counting points}
\label{sec:count}
\subsection{The decent case}\label{subsec:decent case}
For each $s\in \mathbb N$, let $F_s$ be the degree $s$ unramified extension of $F$ in $L$. Let $\Ok_s$ be the valuation ring of $F_s$, and let $k_s$ be residue field. The number $\mathscr N (\mu , b )$ depends on $b$ only via its $\sigma$-conjugacy class $[b] \in B(G)$. Recall that given $b \in G(L)$, one can associate a slope cocharacter  $\nu_b \in \Hom _L (\mathbb D , G)$, where $\mathbb D$ is the pro-torus with character group $\Q$.

\begin{defn}\label{defn:decent}Let $s \in \mathbb N$. We say that an element $b \in G(L)$ is \emph{$s$-decent}, if $s \nu _b$ is an integral cocharacter $\mathbb G_m \to G$ (as opposed to a fractional cocharacter), and
	\begin{align}\label{eq:decent}
	b \sigma(b) \cdots \sigma^{s -1} (b) = (s \nu_b) (\pi_F).
	\end{align}	
\end{defn}

\begin{lem}
	Assume $b\in G(L)$ is $s$-decent. Then $s \nu_b$ is defined over $F_{s}$, and $b$ belongs to $G(F_{s})$.
\end{lem}
\begin{proof}
	The proof is identical to the proof of \cite[Corollary 1.9]{RZ96}.
\end{proof}
By \cite[\S 4.3]{kottwitzisocrystal}, any class in $B(G)$ contains an element which is $s$-decent for some $s\in \mathbb N$. In the following, we hence assume without loss of generality that $b$ is $s_0$-decent, for some fixed $s_0\in \mathbb N$. We may and shall also assume that $s_0$ is divisible enough so that $T$ is split over $F_{s_0}$.

\begin{defn} Let $s\in s_0 \mathbb N$. Let $G_s: = \Res_{F_s / F} G$, so that $b \in G_s(F)$. Let $\Theta$ be the $F$-automorphism of $G_s$ corresponding to the Frobenius $\sigma \in \Gal(F_s/F)$. Let $G_{s, b\Theta}$ be the centralizer of $b\Theta$ in $G_s$, which is a subgroup of $G_s$ defined over $F$. Define
	$$ G(F_s) _{b\sigma} : = \set{ g\in G(F_{s} ) \mid   g^{-1} b \sigma(g) = b }. $$ Thus $G(F_s) _{b\sigma}$ is naturally identified with $G_{s, b\Theta} (F)$.
\end{defn}

\begin{lem}\label{lem:two centralizers} For $s\in s_0\mathbb N$, there is a natural isomorphism of $F$-groups $J_b \cong G_{s, b\Theta}.$
	Moreover, $J_b(F) =  G(F_{s}) _{b\sigma}$ as subgroups of $G(L)$.
\end{lem}
\begin{proof}
	Let $R$ be an $F$-algebra. Recall from (\ref{eq:defn of J}) that $$J_b(R) = \set{ g \in G(R\otimes _F L) \mid  g ^{-1} b \sigma (g) = b}.$$ It suffices to prove that for any $g \in J_b(R)$ we have $g \in G(R \otimes _F F_{s_0})$. Now such a $g$ commutes with $b \rtimes \sigma$, and so it commutes with $(b \rtimes \sigma)^{s_0}$. By (\ref{eq:decent}), we have  $(b \rtimes \sigma)^{s_0} = (s_0 \nu _b) (\pi_f) \rtimes \sigma ^{s_0}$. On the other hand, by the functoriality of the association $b \mapsto \nu_b$, we know that $g$ commutes with $\nu_b$. It follows that $g$ commutes with $ \sigma ^{s_0}$, and so $g\in G(R\otimes _F F_{s_0})$ as desired.
\end{proof}
\subsubsection{}\label{subsubsec:rationality discussion} We keep assuming that $F$ is $p$-adic.
In \S \ref{2}, we discussed the geometric structure on $X_{\mu}(b)$, as a locally closed subscheme of the Witt vector Grassmannian over $k = \overline {k_F}$. In the current setting, $X_{\mu} (b)$ is naturally ``defined over $k_{s_0}$''. More precisely, we can work with the version of the Witt vector affine Grassmannian as an ind-scheme over $k_{s_0}$ rather than over $k$, see \cite[Corollary 9.6]{BS} and cf.~\cite[\S 1.4]{Zhu}. Then the affine Deligne--Lusztig variety can be defined as a locally closed $k_{s_0}$-subscheme of the Witt vector affine Grassmannian, as in \cite[\S 3.1.1]{Zhu}. The key point here is that since $T$ is split over $F_{s_0}$, all the Schubert cells in the Witt vector affine Grassmannian are already defined over $k_{s_0}$, see \cite[\S 1.4.3]{Zhu}. We denote respectively by $\mathbb G \mathrm{r}_G $ and $\mathbb X_{\mu} (b)$ the Witt vector affine Grassmannian and the affine Deligne--Lusztig variety over $k_{s_0}$, and we continue to use $\Gr_G$ and $X_{\mu} (b)$ to denote the corresponding objects over $k$.

Let us recall the moduli interpretations of $\mathbb G \mathrm{r} _G$ and $\mathbb X_{\mu} (b)$. For any perfect $k_{s_0}$-algebra $R$, write
$W_{s_0}(R)$ for $ W(R) \otimes _{W(k_{s_0})} \Ok_{s_0}. $ Then $\mathbb G \mathrm{r}_G(R)$ is the set of pairs $(\mathcal E,
\beta)$, where $\mathcal E $ is a $G_{W_{s_0}(R)}$-torsor on $W_{s_0} (R)$, and $\beta$ is a trivialization of $\mathcal E$ on $W_{s_0}(R)[1/p]$, (see \cite[Lemma 1.3]{Zhu}).
We also have (see~\cite[(3.1.2)]{Zhu})
$$\mathbb X_{\mu} (b) (R) = \set{ (\mathcal E, \beta) \in \mathbb G \mathrm{r} _G (R) \mid  \mathrm{Inv}_x (\beta ^{-1} b\sigma (\beta)) = \mu, ~\forall x \in \mathrm{spec} R}. $$
\begin{lem}\label{lem:finite field points}
	For any $s \in s_0 \mathbb N$, we have
	$$\mathbb X_{\mu } (b) (k_s) = \set{g \in G(F_s)/G (\Ok_s)\mid   g^{-1} b \sigma (g) \in G(\Ok_s) \mu (\pi_F) G(\Ok_s)}.$$
\end{lem}
\begin{proof}
	We only need to show that $\mathbb G \mathrm{r} _G (k_s) = G(F_s)/G(\Ok_s)$. For this it suffices to show that any $G_{\Ok_s}$-torsor over $\Ok_s$ is trivial (cf.~the proof of \cite[Lemma 1.3]{Zhu}). By smoothness this reduces to the Lang--Steinberg theorem, namely that any $G_{k_s}$-torsor over the finite field $k_s$ is trivial.
\end{proof}
\begin{lem}\label{lem:rational action}
	The action of $J_b(F)$ on $X_{\mu} (b)$ descends to a natural action on $\mathbb X_{\mu} (b)$ via $k_{s_0}$-automorphisms.
\end{lem}
\begin{proof}
	By Lemma \ref{lem:two centralizers}, $J_b(F) = G(F_{s_0})_{b\sigma}$. The group $G(F_{s_0}) _{b\sigma}$ naturally acts on $\mathbb X_{\mu} (b) (R)$ by acting on the trivializations $\beta$, for each perfect $k_{s_0}$-algebra $R$.
\end{proof}
\begin{lem}\label{lem:rational comp}
	Up to enlarging $s_0$, all the irreducible components of $X_{\mu} (b)$ are defined over $k_{s_0}$, i.e., they come from base change of irreducible components of $\mathbb X_{\mu } (b)$. \end{lem}
\begin{proof}This follows from  Lemma \ref{lem:finiteness} and Lemma \ref{lem:rational action}.
\end{proof}
\subsection{Twisted orbital integrals and point counting}\label{subsec:twisted orb int} We fix $s_0 \in \mathbb N$ to be divisible enough so as to satisfy all the conclusions in \S \ref{subsec:decent case}. In particular $G$ is split over $F_{s_0}$ and the conclusion of Lemma \ref{lem:rational comp} holds. Let $s\in s_0\mathbb N$.

For any $\C$-valued function $f\in C^{\infty}_c(G(F_s))$, define the twisted orbital integral
\begin{align}\label{eq:defn of twisted orb int}
TO _{b} (f) : =  \int_{G(F_s) _{b\sigma} \backslash G(F_{s}) } f ( g^{-1}  b \sigma(g)) d g,
\end{align} where $G(F_s) _{b\sigma}$ is equipped with an arbitrary Haar measure, and $G(F_s)$ is equipped with the Haar measure giving volume $1$ to $G(\Ok_s)$. The general convergence of $TO_b (f)$ follows from the result of Ranga Rao \cite{Rao}.
However, in our specific case the convergence could be proved more easily. In fact, the decency equation (\ref{eq:decent}) implies that $b\Theta$ is a semi-simple element of $G_s \rtimes \lprod{\Theta}$, from which it follows that the twisted orbit is closed in $G(F_s)$. The convergence of $TO_{b}(f)$ then follows from the closedness of the twisted orbit, cf.~\cite[p.~266]{clozelFL}.

\begin{defn}\label{defn:test function}
	Let $f_{\mu, s} \in C^{\infty}_c( G(F_s))$ be the characteristic function of $G(\Ok_s) \mu(\pi_F) G(\Ok_s). $
\end{defn}

In the following we study the relationship between $TO_b(f_{\mu,s})$ and point counting on $\mathbb X_{\mu} (b)$.

\begin{lem}\label{lem:topology of irred comp}
	Each irreducible component $Z$ of $X_{\mu} (b)$ is quasi-compact, and is isomorphic to the perfection of a quasi-projective variety over $k$. Moreover, $Z$ has non-empty intersection only with finitely many other irreducible components of
	$X_\mu (b)$.
\end{lem}
\begin{proof}
	Since $X_{\mu} (b)$ is a perfect scheme by Theorem \ref{thm:dim}, the generic point $\eta$ of $Z$ and its residue field $k(\eta)$ make sense. Moreover $k(\eta)$ is a perfect field containing $k$. Let $(\mathcal E, \beta) \in X_{\mu} (b) (k(\eta)) \subset \Gr _G(k(\eta))$ correspond to $\eta$, and define $\lambda : = \mathrm{Inv} (\beta) \in X_*(T) ^+.$ Since $\set{\eta}$ is dense in $Z$, it follows from \cite[Lemma 1.22]{Zhu} that $Z$ is contained in $\Gr _{G, \leq \lambda}$, the Schubert variety inside $\Gr _G$ associated to $\lambda$. On the other hand, it follows from \cite[\S 1.4.1, Lemma 1.22]{Zhu} and \cite[Theorem 8.3]{BS} that $\Gr _{G, \leq \lambda}$ is the perfection of a projective variety over $k$. Since $Z$ is closed in $X_{\mu} (b)$ and $X_{\mu} (b)$ is locally closed in $\Gr _G$, we conclude that $Z$ is locally closed in $\Gr _{G , \leq \lambda}$, and hence $Z$ is quasi-compact and isomorphic to the perfection of a quasi-projective variety over $k$.
	
	Since $X_{\mu} (b)$ is locally perfectly of finite type (Theorem \ref{thm:dim}), each point in $X_{\mu} (b) $ has an open neighborhood that intersects with only finitely many irreducible components of $X_{\mu} (b)$. Since $Z$ is quasi-compact, it also intersects with only finitely many irreducible components of $X_{\mu} (b)$.
\end{proof}
For each $x \in J_b (F) \backslash \mathbb X_{\mu} (b) (k_s)$, we pick a representative $\tilde x \in \mathbb X_{\mu} (b ) (k_s)$ and consider the volume of $\Stab_{\tilde x} J_b(F)$, with respect to the chosen Haar measure on $J_b(F)  = G(F_s)_{b\sigma}$ (cf.~Lemma \ref{lem:two centralizers}). This volume is independent of the choice of $\tilde x$, and we shall denote it by $\vol_x$.
\begin{lem}\label{lem:TO}
	The set $J_b(F) \backslash \mathbb X_{\mu} (b) (k_s)$ is finite.	For all $ \tilde x \in \mathbb X_{\mu} (b) (k_s)$, the stabilizer $\Stab_{\tilde x} J_b(F)$ in $J_b(F)$ is a compact open subgroup of $J_b(F)$.  We have
	$$
	TO_b (f_{\mu,s}) = \sum _{x \in J_b (F) \backslash \mathbb X_{\mu} (b) (k_s)} \vol_x ^{-1} .
	$$
\end{lem}
\begin{proof} Let $C = \set{ g \in G(F_s)_{b\sigma} \backslash G(F_s) \mid  g^{-1} b \sigma(g) \in G(\Ok_s) \mu (\pi_F) G(\Ok_s)}$.
	By the discussion below (\ref{eq:defn of twisted orb int}), we know that $C$ is a compact subset of $G(F_s)_{b\sigma} \backslash G(F_s),$ as $C$ is homeomorphic to the intersection of the compact set $ G(\Ok_s) \mu (\pi_F) G(\Ok_s)$ with the closed twisted orbit of $b\sigma$.
	The group $ G(\Ok_s)$ acts on $C$ by right multiplication, and all the orbits under this action are open. Since $C$ is compact, the number of orbits is finite. On the other hand, by Lemma \ref{lem:finite field points} and Lemma \ref{lem:rational action}, these orbits are in one-to-one correspondence with $J_b(F) \backslash \mathbb X_{\mu} (b) (k_s)$. In particular $J_b(F) \backslash \mathbb X_{\mu} (b) (k_s)$ is finite. \ignore{\footnote{Alternatively, one could also show the finiteness using the result of Rapoport--Zink \cite[Theorem 1.4]{RZfinite} as interpreted in \cite[Lemma 1.3]{HV}.}}
	
	Now for each $x \in J_b(F) \backslash \mathbb X_{\mu} (b) (k_s) $, we denote the corresponding $ G(\Ok_s)$-orbit in $C$ by $ C_x.$ Let $\tilde x \in \mathbb X_{\mu} (b) (k_s) $ be a representative for $x$, and fix $r \in G(F_s)$ lifting  $\tilde x$ in the sense of Lemma \ref{lem:finite field points}. It follows from the definition of quotient measure that the volume of $C_x$ is the inverse of the volume of the compact open subgroup
	$G(F_s)_{b\sigma} \cap rG(\Ok_s) r^{-1}$ of $G(F_s)_{b\sigma}$. Since $TO_b(f_{\mu,s})$ is nothing but the volume of $C$, we are left to check that $G(F_s)_{b\sigma} \cap rG(\Ok_s) r^{-1} = \Stab_{\tilde x} J_b(F)$. But this follows from Lemma \ref{lem:two centralizers}.
\end{proof}

\begin{prop}\label{prop:point counting} Let $d = \dim X_{\mu} (b)$. For $s \in \mathbb N$ divisible by $s_0$, we have
	$$TO_b  (f_{\mu ,s}) = \sum _{ Z\in J_b(F) \backslash \Sigma ^{\topp} (X_{\mu} (b))} \vol (Z)^{-1} \abs{k_s}^{d} + o(\abs{k_s}^{d}), \quad s \gg 0 .$$
	Here $Z$ runs through a set of representatives of the $J_b(F)$-orbits in $\Sigma ^{\topp} (X_{\mu} (b))$.
\end{prop}
\begin{proof}
	In view of Lemma \ref{lem:finiteness}, we let $\set{Z_1,\cdots, Z_M}$ be a set of representatives of the $J_b(F)$-orbits in $\Sigma (X_{\mu}(b))$. For each $1\leq i \leq M$, we write $\mathbb J_i$ for $\Stab _{Z_i} (J_b(F))$. For each $y \in \mathbb X_{\mu}(b)(k_s)$, we write $\mathbb J_y$ for $\Stab_y (J_b(F))$.
	
	For each $1\leq i \leq M$, we set
	\begin{align*} U_i : = Z_i - \bigcup _{1\leq j <i, \gamma \in J_b(F)} \gamma Z_j - \bigcup _{\gamma \in J_b(F), \gamma Z_i \neq Z_i} \gamma Z_i.
	\end{align*}
	By Lemma \ref{lem:topology of irred comp}, $U_i$ is open dense in $Z_i$.
	By Lemma \ref{lem:rational comp}, we know that $U_i,  V_i$ are the base change of locally closed $k_{s_0}$-subschemes $\mathbb Z_i, \mathbb U_i$ of $\mathbb X_{\mu} (b)$ respectively, where $\mathbb Z_i, \mathbb U_i$ are perfections of quasi-projective varieties over $k_{s_0}$. Moreover, $\mathbb Z_i$ and $\mathbb U_i$ are irreducible.

	We denote the natural maps
	\begin{align*}
	\coprod_{1\leq i \leq M } \mathbb Z_i (k_s) \To J_b(F) \backslash \mathbb X_{\mu} (b) (k_s)
	\end{align*} and
	\begin{align*}
	\coprod_{1\leq i \leq M } \mathbb U_i (k_s) \To J_b(F) \backslash \mathbb X_{\mu} (b) (k_s)
	\end{align*} by $\Pi$ and $\pi$, respectively. Here we take the disjoint union of the $\mathbb Z_i (k_s)$ for $1\leq i \leq M$ even though they may have non-trivial intersections in $\mathbb X_{\mu}(b) (k_s)$. Then $\Pi$ is a surjection (between finite sets).
	
	Fix an element $y \in \mathbb Z_i(k_s)$, and let $x = \Pi(y)$. We have
	\begin{align}
	\vol_{x} \cdot \abs{\Pi^{-1} (x)} &\geq \vol_{x} \cdot \abs{\mathbb J_i y}  = \vol(\mathbb J_y) \cdot \abs{\mathbb J_i y}  \nonumber\\
	&\geq \vol(\Stab_{y} (\mathbb J_{i})) \cdot  \abs{\mathbb J_i y } = \vol(\mathbb J_{i}),\label{est1}
	\end{align} where the last equality follows from the orbit-stabilizer relation.
	
	Now fix $y \in \mathbb U_i(k_s)$ and let $x = \pi(y)$. We observe that $\pi^{-1}(x) = \mathbb J_i y$, and that $\mathbb J_y = \Stab_y (\mathbb J_i)$. We then have
	\begin{align} \label{est2} &
	\vol_x \cdot \abs{\pi^{-1} (x)}  = \vol(\mathbb J_y) \cdot \abs{\mathbb J_i y}  =  \vol(\Stab_{y} (\mathbb J_{i})) \cdot  \abs{\mathbb J_i y } = \vol(\mathbb J_{i}). \end{align}
	
	We now apply (\ref{est1}) and (\ref{est2}) to estimate $TO_b(f_{\mu,s})$. We have
	\begin{align} \nonumber
	TO _b (f_{\mu,s}) & =  \sum _{x \in J_b(F) \backslash \mathbb X_{\mu} (b) (k_s)} \vol_x ^{-1 } & \text{(by Lemma \ref{lem:TO})} \\ \nonumber & = \sum _{i=1} ^M \sum _{y \in \mathbb Z_i (k_s)} \vol _{\Pi (y)} ^{-1} \cdot \abs{\Pi^{-1} (\Pi (y))} ^{-1}  & \text{(by the surjectivity of $\Pi$)}\\
	\label{eq:TO and pt}
	& \leq    \sum _{i=1} ^M \vol (\mathbb J_i) ^{-1} \abs{\mathbb Z_i (k_s)} & \text{(by (\ref{est1}))}.
	\end{align}
	Similarly, we have
	\begin{align} \nonumber TO _b (f_{\mu,s}) & =  \sum _{x \in J_b(F) \backslash \mathbb X_{\mu} (b) (k_s)} \vol_x ^{-1 } \\ \nonumber & \geq \sum _{x \in \text{ image of }\pi} \vol_x ^{-1 }  \\ \nonumber & =  \sum _{i=1} ^M \sum _{y \in \mathbb U_i (k_s)} \vol _{\pi (y)} ^{-1} \cdot \abs{\pi^{-1} (\pi (y))} ^{-1}  \\ \label{eq:TO and pt'}
	& = \sum_{i=1}^M \vol(\mathbb J_i)^{-1}\abs{\mathbb U_i (k_s)}
	& \text{(by (\ref{est2}))}.
	\end{align}
	
	Now let $\mathcal U_i$ be a quasi-projective variety whose perfection is $\mathbb U_i$. Then $\mathcal U_i$ is irreducible, and $\mathcal U_i (k_s) = \mathbb U_i (k_s)$. By the Lang--Weil bound (see \cite{langweil}) applied to $\mathcal U_i$, we know that \begin{align}\label{eq:LW for d_i}
	\abs{\mathbb U_i (k_s)} = \abs{k_s} ^{\dim Z_i} + o (\abs{k_s} ^{\dim Z_i}), \quad s \gg 0 .
	\end{align} Similarly we have
	\begin{align}\label{eq:LW2}
	\abs{\mathbb Z_i (k_s)} = \abs{k_s} ^{\dim Z_i} + o (\abs{k_s} ^{\dim Z_i}), \quad s \gg 0  .
	\end{align} The proposition follows from (\ref{eq:TO and pt}) (\ref{eq:TO and pt'}) (\ref{eq:LW for d_i}) (\ref{eq:LW2}).
\end{proof}

\subsection{Applying the Base Change Fundamental Lemma in the basic case}\label{subsec:BCFL in basic case}
Recall that we assumed that $[b] \in B(G, \mu)$ and $b$ is $s_0$-decent. We now assume in addition that $b$ is basic.

For $s\in \mathbb N$, recall from \cite[\S 5]{kottwitzrational} that the \emph{$s$-th norm map} is a map
\begin{align*}
& \mathfrak N_s: \set{\text{$\sigma$-conjugacy classes in $G(F_s)$}}\\
&\quad \To \set{\text{stable conjugacy classes in $G(F)$}}.
\end{align*}
By \cite[Proposition 5.7]{kottwitzrational}, two $\sigma$-conjugacy classes in $G(F_s)$ are in the same fiber of $\mathfrak N_s$ precisely when they are \emph{stably $\sigma$-conjugate}, a notion that is defined in \cite[\S 5]{kottwitzrational}.

\begin{lem}\label{lem:norm is singleton} Let $s\in s_0 \mathbb N$.
	Then $\mathfrak N_s (b)$, as a stable conjugacy class in $G(F)$, consists of the single element $(s\nu_b ) (\pi_F)$. Moreover, the cocharacter $s\nu_b:\mathbb G_m \to G$ is defined over $F$.
\end{lem}
\begin{proof}
	By \cite[Corollary 5.3]{kottwitzrational}, any element in $\mathfrak N_s (b)$ is $G(\overline F)$-conjugate to $$b \sigma (b) \cdots \sigma ^{s-1} (b) \in G(F_s), $$ which is equal to $(s\nu_b)(\pi_F)$ since $b$ is $s$-decent. Now since $(s\nu_b)(\pi_F)$ is central, we know that $\mathfrak N_s (b) = \set{ (s\nu_b) (\pi_F)} $ and that $(s\nu_b) (\pi_F) \in G(F)$. It follows from the last statement that $s\nu_b$ is defined over $F$.
\end{proof}

\begin{lem}\label{lem:conseq of basic} Let $s\in s_0 \mathbb N$.
	Let $b'\in G(F_s)$ be an element in the stable $\sigma$-conjugacy class of $b$. Then $\nu_ {b'} = \nu_b$, and $b'$ is $s$-decent. In particular $b'$ is basic. Moreover, if $[b']\in B(G,\mu)$, then $b'$ is $\sigma$-conjugate to $b$ in $G(F_s)$.
\end{lem}
\begin{proof} By hypothesis we have $\mathfrak N_s (b) = \mathfrak N_s (b')$.
	By Lemma \ref{lem:norm is singleton} applied to $b$, we know that the $\mathfrak N_s (b)$ consists of the single central element $(s\nu_b) (\pi_F)\in G(F)$. On the other hand any element of $\mathfrak N_s (b')$ should be $G(\overline F)$-conjugate to $ b' \sigma (b') \cdots \ \sigma ^{s-1} (b')$ (by \cite[Corollary 5.3]{kottwitzrational}). Therefore
	$$ b' \sigma (b') \cdots \ \sigma ^{s-1} (b') = (s\nu_b) (\pi_F) .$$
	By the characterization of $\nu_{b'}$ (see \cite[\S 4.3]{kottwitzisocrystal}), the above equality implies that $\nu_{b'} = \nu _b$ and that $b'$ is $s$-decent. The first part of the lemma is proved.
	
	Now we assume $[b'] \in B(G,\mu)$. Since $ B(G,\mu)$ contains a unique basic class, we have $[b'] = [b]$. Finally, by \cite[Corollary 1.10]{RZ96}, we know that $b$ and $b'$ must be $\sigma$-conjugate in $G(F_s)$, since they are both $s$-decent and represent the same class in $B(G)$.
\end{proof}

Let $s\in s_0 \mathbb N$. We now consider stable twisted orbital integrals along $b$. By our assumption that $b$ is $s$-decent and basic, we know that $$b \sigma (b) \cdots \sigma ^{s-1} (b) = (s\nu_b)(\pi_F)$$ is a central element of $G(F_s)$, and is in fact an element of $G(F)$ by Lemma \ref{lem:norm is singleton}. In particular, this element is semi-simple, and the centralizer of this element (namely $G$) is connected. Therefore, with the terminology of \cite{kottwitzrational}, an element $b' \in G(F_s)$ is stably $\sigma$-conjugate to $b$ if and only if it is $\overline F$-$\sigma$-conjugate to $b$. This observation justifies our definition of the stable twisted orbital integral in the following, cf.~\cite[\S 5.1]{hainesBCFL}.

For any $\C$-valued function $f \in C^{\infty} _c (G(F_s))$, we let $STO_b (f)$ be the stable twisted orbital integral
$$STO_b (f):   = \sum _{ b'} e( G_{s,b'\Theta} ) TO_{b'} (f),$$
where the summation is over the set of $\sigma$-conjugacy classes $b'$ in $G(F_s)$ that are stably $\sigma$-conjugate to $b$, and $e(\cdot)$ denotes the Kottwitz sign. Here each $TO_{b'} $ is defined using the Haar measure on $G(F_s)$ giving volume $1$ to $G(\Ok_s)$, and the Haar measure on $G(F_s) _{b' \sigma} = G_{s, b' \Theta} (F)$ that is transferred from the fixed Haar measure on $G(F_s) _{b \sigma} = G_{s, b \Theta} (F)$.

\begin{defn}\label{defn:Hecke alg}
	We denote by $\mathcal H_s$ the unramified Hecke algebra consisting of $G(\Ok_s)$-bi-invariant functions in $C^{\infty}_c(G(F_s)$, and denote by $\BC_s$ the base change map $\mathcal H_s \to \mathcal H_1$.
\end{defn}
\begin{defn}\label{defn:gamma_s}For $s\in s_0 \mathbb N$, we write $\gamma_s$ for $(s\nu_b) (\pi_F)$, and write $\gamma_0$ for $\gamma _{s_0}$. Thus $\gamma_{0}$ belongs to $G(F)$ (see Lemma \ref{lem:norm is singleton}) and $\gamma _s = \gamma _0 ^{s/s_0}$.
\end{defn}
\begin{prop}\label{prop:BCFL}
	Assume $s\in s_0 \mathbb N$. For any $f \in \mathcal H_s$, we have
	$$ STO_b (f)  =  \vol (G(\Ok_F)) ^{-1} (\BC_s f) (\gamma_s),$$ where $\vol (G(\Ok_F))$ is defined in terms of the Haar measure on $G(F)$ transferred from the fixed Haar measure on $G_{s,b\Theta} (F) $, for the inner form $G_{s, b\Theta}$ of $G$.
\end{prop}
\begin{proof} By Lemma \ref{lem:norm is singleton}, $\mathfrak N_s (b)$ consists of the single central semi-simple element $\gamma_s \in G(F)$.
	By the Base Change Fundamental Lemma proved by Clozel \cite[Theorem 7.1]{clozelFL} and Labesse \cite{labFL}, we know that $STO_b(f)$ is equal to the stable orbital integral of $\BC_s f$ at $\mathfrak N_s (b)$. The latter degenerates to
	$$ e(G_{\gamma_s}) \cdot  \frac {\mu_1} {\mu_2 }  \cdot (\BC_s f )(\gamma_s g)$$ since $\gamma_s$ is central.
	Here $\mu_1$ denotes the Haar measure on $G(F)$ giving volume $1$ to $G(\Ok_F)$, and $\mu_2$ denotes the Haar measure on $G_{\gamma_s} (F) = G(F)$ transferred from $G_{s,b\Theta}(F)$. The notation $\frac{\mu_1}{\mu_2}$ denotes the ratio between these two Haar measures on the same group $G(F)$. Obviously this ratio is equal to $\vol (G(\Ok_F)) ^{-1}$ as in the proposition. Finally, since $G_{\gamma_s} = G$ is quasi-split, we have $e(G_{\gamma_s}) =1$.
\end{proof}

\begin{lem}\label{TO is stable}
	For $s$ divisible by $s_0$, we have $$STO_b (f_{\mu,s}) = e(J_b)TO_b(f_{\mu ,s}).$$
\end{lem}

\begin{proof} Firstly, by Lemma \ref{lem:two centralizers} we have $G_{s,b\Theta} = J_b$. We need to check that $TO_{b'} (f_{\mu ,s}) =0$, for any $b' \in G(F_s)$ that is stably $\sigma$-conjugate to $b$ but not $\sigma $-conjugate to $b$ in $G(F_s)$. Assume the contrary. Then there exists $g\in G(F_s)$ such that $f_{\mu,s} (g ^{-1} b'\sigma (g) ) \neq 0, $ from which $ g ^{-1} b'\sigma (g) \in G(\Ok_s) \mu (\pi_F) G(\Ok_s)$. Hence $\kappa ([b'])  = \mu ^{\natural}$ by Theorem \ref{KR-Gas}. But this contradicts Lemma \ref{lem:conseq of basic}.
\end{proof}
\begin{cor}\label{irred-orb}Keep the notation in Proposition \ref{prop:point counting}.
	For $s \gg 0$, we have
	\begin{align*}
	e(J_b) &  \vol (G(\Ok_F)) ^{-1} (\BC_s f_{\mu ,s}) (\gamma_s) \\ & = \sum _{ Z\in J_b(F) \backslash \Sigma ^{\topp} (X_{\mu}(b))} \vol (Z)^{-1} \abs{k_s}^{d} + o(\abs{k_s}^{d}) .
	\end{align*}
\end{cor}
\begin{proof}This follows from Proposition \ref{prop:point counting}, Proposition \ref{prop:BCFL}, and Lemma \ref{TO is stable}.
\end{proof}

\section{Matrix coefficients for the Satake transform}
\label{sec:Sat}
\subsection{General definitions and facts}\label{subsec:general facts}
In this subsection we expose general facts concerning the Satake isomorphism, for unramified reductive groups over $F$. The aim is to give an interpretation of the coefficients for the matrix of the inverse Satake isomorphism in terms of a $\mathbf q$-analogue of Kostant's partition function. This is well known by the work of Kato \cite{Kato} in the case when $G$ is split; we will need the case of non-split $G$. Our main reference is \cite[\S 1]{CassCelyHales}.

Let $G$ be an unramified reductive group over $F$. At this moment it is not necessary to fix a reductive model over $\Ok_F$ of $G$. Inside $G$ we fix a Borel pair, namely a Borel subgroup $B$ and a maximal torus $ T \subset B$, both defined over $F$. In particular, $T$ is a minimal Levi, and is split over $F^{\un}$.

We denote by $$\mathrm {BRD}(B,T) = (X^*(T), \Phi\supset \Delta, X_*(T), \Phi ^{\vee} \supset  \Delta ^{\vee})$$ the based root datum associated to $(B,T)$. This based root datum has an automorphism $\theta$ induced by the Frobenius $\sigma \in \Gal (F^{\un}/F)$. Let $d = d_\theta < \infty$ be the order of $\theta$.

Fix an $F$-pinning $(B,T, \mathbb X_+)$ of $G$. Since the Galois action on $\mathrm{BRD}(B,T)$ factors through the cyclic group generated by $\theta$, we know that $\theta$ is a Galois-equivariant automorphism of $\mathrm{BRD}(B,T)$, and so it lifts uniquely to an $F$-automorphism of $G$ preserving $(B, T ,\mathbb X_+)$. We denote this $F$-automorphism of $G$ still by $\theta$.

Let $A$ be the maximal split sub-torus of $T$. We have\footnote{In \cite[\S 1.1]{CassCelyHales}, it is stated that $X^*(A) = X^*(T)/(1-\theta) X^*(T)$, which is not true in general.} $$X_*(A) = X_*(T) ^{\theta}, \quad X^*(A) = [X^*(T)/(1- \theta) X^*(T)]_{\mathrm{free}}. $$

Let $_F \Phi \subset X^*(A) $ be the image of $\Phi \subset X^*(T)$. The triple $$(X^*(A),\leftidx _F\Phi, X_*(A))$$ naturally extends to a (possibly non-reduced) root datum
$$(X^*(A), \leftidx _F\Phi , X_*(A), \leftidx _F\Phi^{\vee}), $$ see for instance \cite[Theorem 15.3.8]{springer}.
Elements of $_F\Phi$ are by definition $\theta$-orbits in $\Phi$. For $\alpha \in \Phi$, we write $[\alpha]$ for its $\theta$-orbit. The $\theta$-orbits in $\Delta$ give rise to a set of simple roots in $_F\Phi$, which we denote by $_F \Delta$.
As usual, we denote the structural bijection $\leftidx _F\Phi \isom \leftidx _F\Phi ^{\vee}$ by $[\alpha ]\mapsto [\alpha] ^{\vee}$.

We let $\Phi^1 \subset \leftidx_F\Phi$ be the subset of indivisible elements, namely, those $[\alpha] \in \leftidx _F\Phi $ such that $\frac{1}{2} [\alpha] \notin \leftidx _F\Phi$. The image of $\Phi^1$ under the bijection $\leftidx _F\Phi \isom \leftidx _F\Phi^{\vee}$ is denoted by $\Phi^{1,\vee}$. The tuple $ (X^*(A), \Phi ^1, X_*(A), \Phi ^{1, \vee} )$ has the structure of a reduced root datum.
We note that $_F\Delta$ is also a set of simple roots in $\Phi ^1$. We henceforth also write $\Delta^1$ for $\leftidx_F\Delta$.
For the sets $\Phi, \leftidx_F \Phi, \Phi ^{1}, \Phi ^{\vee}, \leftidx_F \Phi^{\vee}, \Phi ^{1,\vee} $, we put a superscript $+$ to denote their respective subsets of positive elements.

As before we denote by $W_0$ the absolute Weyl group of $G$. We let $W^1\subset W_0$ be the subgroup of elements that commute with $\theta$. Then $W^1$ is a Coxeter group (see \cite[\S 1.1]{CassCelyHales}), and we denote by $\ell_1$ the length function on $W^1$.

The complex dual group $\widehat G$ of $G$ is a connected reductive group
over $\C$, equipped with a Borel pair $(\widehat{B}, \widehat{ T})$ and an
isomorphism $ \mathrm{BRD} (\widehat{B}, \widehat{ T}) \isom \mathrm{BRD}
(B,T)^{\vee}. $ In particular, we have canonical identifications
$X^*(\widehat T) \cong X_*(T),$ $X_*(\widehat T) \cong X^*(T)$, which we
think of as equalities. We fix a pinning $(\widehat{B}, \widehat{ T},
\widehat {\mathbb X}_+)$. The action of $\theta$ on $\mathrm{BRD}(B,T)$
translates to an action on $\mathrm{BRD}(\widehat B,\widehat T)$, and the
latter lifts to a unique automorphism $\hat \theta$ of $\widehat G$ that
preserves $(\widehat{B}, \widehat{ T}, \widehat {\mathbb X}_+)$. The L-group
$\lang G$ is defined as the semi-direct product $\widehat G \rtimes \langle
\hat \theta\rangle$, where $\langle \hat \theta \rangle$ denotes the cyclic
group of order $d$ generated by $\hat \theta$.

We denote the group $X^* (\widehat{T})^{\hat \theta} = X_*(T) ^{\theta} $ by $Y^*$. Let $\widehat{ A}$ be the quotient torus of $\widehat{ T}$ corresponding to $Y^*$. Then
$\widehat A$ is also identified with the dual torus of $A$. Define
\begin{align*}
P^+ & : = \set{\lambda \in Y^* \mid \langle \lambda, \alpha \rangle  \geq 0 , ~\forall \alpha \in \Delta } = \set{\lambda \in Y^* \mid \langle \lambda, [\alpha] \rangle  \geq 0 , ~\forall [\alpha] \in \Delta^1 } ,  \\
R^+ & : = \mbox{~the $\Z_{\geq 0}$-span of $\leftidx_F\Phi ^{\vee,+}$} \subset Y^*.
\end{align*}

The $\C$-vector space $\C[Y^*] ^{W^1}$ has a basis $\set{m_{\mu}}_{\mu \in P^+}$, where
\begin{align}\label{eq:m_mu}
m_{\mu} : = \sum _{ \lambda \in W^1\mu} e^{\lambda}.
\end{align}
Here $W^1 \mu$ denotes the orbit of $\mu$ under $W^1$.

\begin{defn}
	Let $\widehat {\mathfrak n}$ be the Lie algebra of the unipotent radical of $\widehat B$, equipped with the adjoint action by $\lang G$ (see \cite[\S 1.3.2]{CassCelyHales}). For each $\mu \in X^* (\widehat T)$, let $  \widehat {\mathfrak n}(\mu)$ denote the $\mu$-weight space in $\widehat {\mathfrak n}$ for the action of $\widehat T$. Let $w\in W^{1}$ and $\epsilon \in \set{\pm 1}$. We define a $\C[ X^*(\widehat T)]$-linear operator $E^{\epsilon w}$ on $\widehat {\mathfrak n} \otimes_{\C} \C[ X^*(\widehat T) ]$ by letting $E^{\epsilon w}$ act on $\widehat {\mathfrak n}(\mu)$ via the scalar $e^{\epsilon w \mu} \in \C [X^* (\widehat T)]$, for each $\mu \in X^* (\widehat T)$. We define
	\begin{align*}
	D(E^{\epsilon w} , \mathbf{q}) & : =  \det (1- \mathbf{q}  \cdot  \hat \theta \cdot E^{\epsilon w}  , ~\widehat {\mathfrak n} ) \in \C [X^*(\widehat T)] [\mathbf{q}] ,  \\
	P(E^{\epsilon w} , \mathbf{q}) & : =  D(E^{\epsilon w}, \mathbf{q})^{-1} \in \mathrm{Frac} \bigg( \C [X^*(\widehat T)] [\mathbf{q}] \bigg).
	\end{align*}
\end{defn}

\begin{defn}\label{defn:type of roots}
	Let $[\alpha] \in \Phi ^1 \subset \leftidx_F \Phi $. We say that $[\alpha]$ is of type I, if $2[\alpha] \notin \leftidx_F\Phi$. Otherwise we say that $[\alpha]$ is of type II. For $[\alpha] \in \Phi ^1$, we define
	$$ \mathtt{b} ([\alpha]): =  \begin{cases}
	\# [\alpha], &~ \mbox{if $[\alpha]$ is of type I} ,\\
	\frac{1}{2} \#[\alpha] &~ \mbox{if $[\alpha]$ is of type II},
	\end{cases}$$
	where $\#[\alpha]$ denotes the size of $[\alpha]$ viewed as a $\theta$-orbit in $\Phi$. Then $\mathtt b ([\alpha]) \in \Z_{\geq 1}$, see \cite[\S 1.1]{CassCelyHales}.
\end{defn}
\begin{defn}\label{defn:mathtt b}
	For any element $\beta = [\alpha]^{\vee} \in \Phi ^{1,\vee}$ (with $[\alpha] \in \Phi ^1$), we define $\mathtt b(\beta)$ to be $\mathtt b ([\alpha])$, and we say that $\beta$ is of type I or II if $[\alpha]$ is of type I or II. For any $\beta ' \in \leftidx_F \Phi ^{\vee}$, we define $\mathtt b (\beta')$ to be $\mathtt b (\beta)$, where $\beta$ is the element in $\Phi ^{1,\vee}$ that is homothetic to $\beta'$.
\end{defn}

\begin{defn}\label{defn:d_beta}
	For $\beta \in \Phi ^{1,\vee}$, we define $d_{\beta} (\mathbf{q}) \in \C [Y^*][\mathbf{q}]$ as follows:
	$$ d_{\beta} (\mathbf{q}): = \begin{cases}
	1- \mathbf{q} ^{\mathtt b(\beta)} e ^{\beta}, &~ \mbox{if $\beta$ is of type I}, \\
	(1- \mathbf{q} ^{2 \mathtt b (\beta)} e ^{\beta/2}) ( 1+ \mathbf{q} ^{\mathtt b (\beta)} e ^{\beta/2} ),& ~ \mbox{if $\beta$ is of type II}.
	\end{cases}$$
	Here, when $\beta $ is of type II, $\beta/2$ is always an element of $ \leftidx_F \Phi ^{\vee}$ and in particular an element of $Y^*$, see \cite[\S 1.1]{CassCelyHales} or \cite[\S 1.3]{KS99}.
\end{defn}

\begin{lem}\label{lem:factorization} For $\epsilon \in \set{\pm 1}$, we have
	$$ D (E^{\epsilon }, \mathbf{q}) = \prod _{\epsilon \beta \in \Phi ^{1,\vee, +}} d_{\beta} (\mathbf{q}), \qquad P (E^{\epsilon }, \mathbf{q}) = \prod _{\epsilon \beta \in \Phi ^{1,\vee, +}} d_{\beta} (\mathbf{q}) ^{-1}.$$
\end{lem}
\begin{proof}
	The case $\epsilon = 1$ is \cite[Lemma 1.3.7]{CassCelyHales}. The case $\epsilon = -1$ is proved in the same way, by switching the roles of positive elements and negative elements in $\Phi^{1,\vee}$.
\end{proof}

\begin{defn}\label{defn:P} For each $\lambda \in Y^*$, we define $\PP (\lambda ,\mathbf{q}) \in \C [\mathbf{q}]$ as follows. In view of Definition \ref{defn:d_beta} and Lemma \ref{lem:factorization}, we have an expansion
	\begin{align}\label{eq:exp of P}
	P(E^{-1}, \mathbf{q}) = \sum _{\lambda \in R^+} \PP(\lambda, \mathbf{q}) e^{-\lambda} ,
	\end{align}
	with each $\PP(\lambda ,\mathbf{q}) \in \C [\mathbf{q}]$. We set $\PP(\lambda, \mathbf{q}) := 0$ for all $\lambda \in Y^* - R^+$.
\end{defn}
\begin{cor}\label{cor:triv bd}
	For $\lambda \in R^+ - \set{0}$, the constant term $\mathcal P(\lambda, 0)$ of $\mathcal P(\lambda ,\mathbf{q})$ is $0$.
\end{cor}
\begin{proof}
	This immediately follows from Lemma \ref{lem:factorization} and Definition \ref{defn:P}.
\end{proof}

\begin{defn}\label{defn:J}
	Let $\rho^{\vee} \in X_*(T) \otimes _{\Z} \frac{1}{2}\Z $ be the half sum of elements in $\Phi^{\vee,+}$, and let $\rho  \in   X^*(T) \otimes _{\Z} \frac{1}{2}\Z $ be the half sum of elements in $\Phi^+$. Then $\rho^{\vee}$ in fact lies in $ Y^* \otimes _{\Z} \frac{1}{2}\Z $, and is equal to the half sum of elements in $\Phi^{1,\vee,+}$, see \cite[\S 1.2]{CassCelyHales}.
	For $w \in W^1$ and $\mu \in Y^*$, we let $$ w\bullet \mu : = w (\mu +\rho^{\vee}) - \rho^{\vee}  \in Y^*. $$ We also write $w \bullet (\cdot)$ for the induced action of $w$ on $\C[Y^*]$. Define the operator $$J: \C [Y^*] \To \C [Y^*], \quad f\longmapsto \sum _{w \in W^1 } (-1)^{\ell_1 (w)} w \bullet f.
	$$
\end{defn}
\begin{defn}\label{defn:tau}
	For $\lambda \in P^+$ and for a formal variable $\mathbf{q}$, we define
	\begin{align}\label{eq:defn of tau}
	\tau _{\lambda}(\mathbf{q}) : = J(e^{\lambda}) P(E^{-1} , \mathbf{q}) = \sum _{\mu \in R^+} J(e^{\lambda}) \PP (\mu , \mathbf{q}) e^{-\mu} \in  \C [ Y^*] [[\mathbf{q}]] .
	\end{align} \end{defn}

\begin{defn}
	For any $\lambda \in P^+ \subset Y^* =  X^*(\widehat{ T}) ^{\hat \theta}$, we define $V_{\lambda}$ to be the irreducible representation of $\widehat G$ of highest weight $\lambda$.
\end{defn}

\begin{thm}[Weyl character formula, {\cite[Theorem 1.4.1]{CassCelyHales}}] Let $\lambda \in P^+$. Then $\tau _ {\lambda} (1) \in \C [ Y^*] ^{W^1} $. The character of $V_{\lambda}$, as a function on $\widehat T$, descends to the function on $\widehat{ A}$ given by $\tau _ {\lambda} (1) $. \qed
\end{thm}
\begin{defn}\label{defn:tau_mu}
	For any $\lambda \in P^+$, we simply write $\tau_{\lambda}$ for the element $\tau_{\lambda}(1) \in \C [ Y^*]^{W^1}$.
\end{defn}

\subsection{Matrix coefficients} We now fix a reductive model of $G$ over $\Ok_F$ as in \S\ref{subsec:basic notations}. As before we denote by $\mathcal H_1$ the spherical Hecke algebra $\mathcal H(G(F)// G(\Ok_{F}) )$. For each $\mu \in X_*(A)$, we let $f_{\mu} \in \mathcal H_1$ be the characteristic function of $G(\Ok_F) \mu (\pi_F) G(\Ok_F)$. Then the $\C$-vector space $\mathcal H_1$ has a basis given by $f_{\mu}$, for $\mu \in P^+ \subset X_*(A)$.

Recall that the \emph{Satake isomorphism} is a canonical $\C$-algebra isomorphism $$\mathrm {Sat}: \mathcal H_1 \isom \C [Y^*] ^{W^1},$$ see for instance \cite[\S 1.5]{CassCelyHales}. In the following, we simply write $f_{\mu}$ for $\mathrm {Sat}(f_{\mu})$, which shall cause no confusion. At this point we have introduced three bases of the $\C$-vector space $\C [Y^*] ^{W^1}$, namely $\set{m_{\mu}}, \set{\tau _{\mu}}, \set{f_{\mu}}$, all indexed by $\mu \in P^+$ (see (\ref{eq:m_mu}) and Definition \ref{defn:tau_mu} for $m_{\mu}$ and $\tau_{\mu}$). We denote some of the transition matrices between these bases as follows:
$$ m_{\mu} = \sum _{\lambda} n_{\mu} ^{ \lambda} \tau _{\lambda} , \qquad \tau _{\mu} = \sum _{\lambda}  t_{\mu} ^{\lambda} f _{\lambda} , \qquad m_{\mu} = \sum _{\lambda} \mathfrak M _{\mu} ^{\lambda} f_{\lambda}  . $$
In the following we deduce a formula for $\mathfrak{M}_\mu^\lambda$ from known formulas for $n_\mu^\lambda$ and $t_\mu^\lambda$.

\begin{defn}
	As in \cite[\S 1.7]{CassCelyHales}, we have a partition of $Y^*$ into the following subsets: \begin{align*}
	Y_0 ^* & : = \set{\lambda \in Y^* \mid \exists w \in W^1, w\mbox{~is a reflection,~} w\bullet \lambda = \lambda }, \\
	Y_w^*  & : = \set{\lambda \in Y^* \mid w \bullet \lambda \in P^+}, ~ w\in W^1.
	\end{align*}
	For each $x \in W^1 \sqcup \set{0}$ we let $e_x : Y^* \to \set{0,1}$ be the characteristic function of $Y_x^*$.
\end{defn}

\begin{thm}[{van Leeuwen's formula, \cite[Lemma 1.7.4]{CassCelyHales}}]\label{thm:van Leeuwen}
	For $\mu , \lambda \in P^+$, we have
	$$
	n _{\mu} ^{\lambda} = \sum_{ w' \in W^1/W^1_{\mu}} \sum _{ w\in W^1}  (-1) ^{\ell_1 (w)} e_w (w'\mu) \delta ( w \bullet (w'\mu) , \lambda). $$ Here $\delta (\cdot , \cdot )$ is the Kronecker delta, and $W^1_{\mu}$ is the subgroup of $W^1$ generated by the reflections attached to those $[\alpha] \in \Delta^1$ such that $\langle \mu, [\alpha] \rangle =0$. \qed
\end{thm}
\begin{defn}\label{defn:K}
	For $\lambda , \lambda' \in P^+$, we define
	\begin{align}\label{eq:defn of pairing}
	K_{\lambda', \lambda} (\mathbf{q}):  = \sum _{w\in W^1} (-1) ^{\ell_1 (w)}  \PP(w\bullet \lambda '- \lambda,\mathbf{q}) .
	\end{align}
\end{defn}
\begin{rem}
	The notation $K_{\lambda',\lambda}$ in Definition \ref{defn:K} is compatible with \cite{Kato} when $G$ is split.
\end{rem}

\begin{thm}[{Kato--Lusztig formula, \cite[Theorem 1.9.1]{CassCelyHales}}]\label{thm:KL} For $\mu , \lambda \in P^+$, we have \[ \pushQED{\qed}
	t_{\mu} ^{\lambda} = K_{\mu,\lambda} (\abs{k_F}^{-1}) \abs{k_F} ^{- \langle \lambda, \rho  \rangle}. \qedhere \popQED \]
\end{thm}

\begin{cor}\label{cor:formula for M} Write $q$ for $\abs{k_F}$.
	For $\mu, \lambda \in P^+$, we have
	\begin{align*}
	\mathfrak M_{\mu} ^{\lambda}  =   q ^{- \langle \lambda, \rho  \rangle}   \sum_{ w' \in W^1/W^1_{\mu}} \sum _{w\in W^1}  (-1) ^{\ell_1 (w)} \left(1-e_0(w'\mu) \right)      \PP\bigg( w \bullet (w'\mu)   - \lambda ,q ^{-1}\bigg ).
	\end{align*}
\end{cor}
\begin{proof}We compute	
	\begin{align*}
	&	\mathfrak M_{\mu} ^{\lambda}  = \sum _{\lambda' \in P^+}  n_{\mu} ^{\lambda '} t_{\lambda'} ^{\lambda}\\ &  = \sum_{\substack{\lambda' \in P^+\\w' \in W^1/W^1_{\mu} \\ w '' \in W^1 }}   (-1) ^{\ell_1 (w'')} e_{w''} (w'\mu) \delta ( w'' \bullet (w'\mu) , \lambda')  K_{\lambda',\lambda} (q^{-1}) q ^{- \langle \lambda, \rho  \rangle} \\ & = \sum_{\substack{w' \in W^1/W^1_{\mu} \\ w''\in W^1  \\ w \in W^1}}   (-1) ^{\ell_1 (w'')} e_{w''} (w'\mu)   q ^{- \langle \lambda, \rho  \rangle} (-1) ^{\ell_1 (w)} \PP\bigg( (w w'') \bullet (w'\mu)   - \lambda ,q^{-1} \bigg ).   \end{align*} where the second equality is by Theorems \ref{thm:van Leeuwen}, \ref{thm:KL}, and the third equality is by (\ref{eq:defn of pairing}). Under the substitution $ w w'' \mapsto w$, the above is equal to
	\begin{align*}
	q ^{- \langle \lambda, \rho  \rangle}   \sum_{ w' \in W^1/W^1_{\mu}}  \sum _{ w'' \in W^1} \sum _{w\in W^1}  (-1) ^{\ell_1  (w)} e_{w''} (w'\mu)      \PP\bigg( w \bullet (w'\mu)   - \lambda ,q ^{-1}\bigg ).  \end{align*} Since $\sum _{w'' \in W^1} e_{w''} (\cdot) = 1- e_0(\cdot)$, the proof is finished.
\end{proof}

Motivated from Corollary \ref{cor:formula for M}, we make the following definition.
\begin{defn}\label{defn:frakM}
	For $\mu , \lambda \in P^+$ and a formal variable $\mathbf{q}^{-1/2}$, we define
	\begin{align*}
	\mathfrak M_{\mu} ^{\lambda}  (\mathbf{q}^{-1}) :  =   & \mathbf{q} ^{- \langle \lambda, \rho \rangle}   \sum_{ w' \in W^1/W^1_{\mu}} \sum _{w\in W^1}  (-1) ^{\ell_1 (w)} \left(1-e_0(w'\mu) \right)   \cdot \\ &    \cdot \PP\bigg( w \bullet (w'\mu)   - \lambda ,\mathbf{q} ^{-1}\bigg )  \in \C [Y^*][\mathbf{q}^{-1/2}].
	\end{align*} As a special case, we define
	\begin{align}\label{eq:defn of M}
	\mathfrak M_{\mu} ^{0} (\mathbf{q}^{-1}): = &   \sum_{ w' \in W^1/W^1_{\mu}} \sum _{w\in W^1}  (-1) ^{\ell_1 (w)} \left(1-e_0(w'\mu) \right)  \cdot \\ \nonumber &    \cdot     \PP\bigg( w \bullet (w'\mu)    ,\mathbf{q} ^{-1}\bigg )  \in \C [ Y^*] [\mathbf q ^{-1}].
	\end{align}
\end{defn}
\begin{lem}\label{lem:invariant under center}
	Let $\mu , \lambda \in P^+$. Let $\nu \in X_*(A \cap Z_G)$. Then $\mathfrak M_{\mu} ^{\lambda} (\mathbf{q}^{-1}) = \mathfrak M_{\mu -\nu} ^{\lambda -\nu} (\mathbf{q}^{-1}).$
\end{lem}
\begin{proof}
	In fact, we have $\langle \lambda, \rho \rangle = \langle \lambda -\nu , \rho \rangle$, $W^1_{\mu} = W^1 _{\mu -\nu}$, $e_0 (w' \mu) = e_0 (w' (\mu -\nu))$, and $w\bullet (w'\mu) -\lambda = w \bullet (w' (\mu -\nu)) - (\lambda - \nu)$,
	for all $w ' \in W^1 /W^1_{\mu} $ and $ w\in W^1$.
\end{proof}

\subsection{Interpretation in terms of Kostant partitions}\label{subsec:Kostant partitions} In certain cases the polynomial $\mathcal P(\lambda ,\mathbf{q}) \in \C [\mathbf{q}]$ in Definition \ref{defn:P} has a concrete description as a $\mathbf q$-analogue of Kostant's partition function, which we now explain. Let $\mathbb P$ be the set of all functions $ \leftidx_F \Phi ^{\vee, +} \to \Z_{\geq 0}.$ We shall typically denote an element of $\mathbb P$ by $\underline m$, and denote its value at any $\beta \in \leftidx_F\Phi^{\vee,+}$ by $m(\beta)$. For $\underline m \in \mathbb P$, we define
\begin{align*}
\Sigma (\underline m) & : = \sum _{\beta \in \leftidx_F \Phi ^{\vee, +}} m(\beta) \beta \in R^+ \subset  Y^*, \\ \abs{\underline m} & : = \sum _{\beta \in \leftidx_F \Phi ^{\vee, +}} m(\beta) \mathtt b (\beta)  \in \Z_{\geq 0}.
\end{align*}  Here $\mathtt b (\beta)$ is as in Definition \ref{defn:mathtt b}.

For all $\lambda \in Y^*$, we define
$\mathbb P(\lambda)$ to be the set of $\underline m \in  \mathbb P$ such that $\Sigma (\underline m) = \lambda.$ Thus $\mathbb P(\lambda)$ is empty unless $\lambda \in R^+$. Elements of $\mathbb P (\lambda)$ are called \emph{Kostant partitions} of $\lambda$. For any $L \in \Z_{\geq 0}$, we define
$ \mathbb P(\lambda)_L$ to be the set of $\underline m\in \mathbb P(\lambda)$ such that $\abs{\underline m} = L.$
For $\lambda \in Y^*$, we define
$$ \mathcal P_{\mathrm{Kos}} (\lambda, \mathbf{q}): = \sum _{\underline m \in \mathbb P(\lambda)} \mathbf{q} ^{\abs{\underline m}} \in \C[\mathbf{q}]. $$
This is known in the literature as the \emph{$\mathbf q$-analogue of Kostant's partition function}, at least when $G$ is split.

\begin{prop}\label{prop:interpretation}The following statements hold. \begin{enumerate}
		\item Assume $\leftidx_F \Phi = \Phi^1$. For all $\lambda \in Y^*$ we have
		$ \mathcal P (\lambda ,\mathbf{q}) = \mathcal P_{\mathrm{Kos}} (\lambda, \mathbf{q}).$
		\item In general, to each $\underline m \in \mathbb P$, we can attach a polynomial $\mathcal Q(\underline m, \mathbf{q}) \in \C [\mathbf{q}]$, with the following properties:
		\begin{enumerate}
			\item For all $0<x<1$, we have $\abs{\mathcal Q(\underline m, x) } \leq 1$.
			\item For any $\lambda \in Y^*$ we have $$ \mathcal P  (\lambda ,\mathbf{q})  = \sum _{\underline m \in \mathbb P(\lambda)} \mathcal Q(\underline m, \mathbf{q}) \mathbf q ^{\abs{\underline m}}.$$
		\end{enumerate}
	\end{enumerate}
\end{prop}
\begin{proof}
	Part (1) immediately follows from Definitions \ref{defn:d_beta}, \ref{defn:P}. For part (2), we note that if $\beta \in \Phi ^{1,\vee}$ is of type II, then $\beta': = \beta/2$ is an element of $\leftidx_F \Phi ^{\vee}$, and we have
	\begin{align*}
	d_{\beta} (\mathbf{q}) ^{-1} & = \left[\sum_{i=0}^{\infty} (\mathbf{q}^{2\mathtt b (\beta)} e ^{\beta/2})^i\right] \left[\sum _{i=0}^{\infty} (-\mathbf{q}^{\mathtt b (\beta)} e^{\beta/2})^i\right]  \\* &   = \left[\sum_{i=0}^{\infty} (\mathbf{q}^{2\mathtt b (\beta')} e ^{\beta'})^i\right] \left[\sum _{i=0}^{\infty} (-\mathbf{q}^{\mathtt b (\beta')} e^{\beta'})^i\right] \\
	& = \sum _{n=0} ^{\infty} \mathcal R_{\beta, n} (\mathbf{q}) (\mathbf{q}^{\mathtt b (\beta')} e^{\beta'})^n,
	\end{align*}with $$\mathcal R_{\beta, n} (\mathbf{q})  = \sum _{i=0}^n (-1) ^{n-i} \mathbf{q}^{i \mathtt b (\beta')}\in \C[\mathbf{q}], \quad \forall n \in \mathbb Z_{\geq 0}. $$  We observe that for all $0<x<1$ we have
	\begin{align}\label{eq:obs}
	\abs{\mathcal R_{\beta,n} (x) } \leq 1.
	\end{align}
	
	Now for each $\beta' \in \leftidx_F \Phi ^{\vee}$ and each $n \geq 0$, define
	$$ \mathcal Q_{\beta', n} (\mathbf q) : = \begin{cases}
	\mathcal R_{2\beta',n} (\mathbf q), &~ \mbox{if }\ 2\beta ' \in \Phi ^{1,\vee}, \\
	1, &~ \mbox{if }2\beta' \notin \Phi ^{1,\vee}.
	\end{cases}$$ We take $$\mathcal Q(\underline m, \mathbf q): = \prod _{\beta'\in \leftidx_F \Phi ^{\vee}}  \mathcal Q_{\beta', m(\beta')} (\mathbf q). $$ Then condition (a) follows from the construction and the observation (\ref{eq:obs}). Condition (b) follows from Lemma \ref{lem:factorization} and Definition \ref{defn:P}.
\end{proof}
\subsection{Computation with the base change}\label{subsec:comp of BC}
We keep the setting of \S \ref{subsec:BCFL in basic case} and \S \ref{subsec:general facts}. We assume that $s_0$ is divisible by the order $d$ of $\theta$, and consider $s \in s_0 \mathbb N$.

The Satake isomorphism for $\mathcal H_s$ is
$$\mathrm {Sat} : \mathcal H_s \isom \C [X^*(\widehat T)] ^{W_0}. $$
For each $\mu \in X^*(\widehat T) ^+$, let $\tau _{\mu}'$ be the character of the highest weight representation $V_{\mu}$ of $\widehat G$ of highest weight $\mu$. Then
$$ \set{\tau' _{\mu}}_{\mu \in X^*(\widehat T) ^+} $$ is a  basis of $\C [ X^* (\widehat T)] ^{W_0}$. This basis is  the absolute analogue of the basis $\set{\tau_{\mu}}_{\mu \in P^+}$ of $\C [ Y^*]^{W^1}$ (i.e., they are the same if $\theta = 1$).

Recall from \S \ref{subsec:CZ Conj} and \S \ref{subsec:general facts} that we have
$$Y^* = X_*(A) = X^*(\widehat T) ^{\hat \theta},\quad X^*(\hatS) = X^*(\widehat{ T}) _{\hat \theta, \mathrm{free}}.$$
By Lemma \ref{lem:elementary about coinv} (3), the composition
$$ Y^*\otimes \Q \to X^*(\widehat T) \otimes \Q\to X^*(\hatS) \otimes \Q $$ is invertible. We denote its inverse map by $\lambda \mapsto \lambda ^{(1)}.$ For all $\lambda \in X^*(\hatS)$, we define $\lambda ^{(s)}$ to be $s \lambda ^{(1)}$, which lies in $Y^*$ since $s$ is divisible by $d$. Thus we have a map
\begin{align}\label{eq:bracket s}
X^*(\hatS) \To Y^*, \quad  \lambda \longmapsto \lambda ^{(s)},
\end{align} which is an isomorphism after $\otimes \Q$. In the case $\theta=1$, this is none other than the multiplication-by-$s$ map from $Y^*$ to itself. In general, we denote by $X^*(\hatS) ^+ \subset X^*(\hatS)$ the natural image of $X^*(\widehat T)^{+}$. Then (\ref{eq:bracket s}) maps $X^*(\hatS) ^+$ into $P^+ \subset Y^*$. Moreover, the action of $W^1$ on $X^*(\widehat T)$ induces an action of $W^1$ on $X^*(\hatS)$, and the map (\ref{eq:bracket s}) is $W^1$-equivariant.

\begin{prop}\label{prop:formula for BC}
	Under the Satake isomorphisms, the base change map $\BC_s: \mathcal H_s \to \mathcal H_1$ becomes
	\begin{align*}
	\BC_s : \C [ X^* (\widehat T)] ^{W_0} & \To \C [ Y^*] ^{W^1} \\
	\forall \mu \in X^*(\widehat T )^+, ~ \tau _{\mu} ' & \longmapsto \sum _{\lambda \in X^*(\hatS)^+}  \dim V_{\mu} (\lambda) _{\mathrm{rel}} \cdot m _{\lambda ^{(s)}}.
	\end{align*}
\end{prop}
\begin{proof} To simplify notation we write $X^*$ for $X^*(\widehat T)$.
	To compute $\BC_s$ as a map $\C [X^*] ^{W_0} \to \C [ Y^*] ^{W^1}$, it suffices to compose the map with the natural inclusion $\C[ Y^*] ^{W^1} \subset \C [ X^*]$. For each $\mu \in X^{*,+}$, let $$ m'_{\mu} : = \sum _{\lambda \in W_0 (\mu)} e^{\lambda} .$$ Then $\set{m'_{\mu}} _{\mu \in X^{*,+}}$ is a basis of $\C [ X^*]^{W_0}$. This basis is just the absolute analogue of the basis $\set{m_{\mu}}_{\mu \in P^+}$ of $\C [ Y^*] ^{W^1}$. It easily follows from definitions (see for example \cite{borelcorvallisarticle}) that $\BC_s$ as a map $ \C [ X^*] ^{W_0} \to \C [ X^*]$ sends each $m'_{\mu}$ to $$\sum _{\lambda \in W_0 (\mu)} e^{\lambda + \hat \theta \lambda + \cdots + \hat \theta ^{s-1} \lambda}. $$ It follows that for all $\mu \in X^{*,+}$, we have
	\begin{align}\label{eq:first formula for BC}
	\BC_s \tau'_{\mu} = \sum _{\lambda \in X^*} \dim V_{\mu} (\lambda) e^{\lambda + \hat \theta \lambda + \cdots + \hat\theta ^{s-1} \lambda }.
	\end{align}
	Here the summation is over $X^*$ and not over $X^{*,+}$. For each $\lambda \in X^*$, the element $$\lambda + \hat\theta \lambda + \cdots + \hat \theta ^{s-1} \lambda \in X^*$$ lies in $Y^* \subset X^*$, and its image under the natural map
	$$Y^* = (X^*)^{\hat \theta}\To X^*(\hatS) = (X^*) _{\hat \theta, \mathrm{free}} $$ is equal to the image of $s\lambda \in X^*$ under the natural map $X^* \to X^*(\hatS).$ In other words, we have
	$$ \lambda + \hat \theta \lambda + \cdots + \hat \theta ^{s-1} \lambda =  (\lambda|_{\hatS}) ^{(s)},$$ where $ \lambda|_{\hatS} \in X^*(\hatS)$ denotes the image of $\lambda $ under $X^* \to X^*(\hatS)$. Hence by (\ref{eq:first formula for BC}) we have
	$$\BC_s \tau _{\mu}' = \sum _{\lambda \in X^*} \dim V_{\mu} (\lambda) e ^{(\lambda |_{\hatS})^{(s)}},$$ which is easily seen to be equal to
	\[\pushQED{\qed} \sum _{\lambda \in X^*(\hatS)} \dim V_{\mu} (\lambda) _{\mathrm{rel}} ~ e ^{\lambda ^{(s)}} = \sum _{\lambda \in X^*(\hatS)^+} \dim V_{\mu} (\lambda) _{\mathrm{rel}} ~m_{\lambda ^{(s)}}. \qedhere  \]
\end{proof}

Since $b$ is basic and $s_0$-decent, and since $s$ is divisible by $s_0$, by Lemma \ref{lem:norm is singleton} the cocharacter $s\nu_b : \mathbb G_m \to G$ is a cocharacter of $Z_G$ defined over $F$. In particular we may view $s\nu_b \in X_*(A) = Y^*$.
\begin{cor}\label{cor:BC of tau}
	For $\mu \in X^*(\widehat T) ^+$, we have
	$$(\BC_s \tau'_{\mu} ) (\gamma_s) =  \sum _{\lambda \in X^*(\hatS)^+}  \dim V_{\mu} (\lambda) _{\mathrm{rel}} ~  \mathfrak M _{\lambda ^{(s)} - s\nu _b} ^{0} (\abs{k_F}^{-1}). $$
\end{cor}
\begin{proof} By Proposition \ref{prop:formula for BC} we have
	$$(\BC_s \tau'_{\mu} ) (\gamma_s) =  \sum _{\lambda \in X^*(\hatS)^+}  \dim V_{\mu} (\lambda) _{\mathrm{rel}} ~ m_{\lambda ^{(s)}}(\gamma_s) .$$ Recall from Lemma \ref{lem:norm is singleton} that $s\nu_b$ is a central cocharacter of $G$ defined over $F$. By Corollary \ref{cor:formula for M} and Definition \ref{defn:frakM}, each $m_{\lambda ^{(s)}} (\gamma_s)$ is equal to
	$\mathfrak M _{\lambda ^{(s)}} ^{s\nu_b} (\abs{k_F} ^{-1})$, which by Lemma \ref{lem:invariant under center} is equal to $  \mathfrak M _{\lambda ^{(s)} - s\nu _b} ^{0} (\abs{k_F}^{-1}).$
\end{proof}

\subsection{Some inductive relations}\label{subsec:inductive} We keep the setting and notation of \S \ref{subsec:CZ Conj} and \S \ref{subsec:general facts}. We assume in addition that $G$ is adjoint, and that $G$ is $F$-simple. To emphasize the group $G$ we write $\mathfrak M^0_{\lambda, G} (\mathbf q^{-1}) $ for the polynomial $\mathfrak M^0_{\lambda} (\mathbf q^{-1})$ in Definition \ref{defn:frakM}. In the following we discuss how to reduce the understanding of these polynomials to the case where $G$ is absolutely simple.

We write $\dyn_G$ for the Dynkin diagram of  $(G,B,T)$. By our assumption
that $G$ is adjoint and $F$-simple, the action of $\langle\theta\rangle$ on
$\dyn_G$ is transitive on the connected components. Let $d_0$ be the number
of connected components of $\dyn_G$. Fix one connected component $\dyn_G^+$
of $\dyn_G$ once and for all. The connected Dynkin diagram $\dyn_G^+$,
together with the automorphism $\theta^{d_0}$, determines an unramified,
adjoint, absolutely simple group $G'$ over $F$, equipped with an $F$-pinning
$(B',T', \mathbb X_+')$. We apply the constructions in \S \ref{subsec:general
	facts} to $G'$. We shall add an apostrophe in the notation when we denote an
object associated to $G'$, e.g., $A', (Y^*)'$.

We have natural identifications
\begin{align*}&
(X^*(A), \leftidx_F \Phi , X_*(A), \leftidx_F\Phi ^{\vee} ) \cong (X^*(A'), (\leftidx_F \Phi)' , X_*(A'), (\leftidx_F\Phi ^{\vee})' ) ,\\& \Phi ^1 \cong (\Phi ^1)', \qquad \Phi ^{1,\vee} \cong (\Phi ^{1,\vee}) ' , \\& Y^* \cong (Y^*)', \qquad W^1 \cong (W^1)'.
\end{align*}

To be more precise, all the above identifications are derived from an identification
\begin{align}\label{eq:identification}
X_*(T) \cong \bigoplus _{i=0} ^{d_0 -1} X_*(T'),
\end{align}
under which the automorphism $\theta$ on the left hand side translates to the following automorphism on the right hand side:
$$ (\chi _0 , \chi _1 , \cdots, \chi _{d_0-1}) \longmapsto (\theta' \chi _{d_0-1}, \chi_0, \chi_1, \cdots, \chi _{d_0-2}).$$ In particular, the identification $(Y^*)' \cong Y^*$, when composed with $Y^* = X_*(A) \subset X_*(T)$ and with (\ref{eq:identification}), is the diagonal map
\begin{align}\label{eq:diag map}
(Y^*)' \To \bigoplus_{i=0} ^{d_0-1} X_*(T'), \quad  \chi ' \longmapsto (\chi', \cdots, \chi').
\end{align}

\begin{prop}\label{prop:inductive}
	For $\lambda \in Y^* $ and $ \lambda' \in (Y^*)'$ that correspond to each other, we have
	$$\mathfrak M^0_{\lambda, G} (\mathbf{q}^{-1}) = \mathfrak M^0 _{\lambda', G'} (\mathbf{q}^{-d_0}),$$ as an element of $\C[Y^*] [\mathbf{q}^{-1}] \cong \C [ (Y^*)'][\mathbf{q}^{-1}]$.
\end{prop}
\begin{proof}
	When $\beta \in \Phi ^{1,\vee}$ corresponds to $\beta ' \in (\Phi ^{1,\vee})'$, we know that $\beta$ is of the same type (I or II) as $\beta'$, and we have $ \mathtt b (\beta) = d_0 \mathtt b ' (\beta ') .$ It follows from Lemma \ref{lem:factorization} and Definition \ref{defn:P} that
	$ \mathcal P(\lambda, \mathbf{q}) = \mathcal P'(\lambda', \mathbf{q}^{d_0})$ for all $\lambda \in Y^*$ and $\lambda' \in (Y^*)'$ that correspond to each other. The proposition then follows from Definition \ref{defn:frakM}. \end{proof}

Next we deduce a relation between the construction of $\lambda_b$ in \S \ref{subsec:CZ Conj} for $G$ and for $G'$. Denote by $\hatS'$ the counterpart of $\hatS$ for $G'$. Since $G$ (resp.~$G'$) is adjoint, we know that $X^*(\widehat T)$ (resp.~$X^*(\widehat T')$) has a $\Z$-basis consisting of the fundamental weights. It then easily follows from Lemma \ref{lem:elementary about coinv} (\ref{item:2 in coinv}) \ that we have
\begin{align*}
X^*(\widehat T) _{\hat \theta} & =
X^*(\widehat T) _{\hat \theta, \mathrm{free}} = X^*(\hatS), \\
X^*(\widehat T') _{\hat \theta'} & =
X^*(\widehat T') _{\hat \theta', \mathrm{free}} = X^*(\hatS'),
\end{align*}
and we have natural identifications
$$X^*(\hatS) \cong X^*(\hatS'), \quad \widehat Q _{\hat\theta} \cong \widehat Q' _{\hat\theta'},  \quad \pi_1(G) _{\sigma} \cong \pi_1 (G')_{\sigma}. $$

Fix an arbitrary $\mu \in X_*(T)$. Choose $\mu' \in X_*(T')$, such that the image of $\mu'$ in $X^*(\hatS')$ corresponds to the image of $\mu$ in $X^*(\hatS)$. Such $\mu'$ always exists because the map $X_*(T') = X^*(\widehat T') \to X^*(\hatS')$ is surjective. It then follows that the image $\mu^{\natural} \in \pi_1(G)_{\sigma}$ of $\mu$ and the image $(\mu')^{\natural} \in \pi_1(G')_{\sigma}$ of $\mu'$ correspond to each other. Let $[b]$ (resp.~$[b']$) be the unique basic element of $B(G,\mu)$ (resp.~$B(G',\mu')$).
\begin{prop}\label{prop:same lambda_b}
	In the above setting, the elements $\lambda_b \in X^*(\hatS)$ and $\lambda_{b'} \in X^*(\hatS')$ correspond to each other, under the identification $X^*(\hatS) \cong X^*(\hatS')$.
\end{prop}
\begin{proof}This immediately follows from the uniqueness in Lemma \ref{lem:defn of lambda_b}.
\end{proof}
\section{The main result}\label{sec:main}
\subsection{The number of irreducible components in terms of combinatorial data}\label{subsec:comb}
We keep the setting of \S \ref{subsec:CZ Conj} and \S \ref{subsec:decent case}. Thus we fix a reductive group scheme $G$ over $\Ok_F$, an element $\mu \in X_*(T)^+$, and a basic class $[b] \in B(G, \mu)$. In this section, we relate the number of irreducible components $\mathscr{N}(\mu,b)$ to some combinatorial data.

As in \S \ref{subsec:decent case}, we fix $s_0 \in \mathbb N$ such that $b$ is $s_0$-decent. As in \S \ref{subsec:comp of BC} we assume $s_0$ is divisible by the order $d$ of $\theta$, and various natural numbers $s\in \mathbb N$ that are divisible by $s_0$.
In particular, $G$ will always be split over the extension $F_s$ of $F$. We shall write
$$q_s: = \abs{k_s} = \abs{k_F}^s. $$

By Corollary \ref{irred-orb}, we have \begin{align}
\label{eq2}
e(J_b) & \vol (G(\Ok_F)) ^{-1}  (\BC_s f_{\mu ,s}) (\gamma_s)  \\\nonumber & = \sum _{ Z\in J_b(F) \backslash \Sigma ^{\topp} (X_{\mu} (b))} \vol (Z)^{-1} q_s^{\dim X_{\mu}(b)} + o(q_s^{\dim X_{\mu} (b)}). \end{align}
By the dimension formula in Theorem \ref{thm:dim}, we have
\begin{align}\label{eq:dim formula}\dim X_{\mu} (b) =
\langle\mu,\rho\rangle-\frac{1}{2}\mathrm{def}_G(b)
\end{align} (since $\bar\nu_b$ is central). In particular, from (\ref{eq2}) we get
\begin{align}\label{eq:trivial bound of orb int}
(\BC_s f_{\mu ,s}) (\gamma_s) = O(q_s ^{\langle\mu,\rho\rangle-\frac{1}{2}\mathrm{def}_G(b) }).
\end{align}

\begin{prop}\label{prop:from f to tau} With the notation in \S\ref{subsec:comp of BC}, we have $$ \BC_s (\tau _{\mu}') (\gamma_s) = q_s ^{-\langle \mu ,\rho \rangle} (\BC_s f_{\mu,s} )  (\gamma_s) + o(q_s ^{-\frac{1}{2}\mathrm{def}_G(b)}).$$
\end{prop}
\begin{proof}
	For $\lambda $ running over $X^*(\widehat T)^+$, the Satake transforms of $f_{\lambda,s}$, which we still denote by $f_{\lambda,s}$, form a basis of $\C [ X^*(\widehat T)] ^{W_0}$. By the split case of Theorem \ref{thm:KL}, we have
	$$\tau '_{\mu} = \sum _{\lambda \in X^*(\widehat T) ^+} K'_{\mu ,\lambda} (q_s^{-1}) q_s ^{- \langle \lambda ,\rho \rangle} f_{\lambda,s} ,$$
	where $K'_{\mu,\lambda}(\cdot)$ is the absolute analogue of (\ref{eq:defn of pairing}), i.e., it is defined by (\ref{eq:defn of pairing}) with $\theta$ replaced by $1$. By Definition \ref{defn:P}, Corollary \ref{cor:triv bd}, and (\ref{eq:defn of pairing}), we have
	$$K' _{\mu, \lambda} (q_s^{-1}) = \begin{cases}
	1 +O(q_s^{-1}), & \lambda = \mu ,\\
	O(q_s ^{-1}) , & \lambda < \mu , \\
	0, & \text{otherwise}.
	\end{cases}$$
	Therefore
	\begin{align}\label{eq:exp of tau'}
	\tau_{\mu}' = q_s ^{-\langle \mu ,\rho \rangle} f_{\mu ,s} + \sum _{\lambda \in X^*(\widehat T) ^+, ~ \lambda \leq \mu} O(q_s ^{-1 - \langle \lambda, \rho \rangle} ) f_{\lambda ,s}
	\end{align}
	Note that (\ref{eq:trivial bound of orb int}) is valid with $\mu$ replaced by each $\lambda \in X^*(\widehat T) ^+, \lambda \leq \mu$, because we still have $[b] \in B(G, \lambda)$. The proposition then follows from (\ref{eq:exp of tau'}) and the above-mentioned bounds provided by (\ref{eq:trivial bound of orb int}) with $\mu$ replaced by each $\lambda \leq \mu$.
\end{proof}
\begin{cor}
	We have
	\begin{align}\label{eq:exp of BC tau}
	\BC_s (\tau '_{\mu}) (\gamma_s) = e(J_b)\sum _{ Z\in J_b(F) \backslash \Sigma ^{\topp} (X)} \vol (Z)^{-1} q_s^{-\frac{1}{2}\mathrm{def}_G(b) } + o(q_s^{-\frac{1}{2}\mathrm{def}_G(b)}).
	\end{align}
\end{cor}
\begin{proof}
	This follows by combining (\ref{eq2}), (\ref{eq:dim formula}), and Proposition \ref{prop:from f to tau}.
\end{proof}

\begin{thm}\label{main-thm} Assume the Haar measures are normalized such that $G(\Ok_F)$ has volume $1$. There exists a rational function $S_{\mu,b} (t) \in \Q (t)$ that is independent of the local field $F$ (in the same sense as Corollary \ref{cor:vol-ind-p}), such that \begin{equation}\label{eq:constant term} S_{\mu, b }(0)=\mathscr{N}(\mu,b),
	\end{equation}
	\begin{align}\label{eq:S and vol}
	S_{\mu, b } (q_1) = e(J_b)\sum _{ Z\in J_b(F) \backslash \Sigma ^{\topp} (X_{\mu} (b))} \vol (Z)^{-1},
	\end{align} and such that
	\begin{align}\label{eq:main-thm}
	S_{\mu,b}(q_1)q_s^{-\frac{1}{2}\mathrm{def}_G(b)} = \sum_{\lambda \in X^*(\hatS) ^+}\dim V_\mu(\lambda)_{\mathrm{rel}} ~\mathfrak M _{\lambda ^{(s)} - s\nu _b} ^{0} (q_1^{-1}) +o(q_s^{-\frac{1}{2}\mathrm{def}_G(b) }).
	\end{align}
	In particular
	\begin{align}\label{eq:limit formula}
	S_{\mu,b}(q_1)=\lim\limits_{s\rightarrow \infty} q_s ^{\frac{1}{2} \mathrm{def}_G(b)}\sum_{\lambda \in X^*(\hatS) ^+}\dim V_\mu(\lambda)_{\mathrm{rel}} ~\mathfrak M _{\lambda ^{(s)} - s\nu _b} ^{0} (q_1^{-1}).
	\end{align}
\end{thm}
\begin{proof}  Fix a set of representatives $\set{Z_i\mid 1\leq i \leq \mathscr N (\mu, b )}$ for the $J_b(F)$-orbits in $\Sigma ^{\topp} (X_{\mu} (b))$. For each $Z_i$, let $R_{Z_i} (t) \in \Q (t)$ be the rational function associated to $Z_i$ as in Corollary \ref{cor:vol-ind-p}. Let $$S_{\mu,b} (t) : = e(J_b)\sum _{i =1} ^{\mathscr N (\mu,b)} R_{Z_i} (t) ^{-1}.$$ Then $S_{\mu,b} (t)$ belongs to $\Q(t)$ and it satisfies (\ref{eq:constant term}), (\ref{eq:S and vol}). It follows from Corollary \ref{cor:vol-ind-p} and (\ref{eq:exp of BC tau}) that $$
	\BC_s (\tau '_{\mu}) (\gamma_s)  =  S_{\mu,b}(q_1)q_s^{-\frac{1}{2}\mathrm{def}_G(b) } + o(q_s^{-\frac{1}{2}\mathrm{def}_G(b)}) . $$
	Comparing this with Corollary \ref{cor:BC of tau}, we obtain (\ref{eq:main-thm}). \end{proof}

The upshot of this theorem is that the right hand side of (\ref{eq:limit formula}) is purely combinatorial and can be computed (at least in certain instances) using Kostant's partition function $\mathcal{P}_{\mathrm{Kos}}(\lambda,\mathbf{q})$. Moreover the fact that $S_{\mu,b}(t)$ is a rational function independent of the local field $F$, means that it is in principle determined by its values $S_{\mu,b}(q_1)$ for infinitely many choices of $q_1$. Once $S_{\mu,b} (t)$ is determined, the number $\mathscr N(\mu,b)$ can be read off from (\ref{eq:constant term}).

\subsection{The case of unramified elements}\label{sub:unr}
In this subsection we apply Theorem \ref{main-thm} to prove Conjecture \ref{conj:chenzhu} for \emph{unramified} and basic $b$. This is a new proof of the result \cite[Theorem 4.4.14]{xiaozhu}.

We keep the setting of  \S\ref{subsec:comb}.
Assume in addition that $b$ is \emph{unramified}, in the sense of \cite[\S 4.2]{xiaozhu}. Then we have $J_b\cong G$, and hence $\mathrm{def}_G(b)=0, e(J_b) =1$. In view of Theorem \ref{main-thm}, we would like to compute $$\lim\limits_{s\rightarrow \infty}\sum_{ \lambda\in X^*(\hatS) ^+}\dim V_\mu(\lambda)_{\mathrm{rel}}~ \mathfrak{M}_{\lambda^{(s)}-s\nu_b}^0(q_1^{-1}).$$
\begin{prop}\label{prop:est for unram} Let $\lambda \in X^*(\hatS) ^+$ and $s \in s_0 \mathbb N$. We have $$\mathfrak{M}^0_{\lambda^{(s)}-s\nu_b}(q_1^{-1})=\begin{cases}1, & \text{ if $\lambda= \lambda_b$}, \\O(q_1^{-a s}) \text{~for some $a\in\mathbb{R}_{>0}$} ,& \text{ otherwise}.
	\end{cases}		$$
\end{prop}
\begin{proof}Firstly, by \cite[Lemma 4.2.3]{xiaozhu}, $\lambda_b \in X^*(\hatS)$ is the unique element such that $\lambda_b ^{(s)} = s\nu_b$ for one (and hence all) $s \in s_0 \mathbb N$. (In particular, $\lambda_b \in X^*(\hatS) ^+$ as $\nu_b $ is central.) Thus for $\lambda \in X^*(\hatS)^+$, we have $\lambda= \lambda_b$ if and only if $\lambda^{(s) }  - s\nu_b =0$ for one (and hence all) $s\in s_0 \mathbb N$.
	
	Let $w \in W^1$. Since $w\rho^\vee-\rho^\vee $ is not in $R^+$ for $w\neq 1$, it follows from Definition \ref{defn:P} and Definition \ref{defn:frakM} that $\mathfrak{M}_0^0(q^{-1})=1 \in \C[ q^{-1}].$ This proves the case $\lambda^{(s)}=s\nu_b$.
	
	Now assume $\lambda^{(s)}\neq s\nu_b$ (for all $s$). Fix $w,w'\in W^1 $. We write $\mu_s: = \lambda ^{(s)} - s\nu_b$ and $\psi_s : = w \bullet (w'\mu_s)$. By the formula (\ref{eq:defn of M}), it suffices to show that
	\begin{align}\label{desired est unram} \exists a > 0 , ~
	\mathcal P ( \psi_s , q_1^{-1}) = O(q_1^{-as}).
	\end{align}
	
	If $\psi_s \notin R^+$ for some value of $s$, then by definition $\mathcal P(\psi_s, q^{-1} ) = 0 \in \C [ q^{-1}]$. Hence we may ignore these values of $s$. On the other hand, if $\psi_s \in R^+$ for some $s \in s_0 \mathbb N$, then it is easy to see that $\psi _{ns}\in R^+$ for all $n \in \mathbb N$. We thus assume that $\psi_s \in R^+ $ for all sufficiently divisible $s$. Then for such $s$ we have
	$$ w w ' \mu_s \in \mathrm{span}_{\Q} (\leftidx_F \Phi ^{\vee}) - \set{0} \subset Y^*\otimes \Q.$$ Hence there exists a non-zero $\Q$-linear functional $f_0$ on $\mathrm{span}_{\Q} (\leftidx_F \Phi ^{\vee})$, and a constant $c_0\in \Q -\set{0}$, such that \begin{align}\label{eq:ci0}
	f_{0} (ww'\mu_s) = s\cdot  c_0, ~\forall s \gg 0.
	\end{align}
	
	We define
	$$A : =\abs{f_0 ( w \rho ^{\vee} - \rho^{\vee})} \in \mathbb R_{\geq 0} , \qquad B : =  \max _{\beta \in \leftidx _F \Phi^{\vee, +}} \abs{f_0 (\beta)} + 1 \in \mathbb R_{>0} .$$ Then for $s\gg 0$ we have
	$$\min_{\underline m  \in \mathbb P (\psi_s)} \abs{\underline m}  \geq B^{-1} \cdot \abs{f_0 (\psi_s)}  \geq B^{-1} (\abs{f_0 (ww'\mu_s)} - A) =  B^{-1}(s \abs {c_0} -A),$$ where the last equality follows from (\ref{eq:ci0}). Since $B^{-1}$ and $c_0$ are both non-zero, there is a constant $N_0 > 0$ such that
	\begin{align}\label{eq:bd for Ns}
	\min_{\underline m  \in \mathbb P (\psi_s)} \abs{\underline m} > N_0 \cdot s, ~\forall s \gg 0.
	\end{align}
	Combining (\ref{eq:bd for Ns}) and  Proposition \ref{prop:interpretation} (2), we have
	\begin{align}\label{eq:total bd}
	\abs{\mathcal P(\psi_s, q_1^{-1})}  \leq \# {\mathbb P (\psi_s)}\cdot  q_1 ^{-N_0 s} .\end{align}
	On the other hand, note that for any $\Q$-linear functional $f$ on $\mathrm{span}_{\Q} (\leftidx_F \Phi ^{\vee})$, the function $s \mapsto f(\psi_s)$ is an affine function in $s$. Hence there exists a constant $L>0$ such that
	$$\forall s \gg 0 , ~\forall \underline m \in \mathbb P(\psi_s) , ~ \abs{\underline m} \leq L s . $$  It follows that
	\begin{align}\label{eq:bd for number of partitions}
	\# \mathbb P(\psi_s) = \sum _{l=1} ^{Ls} \# \mathbb P(\psi_s)_l \leq \sum _{l=1} ^{Ls} l ^{\# (\leftidx_F \Phi ^{\vee,+})} \leq (Ls) ^M
	\end{align} for some constant $M >0$. The desired estimate (\ref{desired est unram}) then follows from (\ref{eq:total bd}) and (\ref{eq:bd for number of partitions}).
\end{proof}
\begin{thm}\label{thm:unram}
	Keep the setting of  \S\ref{subsec:comb}.
	Assume in addition that $b$ is unramified. Then $\mathscr N (\mu , b) = \mathscr M (\mu,b)$. Moreover, for any $Z \in \Sigma ^{\topp} (X_{\mu} (b))$, the group $\Stab_Z(J_b(F))$ is a hyperspecial subgroup of $J_b(F) = G(F)$.
\end{thm}
\begin{proof}
	Let $S_{\mu,b} (t) \in \Z(t)$ be as in Theorem \ref{main-thm}. By (\ref{eq:limit formula}) and Proposition \ref{prop:est for unram}, we have
	$$S_{\mu, b} (q_1) = \dim V_{\mu} (\lambda) _{\mathrm{rel}} ,$$ where $\lambda$ is the unique element of $X^*(\hatS) ^+$ such that $\lambda ^{(s_0)} = s_0 \nu_b$. By varying the local field $F$, we see that $S_{\mu,b} (t)$ is the constant $\dim V_{\mu} (\lambda_b) _{\mathrm{rel}} = \mathscr M(\mu,b)$. In particular
	$$\mathscr N(\mu,b) = S_{\mu,b} (0) = \mathscr M(\mu,b).$$

	For the second part, by our normalization we have $\vol (Z) \leq 1 $ for each element $Z \in \Sigma ^{\topp} (X_{\mu} (b))$, where equality holds if and only if $\Stab_Z (J_b(F))$ is hyperspecial. On the other hand, combining (\ref{eq:constant term}) and (\ref{eq:S and vol}) and the fact that $S_{\mu,b} (t)$ is constant, we have
	$$ \mathscr N (\mu,b) = \sum _{Z \in J_b(F) \backslash \Sigma ^{\topp} (X_{\mu}(b))} \vol (Z) ^{-1}.$$ It follows that each $\vol ( Z) $ must be $1$, and that $\Stab_Z(J_b(F))$ is hyperspecial.
\end{proof}
\ignore{
	\begin{rem}Arguably the hardest part of the proof of the corresponding result in \cite{xiaozhu} is to show that the stabilizer of any irreducible component in $\Sigma^{\topp}(X_{\mu} (b))$ is hyperspecial.
	\end{rem}
	\begin{rem}
		In Theorem \ref{thm:unram} we assume that $b$ is basic. One can show as in Proposition \ref{prop:reduction} that the general unramified case of Conjecture \ref{conj:chenzhu} reduces to the basic unramified case.
	\end{rem}
}
\subsection{The general case}
We now prove the general case of Conjecture \ref{conj:chenzhu}. By Proposition \ref{prop:reduction}, there is no loss of generality in assuming that $G$ is adjoint and $F$-simple, and that $[b] \in B(G)$ is basic. In particular $\bar \nu_b =0$.
In the following, we fix such $G$ and $[b]$, and freely use the notation from \S \ref{subsec:comb}. Note that we only fix $[b]$ and do \emph{not} fix a prescribed $\mu \in X_*(T)^+$ such that $[b] \in B(G,\mu)$.

As in Definition \ref{defn:lambda_b}, we have $\lambda_b \in X^*(\hatS)$. Denote by $\lambda_{b}^+$ the unique element in the $W^1$-orbit of $\lambda_b$ that lies in $X^*(\hatS)^+$.
We define (cf.~the discussion above Lemma \ref{lem:defn of lambda_b})
$$\Lambda(b) : = \set{\lambda \in X^*(\hatS)^+\mid \lambda \neq \lambda _b^+, ~ \lambda - \lambda _b \in \widehat Q_{\hat \theta}}.$$
The following lemma is the motivation for introducing $\Lambda(b)$.

\begin{lem}\label{lem:Lambda(b)}
	Let $\lambda \in X^*(\hatS) ^+$ and let $\mu \in X_*(T)^+$. Assume that $[b] \in B(G,\mu)$. If $ V_\mu(\lambda)_{\mathrm{rel}} \neq 0$, then $\lambda \in \Lambda(b) \sqcup \set{\lambda_b^+}$.
\end{lem}
\begin{proof}
	Assume that $ V_\mu(\lambda)_{\mathrm{rel}} \neq 0$. Then there exists $\lambda' \in X^*(\widehat T)$ lifting $\lambda$, such that $V_{\mu}(\lambda') \neq 0$.
	It follows that $\mu$ and $\lambda'$ have the same image in $\pi_1(G)$, and hence the same image in $\pi_1(G)_\sigma$.
	Let $\lambda''$ be the image of $\lambda'$ in $X^*(\widehat T)_{\hat \theta}$. Then the image of $\lambda''$ in $\pi_1(G)_{\sigma}$ is equal to that of $\mu$, namely $\kappa(b)$.
	Hence $\lambda''$ and $\tilde \lambda_b$ have the same image in $\pi_1(G)_{\sigma}$.
	By the exact sequence
	$$
	\widehat Q_{\hat \theta} \to X^*(\widehat T)_{\hat \theta} \to \pi_1(G)_{\sigma} $$
	we know that $\tilde \lambda_b - \lambda''$ lies in $\widehat Q_{\hat \theta}$. Now recall from the discussion above Lemma \ref{lem:defn of lambda_b} that $\widehat Q_{\hat \theta}$ injects into both $X^*(\widehat T)_{\hat \theta}$ and $X^*(\widehat {\mathcal S})$.
	Hence $\lambda_b -\lambda$ belongs to $\widehat Q_{\hat \theta}$ viewed as a subgroup of $X^*(\widehat{\mathcal S})$. This shows that $\lambda \in \Lambda(b) \sqcup \set{\lambda_b^+}$.
\end{proof}
Since $G$ is $F$-simple, all the simple factors of $G_{\overline F}$ have the same Dynkin type. We shall call this type the type of $G$.
The following proposition is the key result towards the proof of Conjecture \ref{conj:numerical Chen-Zhu}.
\begin{prop}[Key estimate]\label{key-bound}
	Assume $G$ is adjoint, $F$-simple, and not of type $A$. Let $[b] \in B(G)$ be a basic class. Assume $[b]$ is not unramified.
	\begin{enumerate}
		\item \label{item:1} Assume $G$ is not a Weil restriction of the split adjoint $E_6$. For all $\lambda \in \Lambda (b)$, there exists $a>0$, such that
		\begin{align}\label{eq:key est 1}
		\mathfrak{M}_{\lambda^{(s)}}^0(q_1^{-1})
		=O(q_1^{-s(\frac{1}{2}\mathrm{def}_G(b)+a)}).
		\end{align}
		Moreover, there exists $\mu_1\in X^*(\widehat T) ^+$ that is minuscule, such that $[b] \in B(G,\mu_1)$ and $\mathscr M(\mu_1,b) : = \dim V_{\mu_1} (\lambda _b) _{\mathrm{rel}} =1$.   \item \label{item:2} Assume $G$ is a Weil restriction of the split adjoint $E_6$ (necessarily along an unramified extension of $F$). Then there is an element $\lambda _{\mathrm{bad}} \in \Lambda(b)$ with the following properties:
		\begin{itemize}
			\item For all $\lambda \in \Lambda (b) -\set{\lambda _{\mathrm{bad}}}$, there exists $a>0$, such that
			\begin{align}\label{eq:key est 2}
			\mathfrak{M}_{\lambda^{(s)}}^0(q_1^{-1})
			=O(q_1^{-s(\frac{1}{2}\mathrm{def}_G(b)+a)}).
			\end{align}
			\item There exist $\mu_1 , \mu_2 \in X^*(\widehat T) ^+$, such that $\mu_1$ is minuscule and $\mu_2$ is a sum of dominant minuscule elements, such that $b \in B(G,\mu_1) \cap B(G,\mu_2)$, and such that
			$$\mathscr M(\mu_1,b): = \dim V_{\mu_1} (\lambda _b )_{\mathrm{rel}} = 1 ,$$ and $$ V_{\mu_1} (\lambda _{\mathrm{bad}}) _{\mathrm{rel}} = 0 , \quad V_{\mu_2} (\lambda _{\mathrm{bad}}) _{\mathrm{rel}} \neq 0.$$
			
		\end{itemize}
	\end{enumerate}
	
\end{prop}
The proof of Proposition \ref{key-bound} will occupy \S \ref{sec:proof of key est}, \S \ref{sec:part II}, \S \ref{sec:part III} below. We now admit this proposition.

\begin{thm}\label{thm:CZ}
	Conjecture \ref{conj:numerical Chen-Zhu} holds for $G$ adjoint, $F$-simple, not of type $A$, and for $[b] \in B(G)$ basic.
\end{thm}
\begin{rem}\label{rem:logical}
	Here is the logical dependence of our proof of Theorem \ref{thm:CZ} on the previous work of other authors:
	\begin{itemize}
		\item If $[b]$ is unramified, then our proof is logically independent of the approaches in \cite{xiaozhu}, \cite{HV}, or \cite{nienew}.
		\item If $[b]$ is not unramified, and if we are in the situation of Proposition \ref{key-bound} (\ref{item:1}), then our proof depends on results from \cite{HV}.
		\item If $[b]$ is not unramified, and if we are in the situation of Proposition \ref{key-bound} (\ref{item:2}), then our proof depends on Nie's result Theorem \ref{thm:nie orginal} (3) applied to $G$ and $\mu_1,\mu_2$.
	\end{itemize}
\end{rem}
\begin{proof}[Proof of Theorem \ref{thm:CZ}] If $[b]$ is unramified, then the present theorem is just Theorem \ref{thm:unram} (which is also valid for type $A$). From now on we assume $[b]$ is not unramified. To simplify notation, we write $\mathrm{def}$ for $\mathrm{def}_G(b)$.
	
	Assume we are in the situation of Proposition \ref{key-bound} (\ref{item:1}). By Theorem \ref{main-thm}, Lemma \ref{lem:Lambda(b)}, and Proposition \ref{key-bound} (\ref{item:1}), for all $\mu \in X_*(T) ^+$ such that $[b] \in B(G,\mu)$, we have
	\begin{align}\label{eq:to cite in intro}
	S_{\mu,b} (q_1) =\mathscr M(\mu,b) \lim _{s\to \infty} q_s ^{\frac{1}{2}\mathrm{def}} \mathfrak M^0 _{\lambda _b^{+, (s)}} (q_1 ^{-1}).
	\end{align}
	In particular, we have
	$$S_{\mu,b} (q_1) = \frac{\mathscr M(\mu,b)}{ \mathscr M (\mu_1,b)} S_{\mu _1 , b} (q_1) =  \mathscr M(\mu,b) S_{\mu _1 , b} (q_1).$$ By varying the local field $F$ (whilst preserving the affine root system of $G$) we conclude that
	\begin{align}
	S_{\mu,b} (t) =\mathscr M(\mu,b) S_{\mu _1 , b} (t) \in \Q (t).
	\end{align} Evaluating at $0$, we have
	$$ \mathscr N(\mu,b) = S_{\mu,b} (0) = \mathscr M(\mu,b) S_{\mu _1 , b} (0) = \mathscr M(\mu,b) \mathscr N(\mu_1 , b). $$ On the other hand, as $\mu_1$ is minuscule, it is shown in \cite[Theorem 1.4]{HV} that $\mathscr N(\mu_1, b) \leq \mathscr M(\mu_1,b).$ Since $\mathscr N(\mu_1, b)$ is a positive natural number and $\mathscr M(\mu_1,b)= 1$, we have $\mathscr N (\mu_1, b) =1$. Thus
	$\mathscr N(\mu,b) = \mathscr M(\mu,b),$ as desired.
	
	Now assume we are in the situation of Proposition \ref{key-bound} (\ref{item:2}). For each $\mu \in X_*(T)^+$, we write $d_{\mu}$ for $\dim V_{\mu} (\lambda_{\mathrm{bad}})_{\mathrm{rel}}$. By Theorem \ref{main-thm}, Lemma \ref{lem:Lambda(b)}, and Proposition \ref{key-bound} (\ref{item:2}), for all $\mu \in X_*(T) ^+$ such that $[b] \in B(G,\mu)$, we have
	\begin{align}\label{eq:cancel1}
	S_{\mu,b} (q_1) = \lim _{s\to \infty} \bigg [  \mathscr M(\mu,b) q_s ^{\frac{1}{2}\mathrm{def}}   \mathfrak M^0 _{\lambda_b ^{+, (s)}} (q_1 ^{-1}) +  d_{\mu}  q_s ^{\frac{1}{2}\mathrm{def}}  \mathfrak M^0 _{\lambda_{\mathrm{bad}} ^{(s)}} (q_1 ^{-1}) \bigg ] .
	\end{align}
	In particular, taking $\mu$ to be $ \mu_1$ and $\mu_2$ respectively, we obtain
	\begin{align}\label{eq:cancel2}
	S_{\mu_1, b} (q_1)  & =  \lim _{s\to \infty} q_s ^{\frac{1}{2}\mathrm{def}}   \mathfrak M^0 _{\lambda_b ^{+, (s)}} (q_1 ^{-1}),  \\
	\label{eq:cancel3} S_{\mu_2, b} (q_1) & = \lim _{s\to \infty} \bigg [ \mathscr M(\mu_2,b)  q_s ^{\frac{1}{2}\mathrm{def}}   \mathfrak M^0 _{\lambda_b ^{+, (s)}} (q_1 ^{-1}) +  d_{\mu_2}   q_s ^{\frac{1}{2}\mathrm{def}_G(b)}  \mathfrak M^0 _{\lambda_{\mathrm{bad}} ^{(s)}} (q_1 ^{-1}) \bigg ] .
	\end{align}
	Comparing (\ref{eq:cancel1}) (\ref{eq:cancel2}) (\ref{eq:cancel3}) and using $d_{\mu_2} \neq 0$, we obtain
	\begin{align*}
	S_{\mu, b} (q_1) = \mathscr M(\mu,b) S_{\mu_1, b} (q_1)  + \frac{ d_{\mu} } {d_{\mu_2}}  (S_{\mu_2,b} (q_1) - \mathscr M(\mu_2,b) S_{\mu_1,b} (q_1)).
	\end{align*} By varying $F$, we obtain
	\begin{align}\label{eq:three rational functions}
	S_{\mu, b} (t) = \mathscr M(\mu,b) S_{\mu_1, b} (t)  + \frac{d_{\mu} } {d_{\mu_2}}  (S_{\mu_2,b} (t) - \mathscr M(\mu_2,b) S_{\mu_1,b} (t)),
	\end{align} as an equality in $\Q(t)$. Since $\mu_1, \mu_2$ are sums of dominant minuscule elements, Theorem \ref{thm:nie orginal} (3) implies that
	$$ \mathscr N(\mu_1, b) = \mathscr M(\mu_1,b)=1, ~\text{and } \mathscr N(\mu_2, b) = \mathscr M(\mu_2,b).$$ Consequently we have
	$ S_{\mu_1, b} (0) = 1$ and $S_{\mu_2,b} (0) = \mathscr M(\mu_2,b).$ Evaluating (\ref{eq:three rational functions}) at $t=0$, we obtain
	$$\mathscr N(\mu,b) = S_{\mu,b} (0) = \mathscr M(\mu,b) +  \frac{d_{\mu} } {d_{\mu_2}}  (\mathscr M(\mu_2,b)- \mathscr M(\mu_2,b)) = \mathscr M(\mu,b)$$ as desired.
\end{proof}

\begin{cor}\label{cor:CZ}
	Conjecture \ref{conj:numerical Chen-Zhu} is true in full generality.
\end{cor}
\begin{proof}
	By Proposition \ref{prop:reduction}, we reduce to the case where $G$ is adjoint and $F$-simple, and $[b]$ is basic. If $G$ is not of type $A$, the conjecture is proved in Theorem \ref{thm:CZ}. If $G$ is of type $A$, the conjecture follows from Theorem \ref{thm:nie orginal} (3).
\end{proof}
\textbf{
	The rest of the paper is devoted to the proof of Proposition \ref{key-bound}.}
\subsection{Reduction to the absolutely simple case}
\begin{lem}\label{lem:reduction to absolutely simple}
	Proposition \ref{key-bound} holds true if it holds for all $G$ that are absolutely simple and adjoint, not of type $A$.
\end{lem}
\begin{proof}
	Let $G$ be as in Proposition \ref{key-bound}, not necessarily absolutely simple. Fix a basic $[b] \in B(G)$ as in Proposition \ref{key-bound}. Let $G'$ be the auxiliary absolutely simple and adjoint group  over $F$, constructed in \S \ref{subsec:inductive}. We keep the notation established there. Note that $[b]$ is completely determined by $\kappa_G(b) \in \pi_1(G) _{\sigma}$. We construct a basic $[b'] \in B(G')$ as in \S \ref{subsec:inductive}, such that $\kappa_G(b)$ and $\kappa_{G'} (b')$ correspond to each other under the identification $\pi_1 (G) _{\sigma} \cong \pi_1 (G') _{\sigma}.$
	
	We write $\edyn_G$ for the extended Dynkin diagram of $G$ and write $\Aut(\edyn_G)$ for its automorphism group. We write $\abs{\edyn_G}$ for the set of nodes in $\edyn_G$.  Similarly for $G'$.
	
	We claim that
	\begin{align}\label{eq:equality of defect}
	\mathrm{def}_G(b) = \mathrm{def}_{G'} (b').
	\end{align} In fact, there is a natural embedding $\pi_1(G) \rtimes \lprod{\theta} \hookrightarrow \Aut (\edyn_G)$ given by the identification of $\pi_1(G)$ with the stabilizer $\Omega$ in $W$ of the base alcove, and  the natural faithful action of $\Omega$ on $\edyn_G$. The number $\mathrm{def}_G(b)$ is computed as the number of $\theta$-orbits minus the number of $[\mu] \rtimes \theta$-orbits in $\abs{\edyn_G}$, where $[\mu]\in \pi_1 (G)$ is any lift of $\kappa_G (b) \in \pi_1 (G) _{\sigma}$.  Similarly, choosing a lift $[\mu'] \in \pi _1 (G')$ of $\kappa_{G'} (b')$, we compute $\mathrm{def}_{G'}(b')$ as the number of $\theta'$-orbits minus the number of $[\mu'] \rtimes \theta'$-orbits in $\abs{\edyn_{G'}}$. Now by construction, $\edyn_{G'}$ is identified with a particular connected component of $\edyn _G$. We may thus embed $\Aut (\edyn _{G'})$ into $\Aut (\edyn _{G})$ by extending the action trivially to other connected components. Then inside $\Aut (\edyn _{G})$ we have the following relations:
	$$ \theta' = \theta ^{d_0} , \quad  \pi _1 (G) = \bigoplus _{i=0} ^{d_0-1} \theta ^i \pi_1 (G') \theta ^{-i} .$$
	In particular, we have an embedding $\pi_1 (G') = \theta ^0 \pi_1 (G') \theta ^0 \hookrightarrow \pi_1(G)$. We may arrange that $[\mu]$ is the image of $[\mu']$ under this embedding. Then we have
	\begin{align*}
	\# \set{ \theta\mbox{-orbits in~} \abs{\edyn _G}}    & =  \# \set{ \theta'\mbox{-orbits in~} \abs{\edyn _{G'}}}, \\ \# \set{ [\mu]\rtimes \theta\mbox{-orbits in~} \abs{\edyn _G}} &     =  \# \set{ [\mu']\rtimes\theta'\mbox{-orbits in~} \abs{\edyn _{G'}}}.
	\end{align*}
	The claim follows.
	
	Next, we naturally identify $X^*(\hatS)$ with $X^*(\hatS')$. Then it is easy to see that $\lambda_b$ corresponds to $\lambda_{b'}$ under this identification. For clarity, we denote the analogue of the map (\ref{eq:bracket s}) for $G'$ as:
	$$ X^*(\hatS') \To (Y^*)', \quad \lambda \longmapsto \lambda ^{((s))}.$$ The target of the above map is identified with $Y^*$. Then since the identification $Y^* \cong (Y^*)'$ amounts to the diagonal map (\ref{eq:diag map}), we see that
	\begin{align}\label{eq:multiple for lambda}
	\lambda ^{(d_0s)} = \lambda ^{((s))} , ~\forall \lambda \in X^*(\hatS).
	\end{align}
	
	Combining (\ref{eq:equality of defect}), (\ref{eq:multiple for lambda}) with Propositions \ref{prop:inductive}, \ref{prop:same lambda_b}, we see that the bounds (\ref{eq:key est 1}) and (\ref{eq:key est 2}) in Proposition \ref{key-bound} for $(G,b, s: = d_0 s')$ reduce to the corresponding bounds for $(G',b', s')$. In the situation of Proposition \ref{key-bound} (\ref{item:2}), we define $\lambda_{\mathrm{bad}}$ for $(G,b)$ to be equal to that for $(G',b')$, under the identification $\Lambda (b) \cong \Lambda (b')$.
	
	Finally, by hypothesis the desired $\mu_1'$ or $\set{\mu_1',\mu_2'}$ are already defined for $(G',b')$, as in Proposition \ref{key-bound}.     Under the identification (\ref{eq:identification}) we define $\mu_i \in X_*(T) ^+$ to be $(\mu_i', 0,\cdots, 0)$ for $i=1,2$.
\end{proof}
\subsection{Strategy of proving Proposition \ref{key-bound} in the absolutely simple case}\label{subsec:strategy} In Lemma \ref{lem:reduction to absolutely simple}, we already reduced the proof of Proposition \ref{key-bound} to the absolutely simple case. \textbf{From now on until the end of the paper, we assume that $G$ is an absolutely simple adjoint group over $F$ which is not of type $A$.}

As in the proof of Lemma \ref{lem:reduction to absolutely simple}, we denote by $\edyn_G$ the extended Dynkin diagram of $G$, denote by $\Aut (\edyn_G)$ its automorphism group, and denote by $\abs{\edyn_G}$ the set of nodes. To prove Proposition \ref{key-bound}, consider a basic class $[b] \in B(G)$ which is not unramified.
Since $b$ is not unramified, we have $\kappa_G(b)\neq 0$, and in particular the groups $\pi_1(G)$ and $\Aut (\edyn_G)$ are non-trivial. By our assumptions on $G$, we see that the following are the only possibilities for $\dyn_G$ and $\theta$ (viewed as an automorphism of $\dyn_G$):
\begin{multicols}{2}
	\begin{enumerate}
		\item Type $B_{n}, n \geq 2, \theta = \id$.
		\item Type $C_n, n \geq 3, \theta = \id$.
		\item Type $D_n, n \geq 4, \theta = \id$.
		\item Type $D_n, n\geq 5, \theta$ has order $2$.
		\item Type $D_4, \theta$ has order $2$.
		\item Type $D_4, \theta$ has order $3$.
		\item Type $E_6, \theta = \id$.
		\item Type $E_6, \theta$ has order $2$.
		\item Type $E_7, \theta = \id$.
\end{enumerate} \end{multicols} In fact, the above are the only cases (apart from type $A$) where $\edyn_G$ has non-trivial automorphisms. Our proof of Proposition \ref{key-bound} will be based on this classification.

In \S \ref{sec:proof of key est} below, we prove many cases of the estimates (\ref{eq:key est 1}) and (\ref{eq:key est 2}). To be precise, we define an explicit subset $\Lambda (b)_{\mathrm{good}}$ of $\Lambda (b)$, and prove (\ref{eq:key est 1}) or (\ref{eq:key est 2}) for all $\lambda \in \Lambda (b)_{\mathrm{good}}$. These shall be done in a case-by-case manner according to the previous classification of $(\dyn_G, \theta)$.

In \S \ref{sec:part II}, we finish the proof of (\ref{eq:key est 1}) for all $\lambda \in \Lambda(b)$. We also define $\lambda_{\mathrm{bad}}$ in Proposition \ref{key-bound} (\ref{item:2}), and finish the proof of (\ref{eq:key est 2}) for all $\lambda \in \Lambda(b) - \set{\lambda _{\mathrm{bad}}}$. In view of what we have done in \S \ref{sec:proof of key est}, we only need to analyze elements $\lambda \in \Lambda(b) - \Lambda(b)_{\mathrm{good}}.$
If we are in the situation of Proposition \ref{key-bound} (\ref{item:1}), then we show (\ref{eq:key est 1}) for all $\lambda \in \Lambda(b) - \Lambda(b)_{\mathrm{good}}$. If we are in the situation of Proposition \ref{key-bound} (\ref{item:2}), then we define a distinguished element $\lambda_{\mathrm{bad}} \in \Lambda(b)$ and show (\ref{eq:key est 2}) for all $\lambda \in \Lambda(b) - \Lambda(b)_{\mathrm{good}} \cup \set{\lambda_{\mathrm{bad}}}$.

In \S \ref{sec:part III}, we construct the element $\mu_1$ in Proposition \ref{key-bound} (\ref{item:1}), and the elements $\mu_1, \mu_2$ in Proposition \ref{key-bound} (\ref{item:2}). We then check that they satisfy the desired properties. The proof of Proposition \ref{key-bound} is then finished.

\section{Proof of the key estimate, Part I}\label{sec:proof of key est} The goal of this section is to define an explicit subset $\Lambda (b)_{\mathrm{good}}$ of $\Lambda (b)$, and to prove (\ref{eq:key est 1}) or (\ref{eq:key est 2}) for all $\lambda \in \Lambda (b)_{\mathrm{good}}$.

\subsection{Types $B,C,D,\theta= \id$}

\subsubsection{The norm method}\label{subsubsec:classic}

We follow \cite[Chapitre VI \S 4]{bourbaki} for the presentation of the root systems of types $B_n,C_n,D_n$, and for the choice of simple roots. The root systems will be embedded in a vector space $E= \mathbb R^{n}$, with standard basis $e_1,\cdots, e_n$, and standard inner product $\langle e_i, e_j \rangle = \delta _{ij}$ so that we may identify the coroots and coweights with subsets of the same vector space. Following \textit{loc.~cit.}, we define the following lattices in $E$:
\begin{align*}
L_0  & : = \set{(\xi_1,\cdots, \xi_n)\in E \mid \xi_ i \in \Z }, \\\\
L_1 & : = \set{(\xi_1,\cdots, \xi_n)\in L_0 \mid \sum _{i=1}^n \xi_ i \in 2 \Z }, \\
L_2 & : = L_0 + \Z (\frac{1}{2}  \sum _{i=1}^n e _i).
\end{align*}

We assume $\theta = \id$, so that $T= A $ and $\hatS = \widehat T = \widehat A$. The cocharacter lattice $X_*(T)$ is identified with the coweight lattice in $E$. Moreover $\pi_1(G)$ is equal to the quotient of the coweight lattice modulo the coroot lattice in $E$.

Since $[b]\in B(G)$ is basic, it is uniquely determined by $\kappa_G(b) \in \pi_1(G)_{\sigma} = \pi_1 (G)$. The defect $\mathrm {def}_G (b)$ of $b$ is computed in the way indicated in the proof of Lemma \ref{lem:reduction to absolutely simple}.

For any $v=(\xi_ 1,\cdots, \xi _n)\in E$, we write
\begin{align}\label{eq:defn of norm}
|v|: =
|\xi_1|+\cdots +|\xi_n|. \end{align}
It is easy to verify the following three facts.
\begin{enumerate}
	\item $\abs{\cdot}$ is a norm on $E$.
	\item $|wv|=|v|$ for any $w\in W_0$ and $v \in E$.
	\item For any coroot $\alpha ^{\vee} \in \Phi ^{\vee}$, we have
	$ \abs{\alpha ^{\vee}} \leq \delta,$ where $\delta = 2$.
\end{enumerate}

Now given any subset $S$ of $\Lambda (b)$, we define $
\mathscr D (S) : = \min _{\lambda \in S} \abs{\lambda}.$ (The minimum obviously exists.)
In the following, we will specify a subset $\Lambda (b)_{\mathrm{good}}$ of $\Lambda (b)$, satisfying
\begin{align}\label{eq:adm bd}
\mathscr D (\Lambda (b)_{\mathrm{good}}) >  \delta \cdot \mathrm{def}_G(b)/2.
\end{align}
We show how to get the bound (\ref{eq:key est 1}) for all $\lambda \in \Lambda (b) _{\mathrm{good}}$, from (\ref{eq:adm bd}).

Let $\lambda\in \Lambda (b) _{\mathrm {good}} $.
Fix $w,w'\in W^1 $. We write  $\psi_s : = w \bullet (w'\lambda ^{(s)})$. By the formula (\ref{eq:defn of M}), it suffices to show that
\begin{align}\label{desired est}
\mathcal P ( \psi_s , q_1^{-1}) = O(q_1^{-s (\frac{1}{2}
	\mathrm{def}_G(b) +a ) })
\end{align} for some $a >0$.

By Proposition \ref{prop:interpretation} (2), we have the bound
\begin{align}\label{eq:total bd, classic}
\abs{\mathcal P(\psi_s, q_1^{-1})}  \leq \# {\mathbb P (\psi_s)}\cdot  q_1 ^{-N_s}, \quad N_s: = \min_{\underline m  \in \mathbb P (\psi_s)} \abs{\underline m} .\end{align}
Suppose $\underline m \in \mathbb P(\psi_s)$. Then
\begin{align*}
\delta \sum _{\beta \in \leftidx _F \Phi ^{\vee, + }} m (\beta) & \geq \abs{\sum _{\beta \in \leftidx_F \Phi ^{\vee, +}} m(\beta) \beta } = \abs{\psi_s} \\ &
\geq  \abs{ww'\lambda ^{(s)}} - \abs{w\rho^{\vee} -\rho ^{\vee}} =  s \abs{\lambda} -C \geq s \cdot \mathscr D (\Lambda (b)_{\mathrm{good}}) - C,
\end{align*}where $C$ is a constant independent of $s$, and $\abs{\cdot}$ is the norm defined in (\ref{eq:defn of norm}). Since $\theta = \id$, we have $\leftidx_F \Phi ^{\vee, +} = \Phi ^{\vee, +}$, and $\mathtt b (\beta) =1$ for all $\beta \in \Phi ^{\vee, +}$. Hence the leftmost term in the above inequalities is none other than $\delta \abs{\underline m}$. It follows that
$$N_s \geq (s \mathscr D (\Lambda (b) _{\mathrm{good}}) - C) \delta ^{-1}.$$ By the above estimate and (\ref{eq:adm bd}), we have
\begin{align}\label{eq:est of Ns}
q_1 ^{-N_s}  =  O(q_1^{-s (\frac{1}{2}
	\mathrm{def}_G(b) +a' ) })
\end{align}
for some $a' >0$.

On the other hand, by the same argument as in the proof of Proposition \ref{prop:est for unram}, we have
\begin{align}\label{eq:bd for number of partitions general}
\# \mathbb P(\psi_s)\leq (Ls) ^M
\end{align} for some constants $L, M >0$. The desired estimate (\ref{desired est}) then follows from (\ref{eq:total bd, classic}) (\ref{eq:est of Ns}) (\ref{eq:bd for number of partitions general}).

In the following we specify the definition of $\Lambda (b) _{\mathrm{good}}$ satisfying (\ref{eq:adm bd}), for types $B,C,D$ with $\theta = \id$.

\subsubsection{Type $B_n, n \geq 2, \theta = \id$}
The simple roots are $\alpha_i=e_i-e_{i+1}$ for $1\leq i \leq n-1$, and $\alpha_n=e_n.$ The simple coroots are $\alpha_i ^{\vee} = \alpha_i $ for $1 \leq  i \leq n-1$, and $\alpha_n^\vee = 2e_n.$ The fundamental weights are
$ \varpi_i = e_1 + \cdots + e_i$ for $1\leq i \leq n-1$, and $ \varpi_n = \frac{1}{2} (e_1 + \cdots + e_n).$
The coroot lattice is $L_1$, and the coweight lattice is $L_0$. We have
$$P^+ = \set{(\xi_1, \cdots, \xi_n)\mid \xi_i \in \Z, \xi_1 \geq \xi_2 \geq \cdots \geq \xi_n \geq 0}.$$
We have  $\pi_1(G)\cong\Z/2\Z$, and the non-trivial element is represented by $e_1 \in L_0$. Recall that we assumed that $\kappa_G(b)$ is non-trivial, so there is only one choice of $\kappa_G(b)$ (and hence only one choice of the basic $b\in B(G,\mu)$). We have $\lambda_b = -e_n,$ and $ \lambda_b^+ = e_1.$
Since $\kappa_G (b)$ acts on $\edyn_G$ via its unique non-trivial automorphism, we easily see (both for $n=2$ and for $n \geq 3$) that
$\mathrm{def}_G(b) = 1. $
We take $\Lambda (b) _{\mathrm{good}} := \Lambda (b),$ and we have $\mathscr D (\Lambda (b)) = 2$. The inequality (\ref{eq:adm bd}) is satisfied.

\subsubsection{Type $C_n, n \geq 3, \theta =\id $}
The simple roots are $\alpha_i=e_i-e_{i+1}$ for $1\leq i \leq n-1$, and $\alpha_n=2e_n.$ The simple coroots are $\alpha_i ^{\vee} = \alpha_i $ for $1 \leq  i \leq n-1$, and $\alpha_n^\vee = e_n.$ The fundamental weights are
$ \varpi_i = e_1 + \cdots + e_i$ for $1\leq i \leq n.$
The coroot lattice is $L_0$, the coweight lattice is $L_2$. We have
$$P^+ = \set{(\xi_1, \cdots, \xi_n)\in L_2\mid \xi_1 \geq \xi_2 \geq \cdots \geq \xi_n \geq 0}.$$
We have  $\pi_1(G)\cong\Z/2\Z$, and the non-trivial element is represented by $(\frac{1}{2}, \cdots , \frac{1}{2}) \in L_2$. Since $\kappa_G(b)$ is non-trivial, we have $$\lambda_b =(-\frac{1}{2}, \frac{1}{2}, - \frac{1}{2}, \cdots , (-1)^{n} \frac{1}{2}), \quad  \lambda_b^+ = (\frac{1}{2}, \cdots, \frac{1}{2}).$$
Since $\kappa_G(b)$ acts on the $\edyn_{G}$ via its unique non-trivial automorphism, we easily see that
$\mathrm{def}_G(b) = \lceil {\frac{n}{2} }\rceil$ (i.e.~the smallest integer $\geq n/2$).
We take $\Lambda (b) _{\mathrm{good}} := \Lambda (b)$, and we have $\mathscr D (\Lambda (b)) = (n+2)/2$. The inequality (\ref{eq:adm bd}) is satisfied.

\subsubsection{Type $D_n, n \geq 4,\theta =\id $} \label{subsubsec:Dn}

The simple roots are $\alpha_i=e_i-e_{i+1}$ for $1\leq i \leq n-2$, and $\alpha_{n-1}=e_{n-1}-e_n,\alpha_n= e_{n-1} +e_n.$
The simple coroots are $\alpha_i ^{\vee} = \alpha_i$. The fundamental weights are
\begin{align*}
\varpi_i & = e_1 + \cdots + e_i, ~ 1\leq i \leq n - 2 , \\
\varpi _{n-1} & = \frac{1}{2}(e_1 + e_2 + \cdots + e_{n-1} -e _n),  \\
\varpi _n & =  \frac{1}{2}(e_1 + e_2 + \cdots + e_n) .
\end{align*}
The coroot lattice is $L_1$, the coweight lattice is $L_2$. We have
$$P^+ = \set{(\xi_1, \cdots, \xi_n)\in L_2\mid \xi_1 \geq \xi_2 \geq \cdots \geq \xi_{n-1} \geq \abs{\xi_n}}.$$

\textbf{Case: $n$ is odd.} We have  $\pi_1(G)\cong\Z/4\Z$, and a generator is represented by $(\frac{1}{2}, \cdots , \frac{1}{2}) \in L_2$. For $i = 1,2,3$, we let $b_i \in B(G)$ correspond to the image of $i(\frac{1}{2}, \cdots , \frac{1}{2})$ in $\pi_1(G)$. Then
\begin{align*}
&\lambda_{b_1} = \sum_{i=1}^{n-2} \frac{(-1)^i} {2} e_i -\frac{1}{2}e_{n-1}   + \frac{(-1)^{(n+1)/2}}{2} e_n , &&  \lambda_{b_1}^+ = (\frac{1}{2}, \cdots, \frac{1}{2}) ,\\&
\lambda_{b_2} = -e_{n-1}, && \lambda_{b_2}^+ = e_1 ,\\&
\lambda_{b_3} = \sum_{i=1}^{n-2} \frac{(-1)^i} {2} e_i -\frac{1}{2}e_{n-1}   + \frac{(-1)^{(n-1)/2}}{2} e_n , &&   \lambda_{b_3}^+ = (\frac{1}{2}, \cdots, \frac{1}{2}, -\frac{1}{2}).
\end{align*}

Since up to automorphisms of $\Z/4\Z$, there is only one way that $\Z/4\Z$ could act on $\edyn_{G}$, we easily see that
$$\mathrm{def}_G(b_1) = \mathrm{def}_G(b_3) = \frac{n+3}{2}, ~ \mathrm{def}_G(b_2) = 2 .$$

Let
$$\lambda _{1,\mathrm{bad}} :=  (\frac{3}{2}, \frac{1}{2},\cdots, \frac{1}{2}, - \frac{1}{2}),\quad  \lambda_{3, \mathrm{bad}} : = (\frac{3}{2}, \frac{1}{2},\cdots, \frac{1}{2},  \frac{1}{2}).$$ For $i = 1,3$, we obviously have $\lambda _{i, \mathrm{bad}} \in \Lambda (b_i)$. We take
\begin{align}\label{eq:bad type D}
\Lambda (b_i) _{\mathrm{good}} : = \Lambda (b_i) - \set{\lambda _{i,\mathrm{bad}}}.
\end{align}
Then $ \mathscr D(\Lambda (b_i) _{\mathrm{good}}) = \frac{n+4}{2}$ and the inequality (\ref{eq:adm bd}) is satisfied.

For $i =2$, we take
$ \Lambda(b_2) _{\mathrm{good}}:= \Lambda (b_2),$ and we have $\mathscr D (\Lambda (b_2)) = 3.$
The inequality (\ref{eq:adm bd}) is satisfied.

\textbf{Case: n is even.} We have  $\pi_1(G)\cong\Z/2\Z \times \Z / 2\Z$. The three non-trivial elements are represented by
$$(\frac{1}{2}, \cdots , \frac{1}{2}), \quad e_1, \quad (\frac{1}{2}, \cdots , \frac{1}{2})+e_1 \in L_2. $$ Correspondingly we have
\begin{align*}
&\lambda_{b_1} = \sum_{i=1}^{n-2} \frac{(-1)^i} {2} e_i -\frac{1}{2}e_{n-1}   + \frac{(-1)^{n/2}}{2} e_n , &&   \lambda_{b_1}^+ = (\frac{1}{2}, \cdots, \frac{1}{2})  ,\\ &\lambda_{b_2} = -e_{n-1}, & &  \lambda_{b_2}^+ = e_1 ,\\ &\lambda_{b_3} = \sum_{i=1}^{n-2} \frac{(-1)^i} {2} e_i -\frac{1}{2}e_{n-1}   + \frac{(-1)^{n/2+1}}{2} e_n ,  &  & \lambda_{b_3}^+ = (\frac{1}{2}, \cdots, \frac{1}{2}, -\frac{1}{2}).
\end{align*}
Since $\kappa _G(b_1)$ and $\kappa_G(b_3)$ are related to each other by the automorphism of the based root system $e_{n} \mapsto -e_n$, it is clear that they correspond to the two horizontal symmetries of order two of $\edyn_{G}$. On the other hand, the action of $\kappa_G(b_2)$ on $\edyn_{G}$ is of order two, is distinct from the two horizontal symmetries, and commutes with the two horizontal symmetries. Hence this must correspond to the vertical symmetry of $\edyn_{G}$ that has precisely two orbits of size two and fixes all the other nodes. Thus we have
$$\mathrm{def}_G(b_1) = \mathrm{def}_G(b_3) = \frac{n}{2}, \quad  \mathrm{def}_G(b_2) = 2 .$$ For $i =1,2,3$ we take $\Lambda (b_i) _{\mathrm{good}} := \Lambda (b_i). $ Then we have
$\mathscr D (\Lambda (b_1)) = \mathscr D (\Lambda (b_3)) = (n+2)/2,$ and $ \mathscr D (\Lambda (b_2)) = 3.$
The inequality (\ref{eq:adm bd}) is satisfied.
\subsection{Type $D_n, n \geq 5, \theta$ has order $2$}\label{subsec:Dn nonsplit}
The simple (absolute) roots and coroots are the same as in \S \ref{subsubsec:Dn}, embedded in $E = \mathbb R^n$. We identify $E$ with $X_*(T) \otimes _{\Z} \mathbb R$. Then $\theta$ acts on $E$ by
$$(\xi_1, \cdots, \xi_n) \mapsto (\xi_1 ,\cdots, \xi_{n-1}, - \xi_n ). $$
The subgroup $X_*(A) \subset X_*(T) = L_2$ is given by
$\set{(\xi_1,\cdots, \xi_{n})  \in L_2  \mid \xi _n = 0}$. Let $L_2'\subset \Q ^{n-1}$ be the analogue of $L_2$, namely $L_2' = \Z ^{n-1} + \Z (\frac{1}{2}, \cdots, \frac{1}{2})$. The quotient $X_*(T) = L_2 \to X^*(\hatS)$ is the same as
$$ L_2  \To L_2' , ~ (\xi_1,\cdots, \xi_n) \mapsto (\xi_1, \cdots ,\xi_{n-1}). $$
The map
$$(s): X^*(\hatS) \To X_*(A),~ \lambda \longmapsto \lambda ^{(s)} $$ (for $s \in 2 \Z_{\geq 1}$) is given by
\begin{align}\label{eq:map (s)}
L_2' \To X_*(A), \quad  (\xi _1,\cdots, \xi_{n-1}) \longmapsto (s\xi_1, \cdots, s\xi_{n-1}, 0).
\end{align}

The set $\leftidx_F \Phi ^{\vee, +}$, as a subset of $X_*(A)$, is equal to
$$\set{e_i \pm e_j \mid 1\leq i < j \leq n-1} \cup \set{2e_i \mid 1\leq i \leq n-1} .$$ We have
$$\mathtt b ( e_i \pm e_j) =1, ~ 1\leq i< j \leq n-1; \quad \quad \mathtt b (2e_i) = 2, ~ 1\leq i  \leq n-1. $$ Moreover $\leftidx_F \Phi ^{\vee}$ is reduced. We have
$$ P^+ = \set{(\xi_1, \cdots, \xi_n)\in L_2\mid \xi_1 \geq \xi_2 \geq \cdots \geq \xi_{n-1} \geq \xi_n = 0  }, $$
$$X^*(\hatS) ^+ = \set{(\xi_1,\cdots, \xi_{n-1}) \in X^*(\hatS) =L_2' \mid \xi_1 \geq \xi_2 \geq \cdots \geq \xi_{n-1}\geq 0 }.$$

We again write $e_1,\cdots, e_{n-1}$ for the standard basis of $X^*(\hatS) \otimes \Q = \Q^{n-1}$. The relative simple roots in $\widehat Q_{\hat \theta} \subset X^*(\hatS)$ are:
$$ e_1-e_2, e_2 -e_3,\cdots, e_{n-2} - e_{n-1}, e_{n-1} $$ (i.e., the same as type $B_{n-1}$.)

\textbf{Case: $n$ is odd.} We have  $\pi_1(G)\cong\Z/4\Z$, and $\sigma $ acts on $\pi_1(G)$ by the unique non-trivial automorphism of $\pi_1 (G)$. Hence $\pi_1(G) _{\sigma} \cong \Z/2\Z$, and the non-trivial element is represented by $$-\frac{1}{2}(e_1 -e_2) - \frac{1}{2} (e_3 -e_4) \cdots - \frac{1}{2}(e_{n-2} - e_{n-1}) + \frac{1}{2}e_n \in L_2 = X_*(T).$$ The image of the above element in $X^*(\hatS) \otimes \Q = \Q^{n-1}$ is obviously equal to a linear combination of the relative simple roots in $\widehat Q_{\hat\theta}$ with coefficients in $\Q \cap (-1, 0]$. Hence this image is $\lambda_b$, and so
\begin{align*}
\lambda_b & = ( - \frac{1}{2}, \frac{1}{2}, \cdots , - \frac{1}{2}, \frac{1}{2}) \in \Q ^{n-1} = X^*(\hatS)\otimes \Q, \\*
\lambda_b ^+ & = (\frac{1}{2}, \cdots, \frac{1}{2}) \in \Q^{n-1} = X^*(\hatS)\otimes \Q .
\end{align*}

If $\gamma$ is any generator of $\pi_1 (G) \cong \Z/4\Z$, then the number of orbits of $\gamma \rtimes \theta$ in $\abs{\edyn_{G}}$ is $2+ \frac{n-3}{2}$, while the number of orbits of $\theta$ in $\abs{\edyn_{G}}$ is $n$. Hence
\begin{align}\label{eq:def b Dn odd nonsplit}
\mathrm{def}_G(b) =\frac{n-1}{2} .
\end{align}
We have \begin{align}\label{eq:Lambda(b) Dn odd nonsplit}
\Lambda(b) = \set{(\xi_1+\frac{1}{2},\cdots, \xi_{n-1}+\frac{1}{2}) \in L_2' \mid \xi _1 > 0, \forall i, \xi_i \in \Z,  \xi_i \geq \xi_{i+1} }.
\end{align}
We take $\Lambda (b) _{\mathrm{good}}: =\Lambda(b).$ In the following we show (\ref{eq:key est 1}) for all $\lambda \in \Lambda(b)$. The proof will be similar to the arguments in \S \ref{subsubsec:classic}.

Fix $w,w'\in W^1 $. We write $\psi_s : = w \bullet (w'\lambda ^{(s)})$. By the formula (\ref{eq:defn of M}), to show (\ref{eq:key est 1}) it suffices to show that
\begin{align}\label{desired est, Dn odd nonsplit} \exists a > 0 , ~
\mathcal P ( \psi_s , q_1^{-1}) = O(q_1^{-s (\frac{1}{2}
	\mathrm{def}_G(b) +a ) }).
\end{align} Again we have the bounds (\ref{eq:total bd, classic}) and (\ref{eq:bd for number of partitions general}), so it suffices to find a suitable lower bound of $\abs{\underline m}$, for $\underline m \in \mathbb P(\psi _s)$.
We keep the definition (\ref{eq:defn of norm}) of the norm $\abs{\cdot}$ on $E= \mathbb R^n$.
For $\underline m \in \mathbb P(\psi_s)$, we have
\begin{align*}
2\abs{\underline m} &  : = 2\sum _{\beta } m (\beta ) \mathtt b (\beta) \geq 2\sum _{\beta } m (\beta)  \geq \abs{\sum _{\beta } m(\beta) \beta } \\ &  = \abs{\psi_s} \geq  \abs{ww'\lambda ^{(s)}} - \abs{w\rho^{\vee} -\rho ^{\vee}} =  \abs{\lambda ^{(s)}} -C ,
\end{align*} where the sums are over $\beta \in \leftidx_F \Phi ^{\vee, +}$, and $C$ is a constant independent of $s$. On the other hand, by (\ref{eq:map (s)}) and (\ref{eq:Lambda(b) Dn odd nonsplit}), we have $\abs{\lambda ^{(s)}}  \geq s (n+1)/2.$ In conclusion we have
\begin{align}\label{eq:est for Dn odd nonsplit}
2 \abs{\underline m} \geq s\cdot \frac{n+1}{2} +C.
\end{align} Combining (\ref{eq:def b Dn odd nonsplit}) (\ref{eq:est for Dn odd nonsplit}) with (\ref{eq:total bd, classic}) (\ref{eq:bd for number of partitions general}), we obtain the desired (\ref{desired est, Dn odd nonsplit}). Note that in the above proof, we only used the fact that $(n+1)/2 > \mathrm{def}_G(b)$.

\textbf{Case: n is even.} We have  $\pi_1(G)\cong\Z/2\Z \times \Z / 2\Z$. The action of $\sigma $ on $\pi_1(G)$ swaps the classes represented by
$(\frac{1}{2}, \cdots , \frac{1}{2})$  and $ (\frac{1}{2}, \cdots , \frac{1}{2})+e_1 \in L_2$, and fixes the class represented by $e_1$. Hence $\pi_1(G) _{\sigma} \cong \Z/2\Z$, and the non-trivial element is represented by $$-\frac{1}{2}(e_1 -e_2) - \frac{1}{2} (e_3 -e_4) \cdots - \frac{1}{2}(e_{n-1} - e_{n}) \in L_2 = X_*(T).$$  The image of the above element in $X^*(\hatS) \otimes \Q = \Q^{n-1}$ is obviously equal to a linear combination of the relative simple roots in $\widehat Q_{\hat\theta}$ with coefficients in $\Q \cap (-1, 0]$. Hence this image is $\lambda_b$, and so
\begin{align*}
\lambda_b & = ( - \frac{1}{2}, \frac{1}{2}, \cdots , - \frac{1}{2}) \in \Q ^{n-1} = X^*(\hatS)\otimes \Q, \\
\lambda_b ^+  & = (\frac{1}{2}, \cdots, \frac{1}{2}) \in \Q^{n-1} = X^*(\hatS)\otimes \Q .
\end{align*}

Let $\gamma\in \pi_1 (G)$ be the class of $(\frac{1}{2},\cdots, \frac{1}{2}) \in L_2$. We have seen in \S \ref{subsubsec:Dn} that $\gamma$ acts on $\edyn_{G}$ via one of the two order-two horizontal symmetries of $\edyn_{G}$. Hence $\gamma\rtimes \sigma$ acts on $\edyn_{G}$ via one of the two order-four horizontal symmetries of $\edyn_{G}$,  and the number of orbits is $1+ \frac{n-2}{2} = \frac{n}{2}$. On the other hand the number of orbits of $\theta$ in $\abs{\edyn_{G}}$ is $n$. Hence
$\mathrm{def}_G(b) =n/2 .$
The set $\Lambda(b)$ is again given by (\ref{eq:Lambda(b) Dn odd nonsplit}).
We take $\Lambda (b) _{\mathrm{good}}: = \Lambda(b).$ The proof of (\ref{eq:key est 1}) for all $\lambda \in \Lambda(b)$ is exactly the same as in the odd case, using the fact that $(n+1)/2> \mathrm{def}_G(b)$.

\subsection{Type $D_4, \theta$ has order $2$}\label{subsec:Nonsplit-II}
The difference between this case and \S \ref{subsec:Dn nonsplit} is that the $D_4$ Dynkin diagram has three (rather than one) automorphisms of order two. However we explain why the proof of (\ref{eq:key est 1}) for all $\lambda \in \Lambda(b)_{\mathrm{good}} :=  \Lambda(b)$ is the same. In fact, there exists a permutation $\tau $ of $\set{1,3,4}$, such that the root system can be embedded into $\mathbb R^4$ with simple roots:
$$\alpha_{ \tau(1)} = e_1 - e_2, ~ \alpha _2 = e_2 -e_3, ~  \alpha _{\tau(3)} = e_3 -e _4, ~ \alpha _{\tau (4)} = e_3 + e_4,$$ and such that $\theta$ acts on $\mathbb R^4$ by $ e_4 \mapsto - e_4$.

If $\tau =1$, then the extra node in $\edyn_{G}$ is given by $\alpha_0 = -e_1 -e_2$, and the proof is exactly the same as \S \ref{subsec:Dn nonsplit}. For general $\tau$, we still have $\pi_1 (G) _{\sigma} \cong \Z/ 2\Z$ and hence a unique choice of $b$, and the only place in the proof in \S \ref{subsec:Dn nonsplit} that could change is the computation of $\mathrm{def}_G(b)$, as the extra node in $\edyn_{G}$ is no longer given by $-e_1 - e_2$. However, it is still true that as long as $\kappa_G(b)$ is the non-trivial element in $\pi_1 (G) _{\sigma} \cong \Z/2\Z$, we have
$ \mathrm{def}_G(b) = 2$, which is $n/2$ for $n=4$. In fact, for any order-two element $\gamma \in \pi_1(G) \cong \Z/2 \times \Z/2$ which is not fixed by $\sigma$, the action of $\gamma \rtimes \theta$ on $\abs{\edyn_{G}}$ must be of order four and have two orbits. The computation of $ \mathrm{def}_G(b)$ easily follows.

\subsection{Type $D_4, \theta$ has order $3$}\label{subsec:triality} In this case $\pi_1(G) = \Z/2\Z \times \Z/2\Z$. We know that $\theta$ acts  on $\pi_1(G)$ by an order-three permutation of the three non-trivial elements. Thus $\pi_1 (G)_{\sigma}= 0$ and any basic $b$ is unramified.

\subsection{Type $E_6, \theta = \id $} \label{subsec:E6 split} We consider the root system $E_6$ embedded in $\mathbb R^9$, which we will consider as $\mathbb{R}^3 \oplus \mathbb{R}^3 \oplus \mathbb{R}^3$. The set of roots is given by the 18 elements consisting of permutations of $$(1,-1,0;0,0,0;0,0,0),\quad (0,0,0;1,-1,0;0,0,0),\quad (0,0,0;0,0,0;1,-1,0)$$ under the group $S_3\times S_3\times S_3$, together with the 54 elements given  by the permutations of $$(\frac{2}{3},-\frac{1}{3},-\frac{1}{3};\frac{2}{3},-\frac{1}{3},-\frac{1}{3}:\frac{2}{3},-\frac{1}{3},-\frac{1}{3}),\quad (-\frac{2}{3},\frac{1}{3},\frac{1}{3};-\frac{2}{3},\frac{1}{3},\frac{1}{3}:-\frac{2}{3},\frac{1}{3},\frac{1}{3})$$ under the same group. We will call the first set of roots \emph{type $A$ roots}, and the second set \emph{type $B$ roots}. A type $A$ root is positive if and only if the coordinate $1$ appears to the left of the $-1$. A type $B$ root is positive if and only if the first coordinate is positive.

A choice of simple roots is given by
\begin{align*}
&\alpha_1=(0,0,0;0,1,-1;0,0,0), && \alpha_2=(0,0,0;1,-1,0;0,0,0), \\ & \alpha_3=(\frac{1}{3},-\frac{2}{3},\frac{1}{3};-\frac{2}{3},\frac{1}{3},\frac{1}{3};-\frac{2}{3},\frac{1}{3},\frac{1}{3}), && \alpha_4=(0,1,-1;0,0,0;0,0,0), \\ & \alpha_5=(0,0,0;0,0,0;1,-1,0), && \alpha_6=(0,0,0;0,0,0;0,1,-1).
\end{align*}
The corresponding Dynkin diagram is
$$
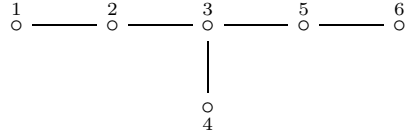
\xymatrix @R-1pc { \overset{1}{\circ } \ar@{-}[r] &  \overset{2}{\circ } \ar@{-}[r] &  \overset{3}{\circ } \ar@{-}[r] \ar@{-}[d]& \overset{5}{\circ }  \ar@{-}[r] &  \overset{6}{\circ }  \\ && \underset{4}{\circ} }$$
Under the standard pairing of $\mathbb{R}^9$ with itself, each root is equal to its own corresponding coroots. We therefore identify $\mathbb{R}^9$ with its dual and do not distinguish between roots and coroots. The subspace of $\mathbb{R}^9$ generated by the roots is given by the equations
\begin{align}\label{eq:rt subspace for E6}
x_1+x_2+x_3=x_4+x_5+x_6=x_7+x_8+x_9=0
\end{align}
where $x_i$ are the standard coordinates. The fundamental weights are given by
\begin{align*}
& \varpi_1=(\frac{2}{3},-\frac{1}{3},-\frac{1}{3};\frac{1}{3},\frac{1}{3},-\frac{2}{3};0,0,0), &&  \varpi_2=(\frac{4}{3},-\frac{2}{3},-\frac{2}{3};\frac{2}{3},-\frac{1}{3},-\frac{1}{3};0,0,0), \\ & \varpi_3=(2,-1,-1;0,0,0;0,0,0), && \varpi_4=(1,0,-1;0,0,0;0,0,0), \\ & \varpi_5=(\frac{4}{3},-\frac{2}{3},-\frac{2}{3};0,0,0;\frac{2}{3},-\frac{1}{3},-\frac{1}{3}), && \varpi_6=(\frac{2}{3},-\frac{1}{3},-\frac{1}{3};0,0,0;\frac{1}{3},\frac{1}{3},-\frac{2}{3}).
\end{align*}

For an element $\lambda=\sum_{i=1}^6a_i\varpi_i$ with $a_i \in \mathbb Z$, we have $\lambda$ lies in the root lattice if and only \begin{equation}\label{E6 root}a_5-a_6-a_2+a_1\equiv 0\mod 3. \end{equation}
We have $\pi_1(G)\cong \Z /3\Z$, with the isomorphism being given by $$\lambda=\sum_{i=1}^6a_i\varpi_i\longmapsto a_5-a_6-a_2+a_1\mod 3 . $$
Moreover $\lambda$ is dominant if and only $a_i\geq 0$ for $i=1,\dotsc,6$.

We let $b_i$, $i=1,2$ denote the non-trivial elements in $\pi_1(G)$. We have $\lambda_{b_1}^+ =\varpi_1$ and $\lambda_{b_2}^+=\varpi_6$. We set
\begin{align*}
\Lambda(b_1)_{\mathrm{good}} &  : =\Lambda(b_1)-\{\varpi_5,\varpi_4+\varpi_1,\varpi_2+\varpi_6,2\varpi_6\}, \\
\Lambda(b_2)_{\mathrm{good}} & : =\Lambda(b_2)-\{\varpi_2,\varpi_4+\varpi_6,\varpi_5+\varpi_1,2\varpi_1\}.
\end{align*}

We let $\abs{\cdot}$ be the standard Euclidean norm on $\mathbb R^9$. Then $\abs{\cdot}$ is $W_0$-invariant, and we have
$\abs{\alpha^{\vee}} \leq \delta: = \sqrt{2}$, for all $\alpha ^{\vee} \in \Phi ^{\vee}$. Given any subset $S$ of $\Lambda (b_i)$, we define
$\mathscr D (S) : = \min _{\lambda \in S} \abs{\lambda}.$ We claim that$$\mathscr{D}(\Lambda(b_1) _{\mathrm{good}})>\sqrt{8},\ \ \mathscr{D}(\Lambda(b_2)_{\mathrm{good}})>\sqrt{8}.$$
Since $\mathrm{def}_G(b)=4$ (which we know by counting orbits of the unique non-trivial symmetry of $\edyn_G$) and $\delta=\sqrt{2}$, the claim  will imply the inequality (\ref{eq:adm bd}), and by exactly the same argument as in \S \ref{subsubsec:classic}, we conclude that (\ref{eq:key est 1}) holds for all $\lambda \in \Lambda (b_i) _{\mathrm{good}}$.

We now prove the claim. By the obvious symmetry $(16)(25)(3)(4)$ of the Dynkin diagram, it suffices to only discuss $\Lambda(b_1)_{\mathrm{good}}$.

Let $\lambda=\sum_{i=1}^6a_i\varpi_i\in \Lambda(b_1),$ with $a_i \in \mathbb Z_{\geq 0}$, and suppose $|\lambda|\leq\sqrt{8}$. We will show $$|\lambda|\in\{\varpi_5,\varpi_4+\varpi_1,\varpi_2+\varpi_6,2\varpi_6\}.$$
Since $\lambda \in\Lambda(b_1)$, we have by (\ref{E6 root}) that \begin{align}\label{E6 root 1}a_5-a_6-a_2+a_1\equiv  1\mod 3. \end{align}By looking at the first three coordinates of $\lambda$ and using the triangle inequality, we easily obtain the inequalities $$a_1\leq 3,\ a_2 \leq 1,\ a_3\leq 1,\ a_4 \leq 3,\ a_5\leq 1,\ a_6 \leq 3.$$

If $a_3=1$, then we have $a_i>0$ for some $i\neq 3$ since $\lambda\in\Lambda(b_1)$, hence $a_3=0$.

If $a_2=1$, we have $a_5=0$ and $a_1,a_4,a_6\leq 1$ (by looking at the first 3 coordinates). We check each case and see that only $\lambda=\varpi_2+\varpi_6$ is possible.

If $a_5=1$, we similarly obtain that $ \lambda=\varpi_5$ is the only possibility using (\ref{E6 root 1}).

The only cases left are when the only non-zero coefficients are $a_1,a_4,a_6$. Again by looking at the first three coordinates, we see that $a_1+a_4+a_6\leq 3$.  We check each case and see that the only possibilities are $\lambda=\varpi_1+\varpi_4$ and $\lambda=2\varpi_6$.
\subsection{Type $E_6, \theta$ has order $2$}\label{subsec:E6 nonsplit}
We keep the notation \S \ref{subsec:E6 split}. Then $\theta$ acts on the root system via the action on $\mathbb{R}^9$ given by $$(x_1,x_2,x_3;x_4,x_5,x_6;x_7,x_8,x_9)\longmapsto (x_1,x_2,x_3;x_7,x_8,x_9;x_4,x_5,x_6).$$
It therefore acts on $\pi_1(G)$ by switching the two non-trivial elements. Hence $\pi_1(G)_\sigma=0$ and all basic elements are unramified.

\subsection{ Type $E_7,\theta= \id$} \label{subsec:E7} We consider the root system $E_7$ as a subset of $\mathbb{R}^8$. The set of roots is given by the 56 permutations of $(1,-1,0,0,0,0,0,0)$ and the $8\choose 4$ permutations of $(\frac{1}{2},\frac{1}{2},\frac{1}{2},\frac{1}{2},-\frac{1}{2},-\frac{1}{2},-\frac{1}{2},-\frac{1}{2}).$
A set of simple roots is given by
\begin{align*}
&\alpha_1=(0,0,0,0,0,0,-1,1),&& \alpha_2=(0,0,0,0,0,-1,1,0),\\& \alpha_3=(0,0,0,0,-1,1,0,0),&& \alpha_4=(0,0,0,-1,1,0,0,0), \\
&\alpha_5=(\frac{1}{2},\frac{1}{2},\frac{1}{2},\frac{1}{2},-\frac{1}{2},-\frac{1}{2},-\frac{1}{2},-\frac{1}{2}), &&  \alpha_6=(0,0,-1,1,0,0,0,0) ,\\
&\alpha_7=(0,-1,1,0,0,0,0,0).
\end{align*}
The corresponding Dynkin diagram is
$$
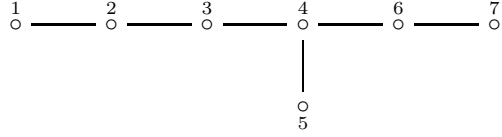
\xymatrix @R-1pc { \overset{1}{\circ } \ar@{-}[r] &  \overset{2}{\circ }\ar@{-}[r] &  \overset{3}{\circ }\ar@{-}[r] &  \overset{4}{\circ } \ar@{-}[r] \ar@{-}[d] & \overset{6}{\circ } \ar@{-}[r] &  \overset{7}{\circ } \\ & && \underset{5}{\circ}}$$

Under the standard pairing of $\mathbb{R}^8$ with itself, roots correspond to coroots and we therefore do not distinguish between them. The subspace of $\mathbb{R}^8$ generated by the roots is the hyperplane given by the equation $\sum_{i=1}^8x_i=0$.

The corresponding fundamental weights are given by
\begin{align*}
\varpi_1 & =(\frac{3}{4},-\frac{1}{4},-\frac{1}{4},-\frac{1}{4},-\frac{1}{4},-\frac{1}{4},-\frac{1}{4},\frac{3}{4}), \\  \varpi_2& =(\frac{3}{2},-\frac{1}{2},-\frac{1}{2},-\frac{1}{2},-\frac{1}{2},-\frac{1}{2},\frac{1}{2},\frac{1}{2}), \\  \varpi_3 & =(\frac{9}{4},-\frac{3}{4},-\frac{3}{4},-\frac{3}{4},-\frac{3}{4},\frac{1}{4},\frac{1}{4},\frac{1}{4}), \\  \varpi_4 &=(3,-1,-1,-1,0,0,0,0), \\ \varpi_5 & =(\frac{7}{4},-\frac{1}{4},-\frac{1}{4},-\frac{1}{4},-\frac{1}{4},-\frac{1}{4},-\frac{1}{4},-\frac{1}{4}), \\ \varpi_6 & =(2,-1,-1,0,0,0,0,0), \\   \varpi_7 & =(1,-1,0,0,0,0,0,0).
\end{align*}
For an element $\lambda=\sum_{i=1}^7a_i\varpi_i, a_i\in\mathbb{Z}$, we know  $\lambda$ lies in the root lattice if and only if $a_1+a_3+a_5\equiv 0\mod 2$. We have $\pi_1(G)\cong\mathbb{Z}/2\mathbb{Z}$. By assumption $\kappa_G(b) \in \pi_1(G)$ is the non-trivial element. Then we have $\lambda_b^+ =\varpi_1.$
We set $\Lambda(b)_{\mathrm{good}}:=\Lambda(b)-\{\varpi_5\}.$

We let $\abs{\cdot}$ be the standard Euclidean norm on $\mathbb R^8$. Then $\abs{\cdot}$ is $W_0$-invariant, and we have
$\abs{\alpha^{\vee}} \leq \delta: = \sqrt{2}$, for all $\alpha ^{\vee} \in \Phi ^{\vee}$. Given any subset $S$ of $\Lambda (b)$, we define
$\mathscr D (S) : = \min _{\lambda \in S} \abs{\lambda}.$ We claim that $$\mathscr{D}(\Lambda(b)_{\mathrm{good}})\geq\frac{\sqrt{22}}{2}.$$
Since $\mathrm{def}_G(b)=3$ and $\delta=\sqrt{2}$, the claim  will imply the inequality (\ref{eq:adm bd}), and by exactly the same argument as in \S \ref{subsubsec:classic}, we conclude that (\ref{eq:key est 1}) holds for all $\lambda \in \Lambda (b) _{\mathrm{good}}$.

We now prove the claim. Suppose $\lambda=\sum_{i=1}^7a_i\lambda_i\in\Lambda(b)$ with $a_i \in \Z_{\geq 0}$, and $|\lambda|\leq\sqrt{22}/2$. We will show that $\lambda=\varpi_5$. By looking at the first four coordinates, we obtain the trivial (in)equalities:
$$\varpi_1\leq 2,\ \varpi_2\leq 1,\ \varpi _3=0,\ \varpi_4=0,\ \varpi_5\leq 1,\ \varpi_6\leq 1,\ \varpi_7\leq 1.$$
We also obtain $\sum_{i=1}^7a_i\leq 2$. It is not hard to see that $\lambda=\varpi_5$ is the only possibility.

\section{Proof of the key estimate, Part II} \label{sec:part II} The goals of this section include:
\begin{itemize}
	\item  to finish the proof of (\ref{eq:key est 1}) in Proposition \ref{key-bound} (\ref{item:1});
	\item to define $\lambda_{\mathrm{bad}}$ and to prove (\ref{eq:key est 2}) in Proposition \ref{key-bound} (\ref{item:2}).
\end{itemize}

In \S \ref{sec:proof of key est}, we already proved (\ref{eq:key est 1}) and (\ref{eq:key est 2}) for all $\lambda \in \Lambda (b) _{\mathrm{good}}$. Moreover the subset $\Lambda (b) _{\mathrm{good}} \subset  \Lambda(b)$ is proper only in the following three cases:
\begin{description}
	\item[Proper-I] Type $D_n, n \geq 5, n $ is odd, $\theta = \id$, $b = b_1 $ or $b_3$. See \S \ref{subsubsec:Dn}.
	\item[Proper-II] Type $E_6, \theta = \id , b = b_1$ or $b_2$. See \S \ref{subsec:E6 split}.
	\item[Proper-III] Type $E_7, \theta = \id$. See \S \ref{subsec:E7}.
\end{description}
\subsection{Combinatorics for $D_n$} \label{subsec:combinatorics for Dn} In order to treat the case \textbf{Proper-I}, we need some combinatorics for the type $D_n$ root system. The material in this subsection is only needed in the proof of Proposition \ref{prop:proper-I} below.

Let $n$ be an integer $\geq 5$. We keep the presentation of the type $D_n$ root system in a vector space $\mathbb R^n$, as in \S \ref{subsubsec:Dn}. In particular we keep the choice of positive roots. We do not distinguish between roots and coroots. Let $\Phi _{D_n}$ be the set of roots and let $\Phi^+_{D_n}$ be the set of positive roots. Thus $$\Phi ^+_{D_n} = \set{e_i \pm e_j \mid 1\leq i < j \leq n}.$$ If $m>n$ is another integer, we embed $\mathbb R^{n}$ into $\mathbb R^{m} $ via the inclusion of the standard bases  $\set{e_1,\cdots, e_n}\hookrightarrow \set{e_1,\cdots, e_{m}}$. In this way we view $\Phi _{D_n}$ (resp.~$\Phi ^+_{D_n}$) as a natural subset of $\Phi_{D_{m}}$ (resp.~$\Phi ^+_{D_{m}}$).

In the following we fix an odd integer  $n \geq 5$. We keep the notation in \S \ref{subsec:Kostant partitions}, with respect to  $\leftidx_F \Phi ^{\vee,+} =\Phi_{D_{n}}^+$ and $\mathtt b \equiv 1$. The goal of this subsection is to prove the following:

\begin{prop}\label{prop:cancellation for Dn odd} Let $(\nu_2,\cdots, \nu_n) \in \set{\pm 1} ^{n-1}$. For $t\in \mathbb N $, let $$\lambda_t: = (6t, 2t\nu_2,2 t \nu_3\cdots, 2t\nu_n).$$ When $t\in \mathbb N$ is sufficiently large, for any integer $L$ in the interval $[0, (n+3.5)t]$, we have
	\begin{align}\label{eq:cancallation}
	\sum _{S \subset \Phi^+_{D_n}} (-1)^{\abs{S}} \# \mathbb P(\lambda_t - \sum _{\beta \in S} \beta)  _L = 0. \end{align}
\end{prop}

In order to prove Proposition \ref{prop:cancellation for Dn odd}, we shall use some graph theory to facilitate the computation of the left hand side of (\ref{eq:cancallation}). We first recall some standard terminology from graph theory. Recall that a \emph{graph} consists of vertices and edges, such that each edge links two distinct vertices. A \emph{path} on a graph is a sequence $(v_1, E_1, \cdots, v_k, E_k, v_{k+1})$, where $v_i$ and $v_{i+1}$ are the two distinct vertices linked by the edge $E_i$, for each $i$, and such that $E_1, \cdots , E_k$ are all distinct. A \emph{tree} is a graph, in which any two vertices are connected by exactly one path. In particular, on a tree we may represent each path $(v_1, E_1, \cdots, v_k, E_k, v_{k+1})$ unambiguously by the sequence $(v_1, v_2 ,\cdots, v_{k+1})$, and we know that $v_1,\cdots, v_{k+1}$ are all distinct. A \emph{rooted tree} is a pair $(\mathscr T, O)$, where $\mathscr T$ is a tree and $O$ is a distinguished vertex of $\mathscr T$, called the \emph{root}.

In a rooted tree $(\mathscr T, O)$, every vertex $v \neq O$ has a unique \emph{parent}, that is, the vertex $w$ such that the unique path from $O$ to $v$ passes through $w$ immediately before reaching $v$. In this situation we call $v$ a \emph{child} of $w$. A vertex of $(\mathscr T, O)$ that does not have children is called an \emph{end vertex}. We denote by $\abs{\mathscr T}$ the set of vertices, and denote by $\abs{\mathscr T}_{\mathrm{end}}$ the set of end vertices. By a \emph{complete family line} on $(\mathscr T, O)$, we shall mean a path $(v_0, v_1, \cdots, v_k)$ such that $v_0 = O$ and $v_k \in \abs{\mathscr T}_{\mathrm{end}}$. Note that in this case each $v_i$ is the parent of $v_{i+1}$.

In the following, we fix an abelian group $E$, and fix subsets $U , V$ of $E$, such that $U$ is finite and $0 \notin V$. In applications, we shall take $E = \mathbb R^n$, $U = \set{e_1 + \nu_j e_j \mid 2\leq j \leq n}$ for some $(\nu_2,\cdots, \nu_n) \in \set{\pm 1}^{n-1}$, and $V = \Phi^+_{D_n}$.
\begin{defn}
	By a \emph{$(U,V)$-tree}, we mean a tuple $(\mathscr T, O, \phi,\psi)$, where $(\mathscr T,O)$ is a finite rooted tree, and $\phi, \psi$ are maps
	\begin{align*}
	\phi  & : \abs{\mathscr T}-\set{O} \To U , \\
	\psi & : \abs{\mathscr T}-\abs{\mathscr T}_{\mathrm{end}} \To V .
	\end{align*} We shall graphically represent a $(U,V)$-tree $(\mathscr T, O, \phi,\psi)$ by marking $$\boxed{\phi(v) \parallel \psi(v)}$$ at each vertex $v$. If $\phi(v)$ or $\psi(v)$ is not defined, we leave it as blank. (Thus $O$ could be recognized as the unique vertex where $\phi(v)$ is left as blank.)
\end{defn}

\begin{defn}\label{defn:adm tree} We say that a $(U,V)$-tree $(\mathscr T, O, \phi,\psi)$ is \emph{admissible}, if the following conditions are satisfied:
	\begin{enumerate}
		\item Each $w \in \abs{\mathscr T} - \abs{\mathscr T}_{\mathrm{end}}$ has precisely two children $v, \bar v$. Moreover, we have $\phi(v) - \phi(\bar v) \in \set{\psi (w), - \psi(w)}$.
		\item Every complete family line $(v_0, v_1 ,\cdots, v_k)$ on $(\mathscr T, O)$ satisfies $k = \abs{U}-1$. Moreover, $\phi(v_i)$ are distinct from each other for $1\leq i \leq k$, and $\psi(v_i)$ are distinct from each other for $0 \leq i \leq k-1$.
	\end{enumerate}
\end{defn}
\begin{defn}
	Let $(\mathscr T, O,\phi,\psi)$ be an admissible $(U,V)$-tree. Let $w \in \abs{\mathscr T} - \abs{\mathscr T}_{\mathrm{end}}$. Since $\psi(w)\neq 0$, there is a unique ordering of the two children $v,\bar v$ of $w$, such that $\phi(v) - \phi(\bar v) = \psi(w)$. We call $v$ a \emph{positive vertex} and call $\bar v$ a \emph{negative vertex}.
\end{defn}

\begin{prop}\label{prop:const of tree}
	For each odd integer $n \geq 5$ and each $(\nu_2,\cdots, \nu_n) \in \set{\pm 1} ^{n-1}$, we let $E: = \mathbb R^n$, $U:= \set{e_1 +\nu_j e_j \mid 2\leq j \leq n}$, and $V := \Phi^+_{D_n}$. Then there exists an admissible $(U,V)$-tree.
\end{prop}
\begin{proof}
	We prove this by induction on $n$. To simplify notation, when $n$ and $(\nu_2,\cdots, \nu_n) $ are fixed, we define
	\begin{align*}
	g_j : = e_1 + \nu_j e_j ,\quad  h_{i,j} : = e_i - \nu_i \nu_j e_j,
	\end{align*} for
	$i , j \in \set{2,\cdots, n}$.
	
	The base chase is $n=5$. One checks that the $(U, V)$-tree represented by the following diagram is admissible:
	$$  \setlength{\qtreepadding}{1 pt}\Tree [.$\boxed{ \parallel  h_{2,3}
	}$ [.$\boxed{{g_2}\parallel { h_{3,4}}}$ [.$\boxed{{g_3}\parallel {h_{4,5}}}$ [.$\boxed{g_4 \parallel  }$ ] [.$\boxed{g_5 \parallel  }$ ] ] [.$\boxed{{g_4} \parallel {h_{3,5}}}$ [.$\boxed{g_3 \parallel }$ ] [.$\boxed{g_5 \parallel  }$ ] ] ] [.$\boxed{{g_3} \parallel {h_{2,4}}}$ [.$\boxed{{g_2} \parallel {h_{4,5}}}$ [.$\boxed{g_4\parallel  }$ ] [.$\boxed{g_5\parallel }$ ] ] [.$\boxed{{g_4} \parallel { h_{2,5}}}$ [.$\boxed{g_2 \parallel  }$ ] [.$\boxed{g_5 \parallel  }$ ] ] ] ] $$
	
	Now assuming the proposition is proved for $n-2$, we prove it for $n$ (with $n \geq 7$). Let $(\nu_2,\cdots, \nu_n)$, $U$, and $V$ be as in the proposition. Let $U': = \set{g_j \mid 2\leq j \leq n-2}$, and $V': = \Phi^+_{D_{n-2}}$. By induction hypothesis there exists an admissible $(U', V')$-tree $(\mathscr T', O, \phi',\psi')$. We shall construct an admissible $(U, V)$-tree based on this.
	
	First, to each end vertex $v$ of $\mathscr T'$, we associate a $(U,V)$-tree $(\mathscr T_v, O_v, \phi_v, \psi_v)$ as follows. Let $(O,v_1,\cdots, v_k =v)$ be the complete family line on $\mathscr T'$ from $O$ to $v$. Since $(\mathscr T', O, \phi',\psi')$ is admissible, we have $k = n-4$, and there is a unique index $j_v \in \set{2,\cdots ,n-2}$ such that $$\set{\phi' (v_i) \mid 1 \leq i \leq k} = \set{g_j \mid 2 \leq j \leq n-2, j \neq j_v}.$$
	We define $(\mathscr T_v, O_v, \phi_v, \psi_v)$ to be:
	$$  \Tree [.$\boxed{ \parallel  h_{j_v,n}}$ [.$\boxed{g_{j_v} \parallel  h_{n-1, n}}$ [.$\boxed{g_{n-1} \parallel }$ ] [.$\boxed{g_{n} \parallel }$ ] ] [.$\boxed{g_{n} \parallel  h_{j_v, n-1}}$ [.$\boxed{g_{j_v}\parallel }$ ] [.$\boxed{g_{n-1} \parallel }$ ] ] ] $$
	
	Now for each end vertex $v$ of $\mathscr T'$, we glue $\mathscr T_v$ to $\mathscr T'$ by identifying the root $O_v$ of $\mathscr T_v$ with $v$. We also combine the marking $\boxed{\phi'(v) \parallel}$ at $v$ with the marking $\boxed{\parallel \psi_v(O_v)}$ at $O_v$ to get the new marking $\boxed{\phi'(v) \parallel \psi_v(O_v)}$. In this way we obtain a larger tree $\mathscr T$ containing $\mathscr T'$, and a $(U,V)$-tree $(\mathscr T, O, \phi,\psi)$.
	
	It is then an elementary exercise to check that the $(U,V)$-tree $(\mathscr T, O, \phi,\psi)$ thus obtained is admissible. As an example, we check that for each complete family line $(v_0, v_1 ,\cdots , v_k)$ on $(\mathscr T,O)$, the elements $\psi(v_i)$ are distinct for $0 \leq i \leq k-1$. The other desired conditions can all be checked similarly. Note that $(v_0, v_1 ,\cdots , v_{k-2})$ is a complete family line on $(\mathscr T',O)$. For $0\leq i \leq k-3$, we have $\psi(v_i) = \psi'(v_i)$, and these are distinct from each other by induction hypothesis. Now writing $v$ for $v_{k-2}$, we have $\psi (v_{k-2})  = \psi_{v}(O_{v}),$ and $\psi(v_{k-1})  =  \psi_{v} (v_{k-1})$. By the definition of $\psi_v$, we see that the set formed by $\psi(v_{k-2})$ and $\psi(v_{k-1})$ is either $\set{h_{j_v, n}, h_{n-1,n}}$ or $\set{h_{j_v, n}, h_{j_v, n-1}}$. Thus these two elements are distinct from each other, and being in $V - V'$, they are distinct from $\psi(v_i)$ for $0\leq i \leq k-3$, since the latter elements are in $V'$. Thus we have shown that $\psi(v_i)$ are distinct for $0 \leq i \leq k-1$. \end{proof}

Fix an odd integer $n \geq 5$, and fix $(\nu_2,\cdots, \nu_n) \in \set{\pm 1} ^{n-1}$. Let $U, V$ be as in Proposition \ref{prop:const of tree}, and fix an admissible $(U, V)$-tree $(\mathscr T, O, \phi,\psi)$ as in that proposition. Let $2^{V}$ denote the power set of $V$.

For each $v \in \abs{\mathscr T} - \set{O}$, we define a subset $C_v \subset 2 ^{V}$ as follows. Let $w$ be the parent of $v$.
We define
$$C_v := \begin{cases} \set{S \in 2^{V} \mid \psi(w) \notin S }, & \text{if $v$ is a positive vertex}, \\
\set{S \in 2^{V} \mid \psi(w) \in S }, & \text{if $v$ is a negative vertex}.	\end{cases} $$
Now let $(O, v_1, \cdots, v_k = v)$ be the unique path on $\mathscr T$ from $O$ to $v$. We define
$$D_v: = \bigcap_{i=1}^k C_{v_i} \subset 2^V.$$ We also define $D_O:= 2^V.$  It is easy to see that if $v, \bar v$ are the two children of any vertex $w$, with $v$ positive, then
\begin{align}\label{eq:Dv}
\begin{aligned}
D_v &= \set{S \in D_w \mid \psi(w) \notin S}  , \quad  D_{\bar v} = \set{ S \in D_w \mid \psi (w) \in S}, \\
D_w &= D_v \sqcup D_{\bar v}.
\end{aligned}
\end{align}In this situation, there is a bijection
\begin{align}\label{eq:bij adding a root}
D_v \isom D_{\bar v}, \quad S \longmapsto S \cup \set{\psi (w)}.
\end{align}

Next, for any $\lambda$ in the root lattice $\mathrm{span}_{\Z} (\Phi_{D_n})$, any $L \in \mathbb Z_{\geq 0}$, and any $v\in \abs{\mathscr T}$, we define a subset $\mathbb P(\lambda)_L^v$ of $\mathbb P(\lambda)_L$. (See \S \ref{subsec:Kostant partitions} for the definition of $\mathbb P(\lambda)_L$.) If $v \neq O$, let $(O, v_1, \cdots, v_k = v)$ be the unique path from $O$ to $v$. We define
$$\mathbb P(\lambda)^{v} _L : = \set {\underline m \in \mathbb P(\lambda) _L\mid m(\phi(v_i)) =  0, ~\forall 1\leq i \leq k }.$$
If $v = O$, we define $$\mathbb P(\lambda)^O_L : = \mathbb P (\lambda)_L.$$
For any two $v_1, v_2 \in \abs{\mathscr T}$, we define
$$ \mathbb P(\lambda) ^{v_1-v_2}_{L}: = \mathbb P(\lambda)_L^{v_1} - \mathbb P(\lambda) _L^{v_2}. $$
\begin{lem}\label{lem:claim1} Let $\lambda$ be an element in the root lattice $\mathrm{span}_{\Z} (\Phi_{D_n})$. For each subset $S \subset \Phi^+_{D_n}$, let $\lambda_S : = \lambda - \sum_{\beta \in S} \beta$. Let $w \in \abs{\mathscr T} - \abs{\mathscr T}_{\mathrm{end}}$, and let $v, \bar v$ be its two children. Then
	$$ \sum _{S \in D_w} (-1)^{\abs{S}} \# \mathbb P (\lambda_S)_L^{w} =  \sum _{S \in D_v} (-1)^{\abs{S}} \#\mathbb P (\lambda_S)_L^{v} + \sum _{S \in D_{\bar v}} (-1)^{\abs{S}} \#\mathbb P (\lambda_S)_L^{\bar v}.$$	
\end{lem}
\begin{proof}
	We may assume $v$ is positive. In view of (\ref{eq:Dv}) and the bijection (\ref{eq:bij adding a root}), it suffices to show that for each $S\in D_v$ we have
	\begin{align}\label{eq:to show for claim 1}
	\#\mathbb P(\lambda_S) _{L} ^{w-v}  = \#\mathbb P(\lambda_S - \psi(w)) _{L} ^{w-\bar v}.
	\end{align}
	To show this we consider the map
	\begin{align*}
	f: \mathbb P(\lambda_S) _{L} ^{w-v}   \To \mathbb P(\lambda_S - \psi(w)) _{L} ^{w-\bar v} \end{align*} defined by $$f (\underline m)(\beta) := \begin{cases} m(\beta) -1,& \text{if }\beta = \phi (v) , \\
	m(\beta) +1,&  \text{if } \beta = \phi (\bar v), \\
	m(\beta) , &  \text{else}.
	\end{cases} $$
	One easily checks that $f$ is well-defined and bijective, using the admissibility of $(\mathscr T, O, \phi,\psi)$. This proves (\ref{eq:to show for claim 1}), and hence the lemma.
\end{proof}

\begin{lem}\label{lem:bdry case}
	Let $\abs{\cdot}$ be the norm on $ \mathbb R^n$ as in (\ref{eq:defn of norm}). Fix a real number $M > n /2$.  Let $(\nu_2,\cdots, \nu_n) \in \set{\pm 1} ^{n-1}$, and $t\in \mathbb N$. Let $\lambda\in \mathrm{span}_{\Z} (\Phi_{D_n})$, such that
	\begin{align}\label{eq:close to lambda t}
	\abs{\lambda - (6t, 2t\nu_2,2 t \nu_3\cdots, 2t\nu_n)} < t/M.
	\end{align} Let $\underline m \in \mathbb P(\lambda)$. Assume there is a subset $I \subset \set{2,\cdots, n}$ of cardinality $n-2$, such that $m(e_1 + \nu _ i e_i) = 0$ for all $i \in I$. Then $$\abs{\underline m} \geq  (n+4 - \frac{n}{2M}) t . $$
\end{lem}
\begin{proof}
	Assume the contrary. Let $j_0$ be the unique element of $ \set{2,\cdots,n} - I$. Define $\underline m' \in \mathbb P$ by:
	$$\forall \beta \in \Phi^+_{D_n}, ~  m' (\beta) :=\begin{cases}
	0, & \text{if }\beta \in \set{e_1 \pm e_i \mid  2 \leq i \leq n } ,\\
	m(\beta), & \text{else}.
	\end{cases}$$
	Write $\lambda = \lambda_1 e_1 + \sum _{i=2}^n \lambda_i \nu_i e_i, $ with $\lambda_1,\cdots, \lambda_n \in \mathbb R$. From our assumption (\ref{eq:close to lambda t}), it easily follows that each $\lambda_i > 0$. We have $$\Sigma (\underline m') =  \sum _{i\in I} (\lambda_i +m(e_1 -\nu_i e_i)) \nu_i e_i + \lambda'_{j_0}\nu_{j_0}e_{j_0},$$ with some $\lambda' _{j_0}\in \mathbb R$.
	Obviously $ \abs{\underline m'} = \abs{\underline m} - \lambda _1,$ so we have
	$$ 2 \abs{\underline m'} = 2 \abs{\underline m} - 2 \lambda_1 \geq \abs{\Sigma (\underline m')} \geq   \sum _{i\in I} \lambda _i+m (e_1 -\nu_i e_i), $$ from which we get:
	\begin{align*}
	0 \leq  & \sum _{i\in I}m (e_1 -\nu_i e_i) \leq 2 \abs{\underline m} - 2 \lambda_1 - \sum _{i\in I} \lambda_i  \\
	& < 2 (n+4 - \frac{n}{2M} ) t - 2 (6t- t/M) - \abs{I} (2t -t/M) = 0,
	\end{align*}
	a contradiction.
\end{proof}
We are now ready to prove Proposition \ref{prop:cancellation for Dn odd}.
\begin{proof}[Proof of Proposition \ref{prop:cancellation for Dn odd}] As in the proposition, we fix $(\nu_2,\cdots,\nu_n)$. We let $\lambda_S : = \lambda_t - \sum _{\beta \in S} \beta$, for any subset $S \subset \Phi^+_{D_n}$. Let $U:= \set{e_1 +\nu_j e_j \mid 2\leq j \leq n}$, and let $V := \Phi^+_{D_n}$. Fix an admissible $(U, V)$-tree $(\mathscr T, O, \phi,\psi)$ as in Proposition \ref{prop:const of tree}. By repeatedly applying Lemma \ref{lem:claim1} (with respect to $\lambda : = \lambda _t$), we obtain
	\begin{align*}
	\sum _{S\subset \Phi^+_{D_n}} (-1) ^{\abs{S}} \#\mathbb P(\lambda_S) _L  & = \sum _{S\in D_O} (-1) ^{\abs{S}} \#\mathbb P(\lambda_S)^O _L  \\* &  = \sum _{v \in \abs{\mathscr T}_{\mathrm{end}}} \sum _{S\in D_v} (-1) ^{\abs{S}} \#\mathbb P(\lambda_S)^v _L .
	\end{align*}
	Hence the proposition is proved once we show the following claim:
	
	\textbf{Claim.} When $t$ is sufficiently large, for each $L \in \Z \cap [0,3.5t] $ and for each $v \in \abs{\mathscr T}_{\mathrm{end}}$, we have
	$$ \sum _{S\in D_v} (-1) ^{\abs{S}} \#\mathbb P(\lambda_S)^v _L = 0. $$
	
	To prove the claim, we fix a real number $M > n$. Let $\abs{\cdot}$ be the norm on $\mathbb R^n$ as in (\ref{eq:defn of norm}). When $t$ is sufficiently large, we have
	$$\max _{S \subset \Phi^+_{D_n}}\abs{ \sum_{\beta \in S} \beta} < t/M.$$ Thus we can apply Lemma \ref{lem:bdry case} to $
	\lambda : = \lambda_S$ for each $S \subset \Phi^+_{D_n} $. Using the admissibility of $(\mathscr T, O, \phi,\psi)$, we know that any $\underline m \in \mathbb P(\lambda_S) ^v _L$ satisfies the hypothesis in Lemma \ref{lem:bdry case} about the vanishing of $m(e_1 +\nu_i e_i)$. Thus by that lemma, $\mathbb P(\lambda_S) ^v _L$ is non-empty only if $$L \geq (n+4 - \frac{n}{2M}) t > (n+3.5)t. $$
	This proves the claim. The proof of the proposition is complete.
\end{proof}

\subsection{The case Proper-I} We treat type $D_n$ with $ n \geq 5$ odd, and $\theta = \id$, $b = b_1 $ or $b_3$. See \S \ref{subsubsec:Dn}. By symmetry we only need to consider $b= b_1$. Recall in this case $\Lambda (b_1)  - \Lambda (b_1) _{\mathrm{good}} = \set{\lambda _{1,\mathrm{bad}}}$, where $$\lambda_{1,\mathrm{bad}} = (\frac{3}{2}, \frac{1}{2},\cdots, \frac{1}{2}, - \frac{1}{2}),$$ and we have $\mathrm{def}_G(b_1) = \frac{n+3}{2}$.
\begin{prop}\label{prop:proper-I}
	The bound (\ref{eq:key est 1}) holds for $\lambda =  \lambda _{1,\mathrm{bad}}$.
\end{prop}
\begin{proof} The proof uses the results from \S \ref{subsec:combinatorics for Dn}. By the formula (\ref{eq:defn of M}), it suffices to show for each $w'' \in W^1 = W_0$, that
	\begin{align}\label{desired est Dn odd bad} \sum _{w \in W^1 =W_0} (-1)^{\ell_1 (w)}
	\mathcal P ( w'' \lambda_{1,\mathrm{bad}} ^{(s)} + w \rho ^{\vee}- \rho^{\vee}, q_1^{-1}) = O(q_1^{-s (\frac{1}{2}
		\mathrm{def}_G(b) +a ) })
	\end{align} for some $a >0$. Here we have made the change of variable $ww'\mapsto w''$ in (\ref{eq:defn of M}), and have used the fact that $e_0 (w'\lambda_{1,\mathrm{bad}} ^{(s)}) \equiv 0$ for all $w'\in W^1$, as long as $s\gg 0$.
	
	Fix $w''$, and write  $\zeta_s : = w''\lambda _{1,\mathrm{bad}}^{(s)}$.
	Let $\abs{\cdot}$ be the norm on $\mathbb R^n$ defined in (\ref{eq:defn of norm}).  Since $W_0 \subset  \set{\pm 1} ^{n} \rtimes S_n$, there exist $1 \leq j \leq n , ~ \varepsilon \in \set{\pm 1}$, and $\underline \nu = (\nu_2,\cdots, \nu_n) \in \set{\pm 1} ^{n-1},$ such that $$  \zeta_s =  ( \frac{1}{2} s \nu_2, \frac{1}{2} s \nu_3 , \cdots, \frac{1}{2} s \nu_j , \frac{3}{2} s \varepsilon, \frac{1}{2} s \nu_{j+1}, \cdots , \frac{1}{2} s \nu_n), $$
	where $\frac{3}{2} s \varepsilon$ is at the $j$-th place.

	Assume either $j \neq 1$ or $\varepsilon = -1$. Then for $s \gg 0$ and all $w\in W_0$, we have $\zeta_s + w \rho ^{\vee} - \rho ^{\vee} \notin R^+$, and so $\mathcal P(\zeta_s + w \rho ^{\vee} - \rho ^{\vee}, q_1 ^{-1}) =0$. We are done in this case.
	
	Hence we assume $j=1$ and $\varepsilon = 1$. Assume without loss of generality that $s = 4t$ for $t \in \mathbb N$. As an easy application of the Weyl character formula, we know that any function $\mathscr P$ from $Y^*$ to any abelian group satisfies
	$$\sum_{w \in W_0} (-1)^{\ell_1(w)} \mathscr P(w \rho^{\vee} -  \rho^{\vee}) = \sum _{S \subset \Phi^{\vee, +} } (-1)^{\abs{S}} \mathscr P( - \sum _{\beta \in S} \beta).$$ In particular, the left hand side of (\ref{desired est Dn odd bad}) is equal to
	$$ \sum _{ S \subset \Phi ^{\vee, + }} (-1) ^{\abs{S}} \mathcal P (\zeta_s - \sum _{\beta \in S} \beta, q_1 ^{-1}).$$ By Proposition \ref{prop:interpretation} (1) and Proposition \ref{prop:cancellation for Dn odd}, the above is equal to
	\begin{align}\label{eq:after cancel}
	\sum _{L \in \Z,~  L > (n+3.5)t} q_1 ^{-L}  \sum _{S\subset \Phi ^{\vee, + }} (-1) ^{\abs{S}} \# \mathbb P(\zeta_s - \sum _{\beta \in S} \beta) _L.
	\end{align}
	By the same argument as in the proof of Proposition \ref{prop:est for unram}, the expression
	$$\abs{\sum _{L \in \Z_{\geq 0}} \sum _{S\subset \Phi ^{\vee, +}} (-1) ^{\abs{S}} \# \mathbb P(\zeta_s - \sum _{\beta \subset S} \beta) _L  } $$
	is of polynomial growth in $s$ (or in $t$). Hence (\ref{eq:after cancel}) is bounded by $O(q_1 ^{-(n +3.4)t})$. Since $s\cdot \mathrm{def}_G(b)/2 = (n+3)t $, the desired bound (\ref{desired est Dn odd bad}) follows.
\end{proof}
\subsection{The case Proper-II} We now treat type $E_6, \theta = \id , b = b_1$ or $b_2$. See \S \ref{subsec:E6 split}. By symmetry, it suffices to treat the case of $b_1$. Recall in this case $$\Lambda(b_1)- \Lambda (b_1) _{\mathrm{good}}=\{\varpi_5,\varpi_4+\varpi_1,\varpi_2+\varpi_6,2\varpi_6\} , $$ and we have $\mathrm{def}_G(b_1)=4$.
\begin{prop}\label{prop:bad for E6}
	The bound (\ref{eq:key est 2}) holds for all $\lambda\in \{\varpi_4+\varpi_1,\varpi_2+\varpi_6,2\varpi_6\}$. In other words, in Proposition \ref{key-bound} (\ref{item:2}) (for $b = b_1$) we may take $\lambda_{\mathrm{bad}}$ to be $\varpi_5$. \end{prop}
\begin{proof} We define a function $\abs{\cdot}': \mathbb R^9 \to \mathbb R_{\geq 0}$ in the following way. For any $v = \sum_{i=1}^9x_ie_i \in \mathbb R^9$, we define
	$$|v|' : =\max_{\substack{i,j\in\{0,1,2\},i\neq j;\\ k,l\in\{1,2,3\}}}|x_{3i+k}-x_{3j+l}|. $$
	In other words, we think of $\mathbb{R}^9$ as $(\mathbb{R}^3)^3$, and we take the largest difference between a coordinate in one factor of $\mathbb{R}^3$ and a coordinate in a different factor. Then $\abs{\cdot}'$ is a semi-norm, i.e., it is compatible with scalar multiplication by $\mathbb R$ and satisfies the triangle inequality. Note that $\abs{\cdot}'$ is not $W_0$-invariant.
	By the explicit description of the roots, we have $|\alpha|'=1$ for all positive roots $\alpha$.
	
	\textbf{Claim.} $|\mu|'\geq 7/3$ for all $\mu\in W_0\lambda$.
	
	We prove the claim. We first record explicitly the $W_0$-orbit of $\lambda$. To state it we need some notation. Let $C_3$ be the cyclic group of order $3$ with a fixed generator $c$. We let $C_3$ act on $S_3\times S_3\times S_3$ via $$c:(\sigma_1,\sigma_2,\sigma_3)\longmapsto (\sigma_2,\sigma_3,\sigma_1).$$ Let $H$ denote the semi-direct product $(S_3\times S_3\times S_3)\rtimes C_3$. Then we have an action of $H$ on $\mathbb{R}^9$, where $S_3\times S_3\times S_3$ acts naturally on the coordinate indices, and $c\in C_3$ acts via $$c:(x_1,x_2,x_3;x_4,x_5,x_6;x_7,x_8,x_9)\longmapsto (x_4,x_5,x_6;x_7,x_8,x_9;x_1,x_2,x_3).$$
	
	For $\lambda=\varpi_2+\varpi_6$, its $W_0$-orbit is given by the union of the $H$-orbits of the following 7 vectors:
	\begin{align*}&
	(2,-1,-1;\frac{2}{3},-\frac{1}{3},-\frac{1}{3};\frac{1}{3},\frac{1}{3},-\frac{2}{3}), && (\frac{2}{3},-\frac{1}{3},-\frac{1}{3};\frac{1}{3},\frac{1}{3},-\frac{2}{3};3,3,-6), \\ &  (\frac{4}{3},\frac{1}{3},-\frac{5}{3};1,0,-1;\frac{2}{3},-\frac{1}{3},-\frac{1}{3}),&& (\frac{1}{3},\frac{1}{3},\frac{2}{3};1,0,-1;\frac{5}{3},-\frac{1}{3},-\frac{4}{3}),  \\ &  (\frac{4}{3},\frac{1}{3},-\frac{5}{3};0,0,0;\frac{2}{3},\frac{2}{3},-\frac{4}{3}), && (\frac{4}{3},-\frac{2}{3},-\frac{2}{3};0,0,0;\frac{5}{3},-\frac{1}{3},-\frac{4}{3}), \\ & (\frac{4}{3},-\frac{2}{3},-\frac{2}{3};1,0,-1;\frac{2}{3},\frac{2}{3},-\frac{4}{3}).
	\end{align*}

	For $\lambda=\varpi_1+\varpi_4$, its $W_0$-orbit is the union of the $H$-orbits of the following 4 vectors:
	\begin{align*} &
	(\frac{4}{3},-\frac{1}{3},-\frac{5}{3};0,0,0;\frac{2}{3},-\frac{1}{3},-\frac{1}{3}), &&  (\frac{1}{3},\frac{1}{3},-\frac{2}{3};0,0,0;\frac{5}{3},-\frac{1}{3},-\frac{4}{3}), \\  &  (\frac{4}{3},-\frac{2}{3},-\frac{2}{3};1,0,-1;\frac{2}{3},-\frac{1}{3},-\frac{1}{3}),&&  (\frac{1}{3},\frac{1}{3},-\frac{2}{3};1,0,-1;\frac{2}{3},\frac{2}{3},-\frac{4}{3}).
	\end{align*}
	
	For $\lambda=2\varpi_6$, its $W_0$-orbit is the $H$-orbit of $$(\frac{4}{3},-\frac{2}{3},-\frac{2}{3};0,0,0;\frac{2}{3},\frac{2}{3},-\frac{4}{3}).$$
	
	One sees easily that $\abs{\cdot}'$ of all the above vectors are $\geq 7/3.$ Since $\abs{\cdot}'$ is invariant under the action of $H$, it follows that $|\mu|'\geq 7/3$ holds for all  $\mu\in W_0 \lambda$. The claim is proved.
	
	Based on the claim, we prove (\ref{eq:key est 2}) for $\lambda \in \{\varpi_4+\varpi_1,\varpi_2+\varpi_6,2\varpi_6\} $, using an argument
	similar to \S \ref{subsubsec:classic}.
	By the formula (\ref{eq:defn of M}), it suffices to show for each $w,w'\in W_0$ that
	$$\mathcal P ( \psi_s , q_1^{-1}) = O(q_1^{-s (\frac{1}{2}
		\mathrm{def}_G(b_1) +a ) })
	$$ for some $a >0$, where $\psi_s : = w\bullet (w'\lambda^{(s)})$. By the same argument as in \S \ref{subsubsec:classic}, we easily reduce to proving: For some constant $a> 0$,
	\begin{align}\label{eq:bd Ns E6}
	\min \set{\abs{\underline m}: \underline m \in \mathbb P(\psi_s) } \geq s (\frac{1}{2}
	\mathrm{def}_G(b_1) +a ) = s(2+a),\end{align}
	for all $s \gg 0$. By the previous claim, for all $\underline m \in \mathbb P(\psi_s)$ we have
	\begin{align*}
	1 \cdot \sum _{\beta \in  \Phi ^{\vee, + }} m (\beta)  & \geq \abs{\sum _{\beta \in \Phi ^{\vee, +}} m(\beta) \beta }' = \abs{\psi_s}' \geq  \abs{ww'\lambda ^{(s)}}' - \abs{w\rho^{\vee} -\rho ^{\vee}}'   \\ & =  s \abs{ww'\lambda}' -C \geq s \cdot \frac{7}{3} - C,
	\end{align*}
	where $C$ is a constant independent of $s$. Here the number $1$ appearing in the leftmost term is equal to $\min _{\beta \in \Phi ^{\vee,+}} \abs{\beta}'$. Since $\theta = \id$, the leftmost term in the above inequalities is none other than $
	\abs{\underline m}$. The desired (\ref{eq:bd Ns E6}) follows.
\end{proof}

\subsection{The case Proper-III}
We now treat type $E_7, \theta = \id ,$ and $[ b]\in B(G)$ being the unique basic class such that $\kappa_G(b)$ is the non-trivial element of $\pi_1(G) = \Z/2\Z$. See \S \ref{subsec:E7}. Recall in this case $\Lambda(b)- \Lambda (b) _{\mathrm{good}}=\{\varpi_5\},$ and $\mathrm{def}_G(b)=3$.
\begin{prop}
	The bound (\ref{eq:key est 1}) holds for $\lambda = \varpi_5$.
\end{prop}
\begin{proof}
	Firstly, the $W_0$-orbit of $\lambda$ is given by all permutations under $S_8$ of the following vectors:
	\begin{align*}
	\lambda_1 & =(\frac{7}{4},-\frac{1}{4},-\frac{1}{4},-\frac{1}{4},-\frac{1}{4},-\frac{1}{4},-\frac{1}{4},-\frac{1}{4}),  \\  \lambda_2 & =(\frac{5}{4},-\frac{3}{4},-\frac{3}{4},-\frac{3}{4},\frac{1}{4},\frac{1}{4},\frac{1}{4},\frac{1}{4}),\\  \lambda_3 & =(-\frac{7}{4},\frac{1}{4},\frac{1}{4},\frac{1}{4},\frac{1}{4},\frac{1}{4},\frac{1}{4},\frac{1}{4}), \\  \lambda_4 & =(-\frac{5}{4},\frac{3}{4},\frac{3}{4},\frac{3}{4},-\frac{1}{4},-\frac{1}{4},-\frac{1}{4},-\frac{1}{4}).
	\end{align*}
	Indeed it is easy to see that all these elements lie in $W_0\lambda$ (using the fact that the $W_0$ contains the copy of $S_8$), and one easily computes the size of $W_0 \lambda$ to prove that these are all the elements of $W_0 \lambda$.
	
	For $1\leq i \leq 8$ and $\tau \in S_8$, we define functions $$\abs{\cdot}_i: \mathbb R^8 \to \mathbb R_{\geq 0},\quad \abs{\cdot}_{\tau}: \mathbb R^8 \to \mathbb R_{\geq 0}$$ in the following way. For any $v = \sum_{i=1}^8 x_ie_i \in \mathbb R^8$, we define
	$$\abs{v}_i := \abs{x_i}, \quad \abs{v}_\tau: =  \abs{x_{\tau (1)}} + \abs{x_{\tau (2)}} +\abs{x_{\tau (3)}} + \abs{x_{\tau (4)}}. $$
	Then $\abs{\cdot}_i, \abs{\cdot}_{\tau}$ are all semi-norms on $\mathbb R^8$.
	
	Note that in the proof of Proposition \ref{prop:bad for E6}, we reduced to proving (\ref{eq:bd Ns E6}) for \textbf{each fixed pair} $(w,w')\in W_0\times W_0$. During the proof of (\ref{eq:bd Ns E6}) for the fixed $(w,w')$, we only needed to apply the semi-norm $\abs{\cdot}'$ to $ww'\lambda$, and not to any other element of $W_0 \lambda $. Hence for each element in $W_0\lambda$, we could use a semi-norm, which is specifically designed for that element, to finish the proof. In the current case, we reduce to proving that each $\mu \in W_0 \lambda$ satisfies at least one of the following inequalities:
	\begin{align} \label{eq:seminorm 1}
	\abs{\mu}_i > \frac{\mathrm{def}_G(b)}{2} \min _{\beta \in \Phi ^{\vee,+}} \abs{\beta}_i  = \frac{3}{2} \min _{\beta \in \Phi ^{\vee,+}} \abs{\beta}_i   ,
	\\ \label{eq:seminorm 2}
	\abs{\mu} _{\tau} > \frac{\mathrm{def}_G(b)}{2}   \min _{\beta \in \Phi ^{\vee,+}} \abs{\beta}_{\tau}   = \frac{3}{2}  \min _{\beta \in \Phi ^{\vee,+}} \abs{\beta}_{\tau}
	\end{align}
	for some $1 \leq i \leq 8$ or for some $\tau \in S_8$.
	
	When $\mu$ is an $S_8$-permutation of $\lambda_1$ or $\lambda_3$, assume the $i_0$-th coordinate of $\mu$ is $\pm 7/4$. Then $\mu$ satisfies the inequality (\ref{eq:seminorm 1}) indexed by $i_0$. In fact, any $\beta \in \Phi ^{\vee, +}$ satisfies $\abs{\beta}_{i_0} \leq 1$, and we have $\abs{\mu} _{i_0} = 7/4$.
	
	When $\mu$ is an $S_8$-permutation of $\lambda_2$ or $\lambda _4$, there exists $\tau \in S_8$ such that  $$\abs{\mu} _{\tau}= \frac{5}{4} + \frac{3}{4} + \frac{3}{4} + \frac{3}{4} =  \frac{7}{2}.$$ On the other hand $|\beta|_{\tau} \leq 2$ for all $\beta \in \Phi ^{\vee, +}$ and all $\tau \in S_8$. Therefore (\ref{eq:seminorm 2}) holds for some $\tau$.
\end{proof}

\section{Proof of the key estimate, Part III} \label{sec:part III} We have proved all the statements in Proposition \ref{key-bound}, except the existence of $\mu_1$ in Proposition \ref{key-bound} (\ref{item:1}), and the existence of $\mu_1, \mu_2$ in Proposition \ref{key-bound} (\ref{item:2}). In this section we construct them.

First assume that $G$ is not of type $E_6$, and that $\theta = \id$. We easily examine all such cases in \S \ref{sec:proof of key est} and see that $\lambda_b^+ \in X^*(\widehat T) ^+$ is always minuscule. Hence we may take $\mu_1 :  = \lambda _b ^+$. Since $G$ is adjoint and $\theta =\id$, the condition that $b \in B(G,\mu_1)$ is equivalent to the condition that $b$ and $\mu_1$ have the same image in $\pi_1 (G) _{\sigma} = \pi_1(G)$, which is true by construction. Moreover, we have $\dim V_{\mu_1} (\lambda _b ) _{\mathrm{rel}} = \dim V_{\lambda _b^+} (\lambda_b^+) = 1$. The proof of Proposition \ref{key-bound} is complete in these cases.

The only remaining cases are the following:
\begin{description}
	\item[Nonsplit-I] Type $D_n, n\geq 5, \theta$ has order $2$. See \S \ref{subsec:Dn nonsplit}.
	\item[Nonsplit-II] Type $D_4, \theta$ has order $2$. See \S \ref{subsec:Nonsplit-II}.
	\item[Split-$E_6$] Type $E_6, \theta = \id$. See \S \ref{subsec:E6 split}
\end{description}
In fact, in all the other cases listed in \S \ref{subsec:strategy} where $\theta$ is non-trivial, namely cases (6) and (8), we have shown in \S \ref{subsec:triality} and \S \ref{subsec:E6 nonsplit} that any basic $[b] \in B(G)$ is unramified, so we do not need to consider these cases.

\subsection{The case Nonsplit-I} \label{subsec:pf of Nonsplit-I}
As we showed in \S \ref{subsec:Dn nonsplit}, we have $\pi_1(G)_{\sigma} \cong \Z/2\Z $, and there is a unique choice of basic $[b] \in B(G)$ corresponding to the non-trivial element in $\pi_1(G) _{\sigma}$. Moreover we have
$$ \lambda_b ^+ = (\frac{1}{2}, \cdots, \frac{1}{2}) \in \Q^{n-1} = X^*(\hatS)\otimes \Q .$$ Recall that $X_*(T) = X^*(\widehat T) = L_2 \subset \mathbb R ^{n}.$ We take $$\mu_1: = (\frac{1}{2},\cdots, \frac{1}{2})  \in L_2 = X_*(T) = X^*(\widehat T). $$ Then $\mu_1$ is in $X^*(\widehat T) ^+$ and is minuscule. From the description of the action of $\sigma$ on $\pi_1(G)$ in \S \ref{subsec:Dn nonsplit}, the image of $\mu_1$ in $\pi_1 (G)_{\sigma}$ is the non-trivial element, and hence $[b] \in B(G,\mu_1)$. Finally, the only weights in $X^*(\widehat T)$ of $V_{\mu_1}$ are the elements of $W_0 \mu_1$. Among all these weights, there is precisely one that restricts to $\lambda _b ^+ \in X^*(\hatS)$, namely $\mu_1$. Hence we have
$$\dim V_{\mu_1} (\lambda _b) _{\mathrm{rel}} = \dim V_{\mu_1} (\lambda _b^+) _{\mathrm{rel}}  = \dim V_{\mu_1} (\mu_1) = 1. $$
\subsection{The case Nonsplit-II} We keep the notation of \S \ref{subsec:Nonsplit-II}. Note that what $\tau$ is does not affect the subset $\set{\alpha_1,\alpha_2, \alpha_3, \alpha_4}$ of $\mathbb R^4$. Nor does it affect the coroot lattice and the coweight lattice. Moreover, no matter what $\tau $ is, the quotient map $X_*(T) \to X^*(\hatS)$ is always the same as
\begin{align*}
L_2 & \To L_2' \\
(\xi_1,\xi_2, \xi_3, \xi_4) & \longmapsto (\xi_1, \xi_2, \xi_3),
\end{align*}
where $L_2 = \Z^4 + \Z (\frac{1}{2}, \frac{1}{2}, \frac{1}{2}, \frac{1}{2})$ and $L_2' = \Z ^4 +  \Z (\frac{1}{2}, \frac{1}{2}, \frac{1}{2})$. Hence the situation is precisely the same as the case \textbf{Nonsplit-I} discussed in \S \ref{subsec:pf of Nonsplit-I}. More precisely, for the unique basic $[b] \in B(G)$ that maps to the non-trivial element in $\pi_1(G) _{\sigma} \cong \Z/2\Z$, we have $\lambda_b = (-\frac{1}{2} , \frac{1}{2}, -\frac{1}{2}), \lambda_b ^+ = (\frac{1}{2}, \frac{1}{2}, \frac{1}{2})$, and we take $\mu_1 : = (\frac{1}{2} , \frac{1}{2}, \frac{1}{2}, \frac{1}{2})$.

\subsection{The case Split-$E_6$} We keep the notation of \S \ref{subsec:E6 split}. To finish the proof of Proposition \ref{key-bound} (\ref{item:2}), we need to construct $\mu_1$ and $\mu_2$. By symmetry we only need to consider $b_1$, among $b_1,b_2$. Recall from \S \ref{subsec:E6 split} that $\lambda_{b_1}^+ = \varpi_1$. Recall from Proposition \ref{prop:bad for E6} that the distinguished element $\lambda _{\mathrm{bad}}$ in $\Lambda (b_1)$ is $\varpi_5$. Since $\theta = \id$, we have $\hatS = \widehat T$.

Note that $\lambda_{b_1}^+ = \varpi_1$ is minuscule. We take $\mu_1: = \varpi_1$. Then the only weight of $V_{\mu_1}$ in $X^*(\hatS) ^+ = X^*(\widehat T) ^+$ is $\mu_1 = \lambda_{b_1} ^+$, and $\dim V_{\mu_1} (\lambda_{b_1}) _{\mathrm{rel}} = \dim V_{\mu_1} (\mu_1) = 1$.
We have $b_1 \in B(G,\mu_1)$, because the image of $\mu_1 = \lambda_{b_1} ^+$ in $\pi_1(G) _{\sigma} =\pi_1(G)$ is the same as that of $\lambda_{b_1}$, which is the same as $\kappa (b_1)$.

Note that $\varpi_6$ is also minuscule. We take $\mu_2 :
=2\varpi_1+\varpi_6$. Then $\mu_2$ is a sum of dominant minuscule coweights.
By (\ref{E6 root}) we know that $\mu_2 - \varpi_1$ is in the coroot lattice.
Hence $\mu_2$ represents the same element in $\pi_1(G) _{\sigma} = \pi_1(G)$
as $\varpi_1$, and in particular $b_1 \in B(G, \mu_2)$. We are left to check
that $V_{\mu_2} (\lambda _{\mathrm{bad}}) _{\mathrm{rel}}$, which is
$V_{\mu_2} (\varpi_5)$, is non-zero. One computes that $\dim V_{\mu_2}
(\varpi_5) = 14 $, (see for example \emph{LiE online service},
\href{http://young.sp2mi.univ-poitiers.fr/cgi-bin/form-prep/marc/dom_char.act?x1=2&x2=0&x3=0&x4=0&x5=0&x6=1&rank=6&group=E6}{\texttt{http://young.sp2mi.univ-poitiers.fr/}}\break%
\href{http://young.sp2mi.univ-poitiers.fr/cgi-bin/form-prep/marc/dom_char.act?x1=2&x2=0&x3=0&x4=0&x5=0&x6=1&rank=6&group=E6}{\texttt{cgi-bin/form-prep/marc/dom\_char.act?x1=2\&x2=0\&x3=0\&x4=0\&x5=0\&x6}}\break%
\href{http://young.sp2mi.univ-poitiers.fr/cgi-bin/form-prep/marc/dom_char.act?x1=2&x2=0&x3=0&x4=0&x5=0&x6=1&rank=6&group=E6}{\texttt{=1\&rank=6\&group=E6}}).\footnote{Note
	that in LiE, our $\alpha_2, \alpha_3, \alpha_4$  are indexed by $3,4,2$
	respectively.} The proof of Proposition \ref{key-bound} is complete.

\appendix
\section{Irreducible components for quasi-split groups}\label{app}
We explain in this appendix how we can use our results combined with \cite{He} to obtain a description of the number of $J_b(F)$-orbits of irreducible components of affine Deligne--Lusztig varieties associated to a group which is quasi-split but not necessarily unramified. The main result Theorem \ref{thm: main quasi-split} is a generalization of Conjecture \ref{conj:numerical Chen-Zhu}.

\subsection{Basic definitions}
We extend the notations introduced in \S \ref{sec:prelim}. We let $F$, $L$, $k_F$, $k$, $\sigma$, $\Gamma$ be as in \S \ref{sec:prelim}. However now we only assume that $G$ is a quasi-split reductive group over $F$. Let $T \subset G$ be the centralizer of a maximal $F$-split torus in $G$. Then $T$ is a maximal torus in $G$ since $G$ is quasi-split. We fix $B$ to be a Borel subgroup of $G$ (over $F$) containing $T$. Let $\breve A \subset T_L$ be the maximal $L$-split sub-torus of $T_L$. Note that $T_L$ is a minimal Levi subgroup of $G_L$, so $\breve A$ is a maximal $L$-split torus in $G_L$. Let $N \subset G_L$ denote the normalizer of $\breve A$. Let $V$ be the apartment of $G_L$ corresponding to $\breve A$.  Let $\mathfrak{a}$ be a $\sigma$-stable alcove in $V$, and let $\mathfrak{s}$ be a $\sigma$-stable special vertex lying in the closure of $\mathfrak{a}$. Denote by $\mathcal I$ the Iwahori group scheme over $\Ok_F$ associated to $\mathfrak a$, and denote by $\mathcal K$ the special parahoric group scheme over $\Ok_F$ associated to $\mathfrak s$. Let $\Gamma_0 \subset \Gamma$ denote the inertia subgroup, which is also identified with $\Gal (\overline L /L)$. The choice of $\mathfrak{s}$ gives an identification $V\cong X_*(T)_{\Gamma_0}\otimes_{\mathbb{Z}}\mathbb{R},$ sending $\mathfrak{s}$ to $0$. In the following we freely use the identification in Lemma \ref{lem:elementary about coinv} (\ref{item:3 in coinv}). We assume that under the identification
\begin{align}\label{eq:id for G}
V \cong X_*(T)_{\Gamma_0} \otimes _{\Z} \mathbb R \cong X_*(T)_{\mathbb R}^{\Gamma_0},
\end{align} the image of $\mathfrak a$ is contained in the anti-dominant chamber $-X_*(T)_{\mathbb R} ^{+}$.

The Iwahori--Weyl group is defined to be $W:=N(L)/(T(L) \cap \mathcal I(\Ok_L)).$ For any $w \in W$ we choose a representative $\dot w \in N(L)$. We write $W_0:=N(L)/T(L)$ for the relative Weyl group of $G$ over $L$. Then we have a natural exact sequence:
$$1 \To X_*(T)_{\Gamma_0} \To W\To W_0 \To 1.$$ Similar to \S \ref{subsec:IW}, the canonical action of $N(L)$ on $V$ factors through an action of $W$, and we split the above exact sequence by identifying $W_0$ with the subgroup of $W$ fixing $\mathfrak s \in V$.
For $\underline{\mu}\in X_*(T)_{\Gamma_0}$ we write $t^{\underline{\mu}}$ for the corresponding element in $W$. The Frobenius $\sigma$ induces an action on $W$ which stabilizes the set  $\mathbb{S}$ of simple reflections. See \cite{HainesRapoport} for more details.

Let $B(G)$ (resp.~$B(W,\sigma)$) denote the set of $\sigma$-conjugacy classes of $G(L)$ (resp.~$W$). Let $X_*(T)^+_{\Gamma_0, \mathbb{Q}}$ denote the intersection of $X_*(T)_{\Gamma_0} \otimes \mathbb{Q} \cong X_*(T)_{\mathbb Q} ^{\Gamma_0}$ with $ X_*(T)_{\mathbb Q}^+$. Similar to \S \ref{subsec:B(G)}, we have a commutative diagram
\begin{align}\label{comm diag}
\xymatrix{B(W, \sigma) \ar@{->>}[rr]^{\Psi} \ar[dr]^{(\bar \nu, \kappa)} & & B(G) \ar@{^{(}->}[ld] _{(\bar \nu, \kappa)} \\ & (X_*(T)^+_{ \Gamma_0,\Q})^\sigma \times \pi_1(G)_\Gamma &},
\end{align}
where $\Psi$ is a surjection, and the map $(\bar \nu , \kappa)$ on $B(W,\sigma)$ can be described explicitly. These statements are proved in \cite{He}, see \cite[\S 1.2]{HZ} for an exposition.

Let $w\in W$ and $b\in G(L)$. We define the affine Deligne--Lusztig variety $X_w(b)$ as follows: $$X_{w}(b):=\{g \mathcal{I} (\Ok _L) \in G(L)/ \mathcal{I}(\Ok_L) \mid g^{-1} b \sigma(g) \in \mathcal{I}(\Ok_L) \dot w \mathcal{I}(\Ok_L)\}.$$
Now let $\underline{\mu}\in X_*(T)_{\Gamma_0}$ be the image of an element $\mu \in X_*(T)^+$. Similarly we define the affine Deligne--Lusztig variety $X_{\underline{\mu},K}(b)$ as follows:
$$ X_{\underline{\mu},K}(b):=\{g \mathcal{K} (\Ok _L) \in G(L)/ \mathcal{K}(\Ok_L) \mid g^{-1} b \sigma(g) \in \mathcal{K}(\Ok_L) \dot {t}^{\underline{\mu}} \mathcal{K}(\Ok_L)\}.$$
When $F$ is of equal (resp.~mixed) characteristic, $X_w(b)$ and $X_{\underline \mu , K} (b)$ are schemes (resp.~perfect schemes) locally of finite type (resp.~locally perfectly of finite type) over $k$, see \cite{HVfinite}.

We define the set $$B(G, \underline \mu): =\{ [b]\in B(G)\mid \kappa([b])=\mu^\natural, \bar \nu_b\leq \mu^\diamond \}.$$
Here $\mu^\natural$ denotes the image of $\mu$ in $\pi_1(G)_\Gamma$, and $\mu^\diamond \in (X_*(T) _{\Gamma_0, \Q} ^+)^{\sigma}$ denotes the average over the $\sigma$-orbit of $\underline \mu \in X_*(T) _{\Gamma_0}$.
Note that both $\mu ^{\natural}$ and $\mu ^{\diamond}$ depend only on $\underline \mu$, which justifies the notation $B(G,\underline \mu)$. The set $B(G,\underline \mu)$ controls the non-emptiness pattern of $X_{\underline \mu, K} (b)$.

\begin{thm}
	We have $X_{\underline \mu, K} (b) \neq \emptyset$ if and only if $[b] \in B(G,\underline \mu)$.
\end{thm}
\begin{proof}
	This is proved in \cite[Theorem 7.1]{He} assuming $\mathrm{char} (F)  >0$. The same proof extends to the case $\mathrm{char} (F) =0$.
\end{proof}
For $b \in G(L)$, the group $J_b(F)$ acts on $X_{\underline{\mu},K}(b)$ via scheme automorphisms. Our goal is to understand the cardinality
\begin{align}\label{eq:card app}
\mathscr N(\underline \mu,b): = \# \bigg(J_b(F) \backslash \Sigma^{\topp}(X_{\underline{\mu},K}(b)) \bigg ) .
\end{align}
\textbf{For simplicity, from now on we assume that $G$ is adjoint.} The general case reduces to this case by a standard argument.
\subsection{A dual group construction} The desired formula for (\ref{eq:card app}) will be expressed in terms of a canonical reductive subgroup of the dual group $\widehat G$. We keep the assumption that $G$ is adjoint.

As in \S \ref{subsec:general facts}, we let $$\mathrm {BRD}(B,T) = (X^*(T), \Phi\supset \Delta, X_*(T), \Phi ^{\vee} \supset  \Delta ^{\vee})$$ be the based root datum associated to $(B,T)$, equipped with an action by $\Gamma$. Let $\widehat G$ be the dual group of $G$ over $\mathbb C$, which is equipped with a Borel pair $(\widehat{B}, \widehat{ T})$ and an isomorphism
$ \mathrm{BRD} (\widehat{B}, \widehat{ T}) \isom \mathrm{BRD} (B,T)^{\vee}. $ We fix a pinning $(\widehat{B}, \widehat{ T}, \widehat {\mathbb X}_+)$. The action of $\Gamma$ on $\mathrm{BRD}(B,T)$ translates to an action on $\mathrm{BRD}(\widehat B,\widehat T)$, and the latter lifts to a unique action of $\Gamma$ on $\widehat G$ via algebraic automorphisms that preserve $(\widehat{B}, \widehat{ T}, \widehat {\mathbb X}_+)$.

We define
\begin{align}\label{eq:hatH}
\widehat H: = \widehat G ^{\Gamma_0,0},
\end{align} namely the identity component of the $\Gamma_0$-fixed points of $\widehat G$. This construction was also considered by Zhu \cite{Zhuram} and Haines \cite{Hainesdualities}. By \cite[Proposition 5.1]{Hainesdualities}, $\widehat H$ is a reductive subgroup of $\widehat G$, and it has a pinning of the form $(\widehat B ^{\Gamma_0 ,0}, \widehat T ^{\Gamma_0, 0}, \widehat {\mathbb X}_{+}')$. Moreover, the induced action of the Frobenius $\sigma\in \Gamma/\Gamma_0$ on $\widehat H$ preserves this pinning.
We write $\widehat B_H : = \widehat B^{\Gamma_0 , 0}$ and $ \widehat T_H : = \widehat T ^{\Gamma_0 , 0}$. Let $\hat \theta$ denote the automorphism of $\widehat H$ given by $\sigma$. We define $$\hatS : = (\widehat T_H)^{\hat \theta, 0}. $$

Note that since $G$ is adjoint, the fundamental coweights of $G$ form a $\Gamma$-stable $\mathbb Z$-basis of $X_*(T)$. It then follows from Lemma \ref{lem:elementary about coinv} (\ref{item:2 in coinv}) that $X_*(T)_{\Gamma_0}$ and $X_*(T)_{\Gamma}$ are both free. Hence we in fact have $\widehat T_H = \widehat T^{\Gamma_0}$ and $\hatS = \widehat T ^{\Gamma}$. This observation will simplify our exposition.
\begin{lem}\label{lem:defn of lambda_b app} Let $b \in G(L)$. There is a unique element $ \lambda_b \in X^*(\hatS)$ satisfying the following conditions:
	\begin{enumerate}
		\item The image of $\lambda_b$ under $X^*(\hatS) = X_*(T)_{\Gamma} \to \pi_1(G) _{\Gamma}$ is equal to $\kappa(b)$.
		\item In $X^*(\hatS)_{\Q} = X^*(\widehat T)_{\Gamma} \otimes \Q = (X_*(T)_{\Gamma_0})_{\sigma} \otimes \Q \cong  (X_*(T)_{\Gamma_0,\Q})^{\sigma}$, the element $\lambda_b - \bar \nu_b$ is equal to a linear combination of the restrictions to $\hatS$ of the simple roots in $\Phi^{\vee} \subset X^*(\widehat T)$, with coefficients in $\Q \cap (-1,0]$.
	\end{enumerate}
\end{lem}
\begin{proof}
	The proof is the same as Lemma \ref{lem:defn of lambda_b}.
\end{proof}
\subsection{The main result}
Assume that $G$ is adjoint and quasi-split over $F$. Let $\underline{\mu}\in X_*(T)_{\Gamma_0}$ be the image of an element $\mu \in X_*(T)^+$. Let $b \in G(L)$. Define $\widehat H$ as in (\ref{eq:hatH}). Let $V^{\widehat H} _{\underline \mu}$ denote the highest weight representation of $\widehat H$ of highest weight $\underline \mu \in X^*(\widehat{T}_H) ^+$. Let $\lambda_b \in X^*(\hatS)$ be as in Lemma \ref{lem:defn of lambda_b app}, and let $V^{\widehat H} _{\underline \mu} (\lambda_b) _{\mathrm{rel}}$ be the $\lambda_b$-weight space in $V^{\widehat H} _{\underline \mu}$, for the action of $\hatS$.
\begin{thm}\label{thm: main quasi-split}
	Assume that $[b] \in B(G,\underline \mu)$. Then $$\mathscr N(\underline \mu,b) = \dim V^{\widehat H} _{\underline \mu} (\lambda_b) _{\mathrm{rel}}. $$
\end{thm}
\begin{rem}
	The appearance of the representation $V^{\widehat H} _{\underline \mu}$ of the subgroup ${\widehat H}$ of $\widehat G$ in Theorem \ref{thm: main quasi-split} is compatible with the ramified geometric Satake in \cite{Zhuram}.
\end{rem}
\begin{proof}[Proof of Theorem \ref{thm: main quasi-split}] The idea of the proof is to reduce to the unramified case. For this we first construct an auxiliary unramified reductive group over $F$.
	
	From $\widehat H$ and its pinned automorphism $\hat \theta$, we obtain an unramified reductive group $H$ over $F$, whose dual group is $\widehat H$. By definition $H$ is equipped with a Borel pair $(B_H, T_H)$, and a $\sigma$-equivariant isomorphism of based root data $\mathrm{BDR} (B_H, T_H) \isom \mathrm{BDR} (\widehat B_H, \widehat T_H) ^{\vee} .$ We write $$\mathrm{BDR} (B_H, T_H)= (X^*(T_H), \Phi_H , X_*(T_H), \Phi _H^{\vee} ). $$
	Then we have canonical $\sigma$-equivariant identifications
	$$X^*(T_H) \cong X_*(\widehat T _H) \cong X_*(\widehat T) ^{\Gamma_0} \cong X^*(T) ^{\Gamma_0}$$ and
	$$X_* (T_H) \cong X^*(\widehat T _H) \cong X^*(\widehat T)_{\Gamma_0} \cong X_*(T) _{\Gamma_0},$$ which we shall treat as identities. Here $X_*(T) _{\Gamma_0}$ is free, as we have noted before.
	
	Note that $T_{H,L}$ is a maximal split torus in $H_L$. Let $V_H$ be the corresponding apartment, and fix a hyperspecial vertex $\mathfrak s_H$ in $V_H$ (coming from the apartment of $H$ corresponding to the maximal $F$-split sub-torus of $T_H$). We fix a $\sigma$-stable alcove $\mathfrak a_H \subset V_H$ whose closure contains $\mathfrak s_H$. We identify
	\begin{align}\label{eq:id for H}
	V_H \cong X_*(T_H) \otimes \mathbb R,
	\end{align}
	sending $\mathfrak s_H$ to $0$, such that the image of $\mathfrak a_H$ is in the anti-dominant chamber.
	
	Since $X_*(T_H)  = X_*(T)_{\Gamma_0}$, the two identifications (\ref{eq:id for G}) and (\ref{eq:id for H})
	give rise to a $\sigma$-equivariant identification $
	V \cong V_H $ which maps $\mathfrak a$ onto $\mathfrak a_H$, and maps $\mathfrak s $ to $\mathfrak s_H$.

	By \cite[Corollary 5.3]{Hainesdualities}, the set of coroots $\Phi _H^{\vee} \subset X_*(T_H) = X_*(T)_{\Gamma_0}$ is given by $\breve \Sigma ^{\vee}$, where $\breve \Sigma$ is the \'{e}chelonnage root system of Bruhat--Tits, see \cite[\S 4.3]{Hainesdualities}. In particular, the coroot lattice in $X_*(T_H)$ is isomorphic to the $\Gamma_0$-coinvariants of the coroot lattice in $X_*(T)$. Moreover from $\Phi _H^{\vee} = \breve \Sigma ^{\vee}$ we know that the affine Weyl group of $G$ and the affine Weyl group of $H$ are equal, under the identification $V \cong V_H $. See \cite{HainesRapoport} for more details. Since the translation groups $X_*(T_H)$ and $X_*(T)_{\Gamma_0}$ are also identified, we have an identification between the Iwahori--Weyl group $W$ of $G$ and the Iwahori--Weyl group $W_H$ of $H$. This identification is $\sigma$-equivariant.

	Note that the bottom group in the diagram (\ref{comm diag}) and its analogue for $H$ are identified. Using the identification of $W$ and $W_H$, and using the surjectivity of the map $\Psi : B(W,\sigma) \to B(G)$ and its analogue $\Psi_H: B(W_H, \sigma) \to B(H)$, we construct $[b_H] \in B(H)$ whose invariants are the same as those of $[b]$. Since the set $ B(G, \underline \mu)$ is defined in terms of the invariants $(\bar \nu ,\kappa)$ and ditto for $ B(H,\underline \mu)$, we see that $[b] \in B(G, \underline \mu)$ if and only if $[b_H] \in B(H,\underline \mu)$. Here in writing $B(H,\underline \mu)$ we view $\underline \mu$ as an element of $X_*(T_H) ^+$.

	To relate the geometry of $X_{\underline{ \mu}, K} (b)$ with the geometry of $X_{\underline \mu} (b_H)$, we use the class polynomials in \cite{He}. For each $w\in W$ and each $\sigma$-conjugacy class $\Ok$ in $W$, we let $$f_{w,\Ok}\in\mathbb{Z}[v-v^{-1}]$$ denote the class polynomial defined in \cite[\S2.3]{He}.
	
	Using the fibration \begin{align*}\bigcup_{w\in W_0t^{\underline{\mu}} W_0}X_w(b)\To X_{\underline{\mu}, K}(b),
	\end{align*} we have an identification
	\begin{align}\label{eq:using fibr} J_b (F) \backslash
	\Sigma^{\topp}\left(\bigcup_{w\in W_0t^{\underline{\mu}} W_0}X_w(b)\right)\cong J_b (F) \backslash \Sigma^{\topp}(X_{\underline{\mu}, K}(b) ).
	\end{align}	 By \cite[Theorem 6.1]{He}, we have the formula $$\dim X_w(b)=\max_{\Ok}\frac{1}{2}(\ell(w)+\ell(\Ok)+\deg f_{w,\Ok}))-\langle\overline{ \nu}_b,2\rho\rangle.$$ where $\Ok$ runs through $\sigma$-conjugacy classes in $W$ such that $(\overline{ \nu},\kappa)(\mathcal{O})=(\overline{ \nu},\kappa)(b)$ and where $\ell(\Ok)$ denotes the length of a minimal length element in $\mathcal{O}$. Moreover the proof of \cite[Theorem 6.1]{He} also shows that the cardinality of $J_b(F)\backslash \Sigma^{\topp}(X_w(b))$ is equal to the leading coefficient of $	\sum_{\Ok}v^{\ell(w)+\ell(\Ok)}f_{w,\Ok}$. Since each $X_w(b)$ is locally closed in the union $\bigcup _{w\in W_0t^{\underline{\mu}}W_0}X_w(b)$, any top dimensional irreducible component in the union is the closure of a top dimensional irreducible component in $X_w(b)$ for a unique $w$. It follows that the cardinality of $$J_b(F)\backslash \Sigma^{\topp}\left(\bigcup _{w\in W_0t^{\underline{\mu}}W_0}X_w(b)\right)$$ is equal to the leading coefficient of \begin{equation}\label{eqn: class polynomial}
	\sum_{w \in W_0t^{\underline{\mu}}W_0}	\sum_{\Ok}v^{l(w)+l(\Ok)}f_{w,\Ok}	.\end{equation} By (\ref{eq:using fibr}), this number is just $\mathscr N(\underline{ \mu}, b)$. Since the term (\ref{eqn: class polynomial}) only depends on the quadruple $(W,\sigma,\underline{\mu},(\overline{ \nu},\kappa)(b))$, the same is true for $\mathscr N (\underline \mu , b)$.
	
	Applying the same argument to $H$, we see that $\mathscr N(\underline \mu, b_H)$ only depends on the quadruple $ (W_H,\sigma,\underline{\mu},(\overline{ \nu},\kappa)(b_H))$. Now since the quadruples $(W,\sigma,\underline{\mu},(\overline{ \nu},\kappa)(b))$ and $ (W_H,\sigma,\underline{\mu},(\overline{ \nu},\kappa)(b_H))$ are identified, we have $\mathscr N(\underline \mu, b) = \mathscr N(\underline \mu, b_H)$. It thus remains to check
	\begin{align}\label{eq:to check app}
	\mathscr N(\underline \mu, b_H) = \dim V^{\widehat H} _{\underline \mu} (\lambda_b) _{\mathrm{rel}}.
	\end{align}  By assumption $[b] \in B(G,\underline \mu)$, and so $[b_H] \in B(H,\underline \mu)$. The right hand side of (\ref{eq:to check app}) is easily seen to be equal to $\mathscr M(\underline \mu, b_H)$ in Conjecture \ref{conj:numerical Chen-Zhu}, for the group $H$. Hence the desired (\ref{eq:to check app}) follows from the main result Corollary \ref{cor:CZ} of the paper.
\end{proof}

\section*{Acknowledgements}

We would like to thank Michael Harris, Xuhua He, Chao Li, Shizhang Li, Thomas
Haines, Michael Rapoport, Liang Xiao, Zhiwei Yun, and Xinwen Zhu for useful
discussions concerning this work, and for their interest and encouragement.
We thank the anonymous referee for several corrections and useful
suggestions. R.~Z.~is partially supported by NSF grant DMS-1638352 through
membership at the Institute for Advanced Study. Y.~Z.~is supported by NSF
grant DMS-1802292.

\bibliographystyle{hep}
\bibliography{myref}

\end{document}